\documentclass[reqno]{amsart}
\usepackage{amssymb,amscd,enumerate,mathrsfs,curves}

\numberwithin{equation}{section}
\newtheorem{theorem}[equation]{Theorem}
\newtheorem{proposition}[equation]{Proposition}
\newtheorem{lemma}[equation]{Lemma}
\newtheorem{corollary}[equation]{Corollary}

\theoremstyle{definition}
\newtheorem{definition}[equation]{Definition}
\newtheorem{remark}[equation]{Remark}
\newtheorem{example}[equation]{Example}

\DeclareMathOperator{\Diff}{Diff}

\DeclareMathOperator{\End}{End}

\DeclareMathOperator{\Hom}{Hom}
\DeclareMathOperator{\Ind}{Ind}
\DeclareMathOperator{\op}{op}
\DeclareMathOperator{\sing}{sing}
\DeclareMathOperator{\spec}{spec}
\DeclareMathOperator{\supp}{supp}
\DeclareMathOperator{\sym}{ \sigma\!\!\!\sigma}

\DeclareMathOperator{\Vect}{Vect}

\def\agen{\mathfrak a}
\def\ext{\mathfrak e}
\def\extsym{e}
\def\Ext{\mathfrak E}
\def\gen{\mathfrak g}

\def\KK{\mathfrak{K}}
\def\m{\mathfrak m}

\def\p{\mathfrak p}

\def\preP{\mathfrak P}

\def\C{\mathbb C}
\def\NN{\mathbb N}
\def\R{\mathbb R}

\def\Dom{\mathcal D}
\def\Ha{\mathcal H}
\def\K{\mathcal K}
\def\Lie{\mathcal L}
\def\M{\mathcal M}
\def\N{\mathcal N}

\def\Ring{\mathcal R}
\def\Y{\mathcal Y}
\def\Z{\mathcal Z}

\def\target{\mathscr G}
\def\kerb{\mathscr K}
\def\L{\mathscr L}

\def\Sch{\mathscr S}
\def\trb{\mathscr T}

\def\Ss{\mathfrak s}
\def\eps{\varepsilon}
\def\minus{\backslash}
\def\im{i}
\def\Id{I}
\def\embed{\hookrightarrow}

\def\open#1{\smash[t]{\overset{{}_{\,\,\circ}}{#1}{}}}
\def\set#1{\left\{#1\right\}}
\def\Set#1{\Big\{#1\}}
\def\display#1#2{\mbox{\parbox{#1} {#2}}}
\def\ghost#1{}

\def\rpar{)}
\def\lbra{[}

\def\wT{\,{}^w\hspace{-0.5pt}T}
\def\wpi{\,{}^w\!\pi}
\def\wsym{\,{}^w\!\!\sym}
\def\wdiff{\,{}^w\!d}

\def\cl{\mathrm{cl}}
\def\cone{\mathrm{cone}}
\def\comp{\mathrm{comp}}
\def\diag{\mathrm {diag}}
\def\loc{\mathrm{loc}}
\def\sc{\mathrm{sc}}

\def\scH{H_{\sc}}

\def\Wedge{\raise2ex\hbox{$\mathchar"0356$}}

\def\ft{\!\widehat{\phantom{m}}}

\def\bT{\,{}^b\hspace{-0.5pt}T}
\def\bP{\,{}^b\!P}
\def\bPhat{\,{}^b\!\widehat P}
\def\bPhatstar{\,{}^b\!\widehat {P^\star}}
\def\bPhatstary{\,{}^b\!\widehat {P^\star_y}}
\def\bA{\,{}^b\!A}

\def\leibniz{\mathbin{\#}}

\def\contract{\mathbin{\rfloor}}

\begin{document}

\title[Boundary value problems]{Boundary value problems for first order elliptic wedge operators}
\thanks{Work partially supported by the National Science Foundation, Grants DMS-0901202 and DMS-0901173}

\author{Thomas Krainer}
\address{Penn State Altoona\\ 3000 Ivyside Park \\ Altoona, PA 16601-3760}
\email{krainer@psu.edu}
\author{Gerardo A. Mendoza}
\address{Department of Mathematics\\ Temple University\\ Philadelphia, PA 19122}
\email{gmendoza@temple.edu}

\begin{abstract}
We develop an elliptic theory based in $L^2$ of boundary value problems for general wedge differential operators of first order under only mild assumptions on the boundary spectrum. In particular, we do not require the indicial roots to be constant along the base of the boundary fibration.

Our theory includes as a special case the classical theory of elliptic boundary value problems for first order operators with and without the Shapiro-Lopatinskii condition, and can be thought of as a natural extension of that theory to the geometrically and analytically relevant class of wedge operators. Wedge operators arise in the global analysis on manifolds with incomplete edge singularities. Our theory settles, in the first order case, the long-standing open problem to develop a robust elliptic theory of boundary value problems for such operators.

\end{abstract}


\maketitle


\section{Introduction}

We present here a theory of boundary value problems for general first order elliptic wedge differential operators closely paralleling, and including as a special case, that of regular elliptic first order boundary value problems. Wedge differential operators, similar in structure to the operators that result when rewriting a regular differential operator in cylindrical coordinates, constitute an important tool for doing the various flavors of analysis---topological, geometric, global---on manifolds with singularities, see e.g. \cite{AlbinLeichtMazzPiazza12,AlbinLeichtMazzPiazza13}. We describe the operators in the next few paragraphs and our results immediately afterwards in this introduction.

Wedge differential operators (and the closely related edge operators on which they are based) are operators defined on manifolds with boundary, where the latter is the total space of a fibration, a situation reminiscent of the spheres in the boundary of the blow-up of a manifold along a submanifold via cylindrical coordinates (see Mazzeo \cite{Mazz91}). Thus we will be dealing with a compact connected manifold $\M$ with non-empty boundary $\N$, the latter the total space of a locally trivial fibration 
\begin{equation}\label{BdyFiberBundle}
\display{300pt}{\begin{center}
\begin{picture}(116,53)
\put(30,45){$\Z\embed\N=\partial\M$	
\put(-38,-3){\curve(0,0,0,-27)\put(0,-27){\vector(0,-10){0}}}
\put(-33,-18){\small $\wp$}
\put(-42,-40){$\Y$}
}
\end{picture}
\end{center}}
\end{equation}
with compact fibers $\Z$. Associated with the boundary fibration there is a subspace of vector fields on $\M$, $\mathcal V_e$, a Lie subalgebra of $C^\infty(\M;T\M)$, consisting of those vector fields which over the boundary are tangent to the fibration. The ring of edge differential operators, $\Diff_e(\M)$ is then defined to be the algebra generated  by $\mathcal V_e$ and $C^\infty(\M)$. And if $E$ and $F$ are vector bundles, then $\Diff_e(\M;E,F)$ denotes the space of differential operators which locally are represented by matrices of elements of $\Diff_e(\M)$, cf. Mazzeo, op. cit. Alternatively, adapting a well known scheme, we may begin with the ring 
\begin{equation}\label{Ring}
\Ring=\set{f\in C^\infty(\M):f|_\N\text{ is constant on the fibers of }\wp},
\end{equation}
$f\in \Ring \iff f\in C^\infty(\M)$ and $f\big|_{\N}\in \wp^*C^\infty(\Y)$. Then the class of edge differential operators can be characterized as follows:
\begin{equation}\label{InductiveDiff_eDef}
\display{280pt}{
$P\in \Diff^m_e(\M;E,F)\iff P\in\Diff^m(\M;E,F)$ and, inductively, $
[P,f]\in x\Diff^{m-1}_e(\M;E,F)\text{ for any }f\in \Ring\text{ if } m\geq 1$.
}
\end{equation}
A wedge differential operator of order $m$ is an element of 
\begin{equation*}
x^{-m}\Diff^m_e(\M;E,F),
\end{equation*}
that is, an operator of the form $A=x^{-m}P$ with $P\in\Diff^m_e(\M;E,F)$, see Schulze \cite{SchuNH}. Here and elsewhere $x$ is a defining function for $\partial\M$, positive in $\open\M$. The presence of the factor $x^{-m}$ makes these operators differ in a fundamental way from edge operators except when $m=0$. Because of \eqref{InductiveDiff_eDef} we have the following natural property of wedge differential operators:
\begin{equation}\label{BracketWR}
A\in x^{-m}\Diff^m_e(\M;E,F),\ f\in \Ring\implies[A,f]\in x^{-(m-1)}\Diff^{m-1}_e(\M;E,F).
\end{equation}
Evidently $\Ring$ plays an important role. In local coordinates $x,y,z$ with $x$ as just described and the $y_j|_\N$ being lifts of local coordinates on $\Y$,
\begin{equation}\label{TypicalWOp}
A=\frac{1}{x^m}\sum_{k+|\alpha|+|\beta|\leq m} a_{k\alpha\beta}(x,y,z)(xD_x)^k (xD_y)^\alpha D_z^\beta.
\end{equation}
The coefficients $a_{k\alpha\beta}$ are smooth up to the boundary. The kind of coordinates just used will be referred to as adapted coordinates.

Two extreme cases of the boundary fibration should be mentioned, according to the extremal dimensions of $\Z$. The case $\dim \Z=\dim\N$ corresponds to totally characteristic (or $b$-operators) rather than $e$-operators, and in the case that $\Z$ is just a point and $\Y = \N$ the class $x^{-m}\Diff^m_e(\M;E,F)$ includes all regular differential operators $\Diff^m(\M;E,F)$. 

Ellipticity of a wedge operator $A\in x^{-m}\Diff^m_e(\M;E,F)$, $w$-ellipticity for short, is discussed in full detail in \cite{GiKrMe10} and briefly reviewed in the next section. It implies ellipticity of $A$ over the interior of $\M$. Let $\m_b=x^{-1}\m$ be a $b$-measure (so $\m$ is a smooth positive density on $\M$), let the bundle $E$ be given a Hermitian metric, also fix a number $\gamma\in \R$ and form the $L^2$ spaces of sections based on $\m_b$ with respect to the weight $x^{2\gamma}$, denoted $x^{-\gamma}L^2_b(\M;E)$, analogously for $F$. The notation is from Melrose \cite{Mel93} but see Section~\ref{sec-Domains} for a more detailed explanation. Already assuming $w$-ellipticity of $A$, recall the two standard closed extensions of $A$ (the concept goes essentially back to Weyl's thesis), the maximal extension, whose domain is
\begin{equation*}
\Dom_{\max}(A)=\set{u\in x^{-\gamma}L^2_b(\M;E):Au\in x^{-\gamma}L^2_b(\M;F)}
\end{equation*}
and the minimal domain whose domain $\Dom_{\min}(A)$ is the closure of $C_c^\infty(\open\M;E)$ in $\Dom_{\max}(A)$ with respect to the norm defined by the inner product
\begin{equation}\label{GraphNorm}
(u,v)_A=(Au,Av)_{x^{-\gamma}L^2_b}+(u,v)_{x^{-\gamma}L^2_b}.
\end{equation}

In order to set up a boundary value problem one of course needs to first be able to associate to each $u\in\Dom_{\max}(A)$ an object that can properly be regarded as a boundary value. In connection with this we showed in \cite{KrMe12a} how to construct, under a mild condition on the boundary spectrum stated here as \eqref{NobSpecOnCritLinesOrderm}, a vector bundle $\trb\to\Y$ for any elliptic wedge operator $A$. This is the trace bundle of $A$. Its definition and relevant properties are reviewed here in Section~\ref{sec-TraceBundles}. Continuing in this direction we construct here (for a first order operator) a map $\gamma_A:\Dom_{\max}(A)\to C^{-\infty}(\Y;\trb)$, the trace map associated with $A$, see Definition~\ref{gammaA}. A moment's thought gives that the kernel of $\gamma_A$ should contain $\Dom_{\min}(A)$; Proposition~\ref{DminSubsetKerTrace} states that this is indeed the case.

Once the notion of boundary value is secured, the meaning of a boundary value problem in its most classical sense is clear: given a vector bundle $\target\to\Y$ and a pseudodifferential operator
\begin{equation*}
B : C^{\infty}(\Y;\trb) \to C^{\infty}(\Y;\target),
\end{equation*}
one seeks $u$ such that
\begin{equation}\label{BVP}
\left\{\begin{aligned}
&Au=f \text{ in }\open \M\\
&B(\gamma_A u)=g\text{ on }\Y,
\end{aligned}\right.
\end{equation}
typically with $f\in x^{-\gamma}L^2_b(\M;F)$ and $g$ a section of $\target$. 

One aims at obtaining conditions relating $B$ to $A$ that imply that the problem is Fredholm. This requires setting up a notion of ellipticity for the full boundary value problem \eqref{BVP}. In preparation for this we discuss, in Section~\ref{sec-KernelBundle}, the kernel bundle of the normal family $A_\wedge(\pmb\eta)$ of $A$, a family consisting of cone-differential operators depending on $\pmb\eta\in T^*\Y\minus 0$ as a parameter. This family was discussed in an invariant setting in \cite{GiKrMe10} and is reviewed here in Section~\ref{sec-SummaryWedgeOps} (see \eqref{NormalFamilyAsLimit}), see also \cite[Section 7.1.3]{EgorovSchulze}, \cite[Section 3.3]{SchuNH}, and in quantized form \cite{Mazz91}. The underlying manifold for the operators of the normal family is $\Z^\wedge=\lbra0,\infty\rpar\times \Z$, essentially the part of the inward pointing normal bundle of $\N$ in $\M$ that lies over the typical fiber $\Z\subset \N$.

To place the normal family and its role in context, we note that, classically, for a regular elliptic operator in $\Diff^m(\M;E,F)$ ($\Z$ consists of just one point in this case and one views $\Y$ as being $\N$ itself), the normal family is the family of ordinary differential operators obtained from the restriction to $T^*_\Y\M$ of the principal symbol of the operator when the conormal variable is replaced by ordinary differentiation. More precisely, writing $\iota:\Y\to \M$ for the inclusion map (recall that $\Y=\N$ when discussing the regular elliptic situation), it is the family of ordinary differential operators on the fibers of the bundle dual to $\iota^*:T^*_\Y\M\to T^*\Y$ obtained by quantizing only in the fiber direction. The ellipticity condition implies that the dimension of the space of null solutions of the resulting operator that decrease rapidly in the positive (interior) direction does not change with the base point $\pmb\eta\ne0$, therefore defines a vector bundle. We note in passing that this may not be true if $T^*\Y\minus 0$ is not connected and one insists on vector bundles having constant rank. However, the statement just made becomes correct if one allows for fiber bundles to have nonconnected bases and/or fibers and nonconstant base and/or fiber dimension. All fiber bundles throughout the paper, including the boundary fibration \eqref{BdyFiberBundle}, should be taken in this extended sense. The relevant information the rapidly decreasing null solutions bring with them is encoded in their Cauchy data at $0\in \iota^*\pmb\eta$, which we are calling their trace. The latter spaces form a vector bundle $\kerb$ over $T^*\Y\minus 0$ which should be interpreted as a subbundle of the pull-back to $T^*\Y\minus 0$ of the trace bundle. The classical Shapiro-Lopatinskii condition is the condition that the restriction of the principal symbol of $B$ in \eqref{BVP} to $\kerb$ be an isomorphism onto the pullback of $\target$ to $T^*\Y$, see for example \cite[Chapter VI, Definition 1]{Seeley1969}.

In our setting, the above still holds, with a proviso. We require that each $A_\wedge(\pmb\eta)$, $\pmb\eta\ne0$, be injective on its minimal domain and surjective on its maximal domain. This condition, automatically satisfied in the case of a regular elliptic operator, is again reminiscent of a condition required by Mazzeo in \cite[Hypothesis 5.14]{Mazz91}. As a consequence, the family of kernels in $x^{-\gamma}L^2_b(\Z^\wedge)$ of the various $A_\wedge(\pmb\eta)$, $\pmb\eta\ne 0$, form the fibers of a finite rank smooth vector bundle $\KK\to T^*\Y\minus 0$, see Theorem~\ref{KernelBundle}. For each $\pmb\eta$ there is a trace map $\gamma_{A_{\wedge}(\pmb\eta)}:\Dom_{\max}(A_\wedge(\pmb\eta))\to \pi^*_\Y\trb_y$, $y=\pi_\Y(\pmb\eta)$, giving an isomorphism from $\KK$ to a subbundle $\kerb\subset \pi^*_\Y\trb$. In analogy with the classical case, each of the elements in $\KK$ is rapidly decreasing at $\infty$. The Shapiro-Lopatinskii condition can now be stated in our general setting in exactly the same terms: that the restriction of the principal symbol of $B$ to $\kerb$ be an isomorphism onto $\pi_\Y^*\target$.

At this point, the task becomes to construct a parametrix for the problem. This requires among other things a more careful set up for the spaces in which solutions are to be sought. To motivate this point we again resort to the familiar classical situation. From our perspective, the fiber of $\trb$ at $y\in \Y$ is, in the case of a regular elliptic differential operator of order $m$, the (restriction to $\R_+$ of the) space of $E_y$-valued polynomials of degree $\leq m-1$ in one variable $x$; in invariant terms, the variable ranges, for each $y$, in the fiber of the inward pointing part of the normal bundle to the boundary (we mentioned already that $\N=\Y$ in the classical situation). In a sense that can be made precise, these are the solutions of $A_\wedge(\pmb\eta)u=0$ when $\pmb\eta$ is an element of the zero section of $T^*\Y$. The {\it sections} of interest have as coefficients the traces at $\Y$ of the normal derivatives of orders up to $m-1$ (except for constant nonzero factors)  of elements typically in the Sobolev space $H^m(\M;E)$. As such, the regularity of the coefficients decreases in steps of $1$, beginning with the zeroth order term, the restriction of $u$ to $\Y$, which belongs to $H^{m-1/2}(\Y;E_\Y)$ down to the coefficient of order $m-1$ which of course belongs to $H^{1/2}(\M;E)$. These  changes in the Sobolev regularity of the various terms is reflected in the structure of the operator $B$, a matrix of pseudodifferential operators of various orders acting on the individual coefficients of the trace polynomial, with the orders reflecting the {\it a priori} Sobolev regularity of the coefficients. This is why, and how, the Douglis-Nirenberg calculus appears in the standard theory (see for instance \cite{ChazarainPiriou,Hor67}). 

In the general case treated here the structure of the fibers of $\trb$ is more complicated even though we are dealing only with first order operators. Normalizing the weight parameter $\gamma$ to $1/2$ henceforth in this introduction (this entails no loss of generality), $\trb_y$ is, over $y\in\Y$, a finite-dimensional space whose elements are finite sums
\begin{equation*}
u=\sum u_{\sigma,\ell}x^{\im \sigma}\log^\ell x 
\end{equation*}
where the $u_{\sigma,\ell}$ are smooth sections of $E$ over $\wp^{-1}(y)$, see \eqref{ProtoTrace}, with $|\Im\sigma|<1/2$ and nonnegative $\ell$. An effect of this is that the infinitesimal generator of the radial action (replace $x$ by $\varrho x$), acting fiberwise on elements of $\trb$, may have eigenvalues (the numbers $\im\sigma$) and multiplicities (the largest number $\ell+1$ for a given $\sigma$) that vary with $y$. Aside from the complications that this entails when discussing the $C^\infty$ structure of $\trb$, a matter resolved in \cite{KrMe12a}, it requires a careful redefinition of Sobolev spaces to account for varying regularity and a reassessment of the nature of the operator $B$, whose structure should be a generalization of the Douglis-Nirenberg scheme. These last two issues were both discussed in full length in \cite{KrMe12b}. In terms of these varying regularity spaces Theorem~\ref{RangeofGamma} gives that
\begin{equation*}
\gamma_A:\Dom_{\max}(A)\to H^{-\gen}(\Y;\trb)
\end{equation*}
is continuous. Here $\gen=x\partial_x+1/2$, $x\partial_x$ is the generator of the radial action, and the shift by $1/2$ ensures that the action is unitary on the space $x^{-1/2}L^2_b$. Theorem~\ref{RangeofGamma} reduces to the classical assertion if $x\partial_x$ acts trivially on $\trb$, that is, if $\trb_y$ is, for each $y$, a space of sections of $E$ over $\Z_y$ independent of $x$, since in this case the image space is just the standard Sobolev space $H^{-1/2}(\Y;\trb)$. Having been guided by the nature of the problem to allow for the action generated by $\gen$, we allow also the target bundle $\target$ of the boundary condition to come equipped with an action with infinitesimal generator $\agen$. The operator $B$ in \eqref{BVP} is then taken to be in $\Psi^{\mu}_{1,\delta}(\Y;(\trb,-\gen),(\target,{\agen}))$ for some fixed $0 < \delta < 1$. For details on the meaning of this class we again direct the reader to \cite{KrMe12b}.

The correct Sobolev space in which to pose the problem \eqref{BVP} is not $\Dom_{\max}(A)$, which is too large, but a space we refer to as $H^1_{\trb}(\M;E)$,  constructed in Section~\ref{sec-H1A}, see Definition~\eqref{H1ADef}. The notation emphasizes the role of $A$ (explicitly through $\trb$, implicit through the norm). The order reflects the fact that we are dealing here with first order operators. This space reduces to the standard Sobolev space $H^1(\M;E)$ in the classical case. The elements of $H^1_{\trb}(\M;E)$ do have local $H^1$ regularity in the interior and share with the standard Sobolev space $H^1(\M;E)$ the important fact that the inclusion into $x^{-1/2}L^2_b(\M;E)$ is compact (Corollary~\ref{compactembedding}).

With this the analytic setup of the problem is completed, and the next task is the construction of a parametrix. The problem we consider in Section~\ref{sec-BVP} is more general than the reader may have surmised from our discussion so far in that we allow for Atiyah-Patodi-Singer type boundary conditions. These are accounted for in the presence of a projection $\Pi\in \Psi^0_{1,\delta}(\Y;(\target,{\agen}),(\target,{\agen}))$. Both $\Pi$ and $B$ are assumed to have principal symbols
\begin{equation*}
\sym(\Pi)\in C^\infty(T^*\Y\minus 0;\End(\pi_\Y^*\target)),\quad \sym(B)\in C^\infty(T^*\Y\minus 0;\Hom(\pi_\Y^*\trb,\pi_\Y^*\target))
\end{equation*}
assumed to be homogeneous in a way that respects the actions $\agen$ and $\gen$. 

Thus we seek a parametrix for the operator
\begin{equation}\label{BVPOperator}
H^1_{\trb}(\M;E)\ni u\mapsto Au\oplus \Pi B\gamma_Au\in x^{-1/2}L^2_b(\M;F)  \oplus \Pi H^{1-\mu+{\agen}}(\Y;\target)
\end{equation}
assuming that $A$ is $w$-elliptic, that its boundary spectrum satisfies a mild condition alluded to above and that the normal operator satisfies the already mentioned injectivity and surjectivity conditions on its minimal and maximal domains, and in addition, that $\sym(\Pi)\circ\sym(B)$ restricts to an isomorphism from $\kerb$ to $\target_{\Pi}\subset \pi^*_\Y\target$ where the latter bundle is the range of the projection $\sym(\Pi)$. Theorem \ref{FredholmTheorem} asserts that \eqref{BVPOperator} is a Fredholm operator. We show in Section ~\ref{sec-BVP} that after a compact perturbation the problem, namely the construction of a suitable parametrix, becomes amenable to analysis with Schulze's apparatus. We have collected in the appendix (Appendices \ref{app-PseudosOpValued} and \ref{GreenOperators}) a primer on enough aspects of Schulze's theory to complete a construction of the parametrix.

The first order case addressed in this work is technically simpler than the higher order case in good measure because the maximal domain of a ($w$-elliptic) first order operator is a module over the ring $\Ring$ (see Lemma \ref{DmaxDminlocalize}). This is generally not the case for higher order operators. We plan to address boundary value problems for higher order operators, still with minimal assumptions in the same spirit as in this paper, in the near future. 

\medskip
Mazzeo and Vertman have shared with us a preliminary draft of their work [18] that is aimed at developing a theory of boundary value problems for higher order edge and wedge operators in the case of constant boundary spectrum. The trace bundle defined in [13] also appears in their theory as an important ingredient. Their paper [18] is a continuation of work that was begun in [16].

Schulze and collaborators have developed an extensive elliptic theory of pseudodifferential operators on manifolds with singularities over the past 25 years, see for example [6, 12, 22, 23, 25] for work pertaining to edge singularities. The paper [12] by Kapanadze, Schulze, and Seiler is intended to be a first step towards developing a theory of boundary value problems for edge pseudodifferential operators in SchulzeÕs calculus along the lines of Boutet de MonvelÕs calculus under strong assumptions on the admissible operators. The preprint [29] suggests possible modifications of the approach taken in [12] to pursue this task further.

\medskip
As a final point in this introduction we briefly discuss a curious perspective of boundary values (different from the one used in the rest of the paper) which may also help differentiate the operators used here from other classes  analyzed in the existing literature. 

The spaces $\Dom_{\max}(A)$ and $\Dom_{\min}(A)$ can of course be defined for any differential operator on $\open \M$. If $A$ is elliptic, then 
\begin{equation*}
\mathcal E=\set{u\in \Dom_{\max}:u\perp \Dom_{\min}(A)},
\end{equation*}
where orthogonality is meant in the sense of the inner product \eqref{GraphNorm}, can be viewed as the space of boundary values of elements of $\Dom_{\max}(A)$, and the orthogonal projection $\pi_{\max}:\Dom_{\max}(A)\to\Dom_{\max}(A)$ onto $\mathcal E$ as the trace map (in the sense of Cauchy data at the boundary), as follows. First let $A^\star$ denote the formal adjoint of $A$ and note that $\mathcal E$ is the subspace of $\Dom_{\max}(A)$ whose elements satisfy $A^\star Au=-u$. This is easy to prove (alternatively the reader may consult the proof of Lemma 4.2 of \cite{GiKrMe07}; the context there is slightly different but the proof is the same). The implication is that elements of $\mathcal E$ are smooth in $\open\M$, by ellipticity. So, if $\omega\in C^\infty(\M)$ is equal to $1$ near $\partial\M$, then one has $(1-\omega)u\in C_c^\infty(\open\M;E)$, hence $\omega u\in \Dom_{\max}(A)$ and $\pi_{\max}(\omega u)=\pi_{\max}(u)$. Thus what $u$ is in the complement of any neighborhood of the boundary has no bearing on what $\pi_{\max}(u)$ is. It is in this sense that elements of $\mathcal E$ can be regarded as boundary values of elements of $\Dom_{\max}(A)$, and $\pi_{\max}$ as a trace map. If $A$ is a first order operator, then $\Dom_{\max}(A)$ is closed under multiplication by elements of $\Ring$ (see Lemma~\ref{DmaxDminlocalize}) which hints at the (correct) notion that boundary values live on $\Y$ (as opposed to $\N$).

Evidently, if $\Dom_{\min}(A)=\Dom_{\max}(A)$, then there is no issue with boundary values. Typically, elliptic wedge differential operators of positive order do have $\Dom_{\min}(A)\ne\Dom_{\max}(A)$, and so boundary value problems for these do make sense. Situations in which these two domains are equal (always assuming ellipticity in the right context) include Mazzeo and Melrose's $\mathscr V_0$-differential operators \cite{MaMe87}, the already mentioned edge operators of Mazzeo \cite{Mazz91} in which the former correspond to $\Z=\textrm{point}$, differential operators in the $\Theta$-calculus of Epstein, Melrose and Mendoza \cite{EpRBMMe91}, the families associated with Lie structures at infinity analyzed by Amman, Lauter, and Nistor's operators \cite{AmLaNi07}, as well as classes subsumed in the latter such as the operators of fibered (Mazzeo and Melrose \cite{MaMe98}) or foliated cusp type (Rochon \cite{Ro12}). In all these cases, the minimal and maximal domains of elliptic elements coincide, so as indicated above, there are no issues with boundary values, hence no boundary value problems. For the theoretical underlying, see Kordyukov \cite{Ko91} and Shubin \cite{Shub92}.

\section{Wedge operators and their symbols}\label{sec-SummaryWedgeOps}

A wedge operator $A$ has a principal symbol $\wsym(A)$ defined on the $w$-cotangent bundle, see \cite[Section 2]{GiKrMe10}. The latter, denoted $\wpi:\wT^*\M\to\M$, is the bundle whose space of smooth sections is
\begin{equation}\label{wBracket}
\set{\alpha\in C^\infty(\M;T^*\M):\iota_y^*\alpha=0\ \forall y\in \Y}
\end{equation}
(recall that $\iota_y:\Z_y\to \N$ is the inclusion map). These sections are locally spanned by linear combinations of differentials of functions in $\Ring$. In local adapted coordinates $x,y,z$, the elements 
\begin{equation*}
dx,\ dy_j,\ x dz_\mu
\end{equation*}
give a local frame of $\wT^*\M$. The principal $w$-symbol at
\begin{equation*}
\nu=\xi dx+\sum_\mu \eta_jdy^j+x\zeta_\mu dz_\mu
\end{equation*}
of $A\in x^{-m}\Diff^m_e(\M;E,F)$ given in coordinates by \eqref{TypicalWOp} is
\begin{equation*}
\wsym(A)(\nu)=\sum_{k+|\alpha|+|\beta|=m} a_{k\alpha\beta}(x,y,z)\xi^k \eta^\alpha\zeta^\beta
\end{equation*}
This object can be obtained as a section of
\begin{equation*}
\Hom(\wpi^*E,\wpi^*F)\to\wT^*\M
\end{equation*}
via an oscillatory test involving elements of $\Ring$, see \cite{GiKrMe10}. As expected, an element $A\in x^{-m}\Diff^m_e(\M;E,F)$ is $w$-elliptic if $\wsym(A)(\nu)$ is invertible for each $\nu\in \wT^*\M\minus 0$.

There are two other symbols associated with $A$. The first of these is the indicial family. To define it, recall that an element $P\in \Diff^m(\M;E,F)$ is totally characteristic, written $P\in \Diff^m_b(\M;E,F)$, if the coefficients of $x^{-\nu}Px^\nu$ are smooth up to the boundary for $\nu=1,\dotsc,m$. Next, note that if $P\in \Diff^m_b(\M;E,F)$, then $P$ gives by restriction an operator on the boundary, $\bP\in \Diff^m_b(\N;E_\N,F_\N)$. The indicial family of $P$ is then defined to be the operator
\begin{equation*}
\bPhat(\sigma)= {}^b(x^{-\im\sigma}Px^{\im\sigma}).
\end{equation*}
This operator depends on the defining function $x$. However, the indicial families of $P$ obtained through different defining functions differ only by conjugation with a  factor $e^{\im\sigma g}$ where $g$ is a smooth real-valued function on $\N$. Let $\pi_\wedge:\N^\wedge\to\N$ be the closed inward  pointing normal bundle of $\N$ in $\M$. The zero section of the normal bundle is $\partial\N^\wedge$, and is identified with $\N$. If $x$ is the defining function used above, then the fiber variable is taken to be the restriction of $dx$ to $\N^\wedge$. We also write $x$ for this variable. Via inverse Mellin transform we obtain an operator on $\N^\wedge$, denoted $\bP$. When $P\in\Diff^m_e(\M;E,F)$, the resulting operator commutes with multiplication by elements of $\wp_\wedge^*C^\infty(\Y)$, so $\bP$ can be regarded as a family of operators $\bP_y$ along $\Z_y^\wedge$, the part of $\N^\wedge$ over the fiber $\Z_y$. Here we used the notation defined through the bottom part of the diagram
\begin{equation}\label{MicrolocalBdyFiberBundle}
\display{300pt}{\begin{center}
\begin{picture}(140,74)
\put(120,62){$T^*\Y$	
\put(-11,-3){\curve(0,0,0,-23)\put(0,-25){\vector(0,-1){0}}}
\put(-5,-18){\small $\pi_\Y$}
\put(-14,-40){$\Y$}
\put(-50,-37){\curve(0,0,30,0)\put(32,0){\vector(1,0){0}}}
\put(-39,-31){\small $\wp$}
\put(-63,-40){$\N$}
\put(-99,-37){\curve(0,0,30,0)\put(32,0){\vector(1,0){0}}}
\put(-88,-31){\small $\pi_\wedge$}
\put(-116,-40){$\N^\wedge$}
\put(-113,-45){\curve(6,3,51,-6,95,3)\put(97,4){\vector(3,1){0}}}
\put(-63,-58){\small$\wp_\wedge$}
\put(-126,0){$\wp_\wedge^*T^*\Y$}
\put(-110,-3){\curve(0,0,0,-23)\put(0,-25){\vector(0,-1){0}}}
\put(-128,-18){\small$\pi^\wedge_\Y$}
\put(-90,2){\curve(0,0,64,0)\put(66,0){\vector(1,0){0}}}
\put(-68,7){\small$\wp_{\wedge,\Y}$}
}
\end{picture}
\end{center}}
\end{equation}
which also serves to introduce $\wp_{\wedge,\Y}$ as notation for the canonical map at the top. We will also write $\open\wp_\wedge$ and $\open\wp_{\wedge,\Y}$ for the restrictions of these maps to the interior of their domains. In local adapted coordinates we have
\begin{equation*}
\bP_y=\sum_{k+|\beta|\leq m}a_{k0\beta}(0,y,z)(x D_x)^kD_z^\beta
\end{equation*}
along $\Z^\wedge_y$. We write
\begin{equation}\label{IndicialOp}
\bA=x^{-m}\bP
\end{equation}
when $P=x^mA$. Thus
\begin{equation*}
\bA_y\in \Diff^m_b(\Z_y^\wedge;E_{\Z_y^\wedge},F_{\Z_y^\wedge})
\end{equation*}
for each $y$; somewhat inconsistently we write $E_{\Z_y^\wedge}$ for the part of $E^\wedge = \pi_\wedge^*E$ over $\Z_y^\wedge$. If $A$ is $w$-elliptic, then $\bP_y$ is $b$-elliptic for each $y\in \Y$ as an element of $\Diff^m_b(\Z_y^\wedge;E_{\Z_y^\wedge},F_{\Z_y^\wedge})$ and we let
\begin{equation*}
\spec_e(A)=\set{(y,\sigma)\in \Y\times\C:\sigma\in\spec_b(\bP_y)}.
\end{equation*}
Recall that $\spec_b(\bP_y)$ consists of those $\sigma\in \C$ such that $\widehat \bP_y(\sigma)$ is not invertible, see Melrose \cite{Mel93}.

It will be convenient to denote by $\Ring^\wedge$ the analogue of $\Ring$ in the case of $\N^\wedge$:
\begin{equation*}
\Ring^\wedge=\set{f\in C^\infty(\N^\wedge):f|_\N\text{ is constant on the fibers of }\wp}.
\end{equation*}

The last symbol associated to $A\in x^{-m}\Diff_e(\M;E,F)$ is its normal family, obtained as follows. Let $\varphi$ be a tubular neighborhood map from a neighborhood of $\partial\N^\wedge=\N$ in $\N^\wedge$ to one of $\N$ in $\M$. If $g\in C^\infty(\Y)$ is real-valued, then 
\begin{equation}\label{NormalFamilyAsLimit}
A_\wedge (dg)u= \lim_{\varrho\to \infty} \varrho^{-m}\kappa^{-1}_\varrho e^{-\im \varrho \wp_\wedge^*g} \Phi^* A \Phi_* e^{\im \varrho \wp_\wedge^*g} \kappa_\varrho u,
\end{equation}
see \cite[Proposition 2.10]{GiKrMe10}. Here $\Phi_*$, it inverse $\Phi^*$, and $\kappa_\varrho$ are defined using $\varphi$, the radial action on $\N^\wedge$, and parallel transport with respect to Hermitian connections on $E$ and $F$ and their liftings to $E^\wedge$ and $F^\wedge$ via $\pi_\wedge$. The map $\kappa_\varrho$ is given by 
\begin{equation}\label{KappaOnNWedge}
(\kappa_\varrho u)(\nu)=\varrho^\gamma u(\varrho \nu)
\end{equation}
for some fixed number $\gamma$. The limit in \eqref{NormalFamilyAsLimit} yields a family $\pmb\eta\mapsto A_\wedge(\pmb\eta)$ of cone-differential operators which commutes with multiplication by functions which are liftings to $\N^\wedge$ of smooth functions on $\Y$ so may be viewed as a family
\begin{equation*}
T^*\Y\ni \pmb\eta \mapsto A_\wedge(\pmb\eta) \in x^{-m}\Diff_b^m(\Z_y^\wedge;E_{\Z_y^\wedge}^\wedge,F_{\Z_y^\wedge}^\wedge), \quad y=\pi_{\Y}\pmb\eta,
\end{equation*}
see \cite[formula (2.14)]{GiKrMe10}. Furthermore, 
\begin{equation}\label{WedgeSymbolKappaHom}
A_\wedge(\varrho\, \pmb\eta) = \varrho^m \kappa_\varrho  A_\wedge(\pmb\eta)\kappa_\varrho^{-1}.
\end{equation}
Again in local adapted coordinates, with $\pmb\eta=\eta\cdot dy$ as a covector at $y$, we have
\begin{equation*}
A_\wedge(\pmb\eta)=x^{-m}\sum_{k+|\alpha|+|\beta|\leq m} a_{k\alpha\beta}(0,y,z)(xD_x)^k (x\eta)^\alpha D_z^\beta.
\end{equation*}

\section{Spaces and domains}\label{sec-Domains}

We now pick a $b$-density $\m_b$ on $\M$ (see \cite{Mel93}) which shall remain fixed for the rest of the paper, a density $\m_\Y$ on $\Y$, and Hermitian structures on $E$ and $F$; the densities are all smooth and  positive, of course. The chosen $b$-density determines uniquely a density on $\N$ (by contraction with the canonical section $x\partial_x$ of $\bT\M$ along $\N$) which we denote $\m_\N$. Using $\m_\Y$ and $\m_\N$ we now get a density $\m_{\Z_y}$ on $\Z_y$ for each $y\in \Y$. The vector field $x\partial_x$ on $\N^\wedge$ is canonical (the infinitesimal generator of the radial action). With it we get a canonical $b$-density $\m^\wedge_b$ on $\N^\wedge$ by requiring
\begin{equation*}
\Lie_{x\partial_x}\m^\wedge_b +\m^\wedge_b = 0,\quad x\partial_x \contract\m^\wedge_b=\m_\N
\end{equation*}
with $\Lie_{x\partial_x}$ denoting Lie derivative. Since $\wp_\wedge:\N^\wedge\to\Y$ is surjective (this is a locally trivial fibration with typical fiber $\Z\times \lbra 0,\infty\rpar$), we also have densities induced on each of the fibers of $\Z^\wedge$. All together, we have a coherent choice of measures on all manifolds of interest. The densities and Hermitian forms are now used to define the various $L^2$ spaces of sections, $L^2_b$ being the $L^2$ space with respect to $\m_b$.

The edge-Sobolev spaces of integral nonnegative orders $m$ were defined in \cite[Formula 3.22]{Mazz91} as consisting of those sections $u$ of $E$ which are in $L^2$ with respect to a smooth positive density on $\M$ such that $Pu$ is again in $L^2$ for every edge differential operator of order at most $m$. Here we shall diverge slightly from this convention and base the definition on $\m_b$ rather than a regular density: let $H^m_e(\M;E)$ be the space of sections $u$ of $E$ such that $Pu\in L^2_b(\M;E)$ for every $P\in \Diff^m_e(\M;E)$. The edge-Sobolev spaces of general order $m\in \R$ are defined using interpolation and duality. These spaces come equipped with norms that shall remain unnamed (see \cite{Mazz91} for details on this). The space $C_c^\infty(\open\M;E)$ is dense in $H^m_e(\M;E)$ for any $m$.

\medskip
Let $A\in x^{-m}\Diff^m_e(\M;E,F)$ be an elliptic wedge operator. Recall that the maximal and minimal domains of $A$ relative to the weight $x^{2\gamma}$, $\gamma\in \R$ are defined by setting
\begin{equation*}
\Dom_{\max}(A)=\set{u\in x^{-\gamma}L^2_b(\M;E):Au\in x^{-\gamma}L^2_b(\M;F)}
\end{equation*}
and $\Dom_{\min}(A)$ equal to the closure of $C_c^\infty(\open\M;E)$ in $\Dom_{\max}$ with respect to the graph norm, the norm defined by the inner product
\begin{equation*}
(u,v)_A=(Au,Av)_{x^{-\gamma}L^2_b}+(u,v)_{x^{-\gamma}L^2_b}.
\end{equation*}
These definitions are applicable also to the case of $b$-operators. If $A$ is $w$-elliptic, each of the operators $A_\wedge(\pmb\eta)$ is $b$-elliptic and has its own maximal and minimal domains. By \cite[Proposition 4.3, part (\emph i)]{GiKrMe10}, $\Dom_{\min}(A_\wedge(\pmb\eta))$ depends only on $y$. Note that
\begin{equation*}
A_\wedge(0)=\bA
\end{equation*}
($\bA$ was defined in \eqref{IndicialOp}).

Let $A^\star$ denote the formal adjoint of $A$ with respect to $x^{-\gamma}L^2_b$:
\begin{equation*}
(Au,v)_{x^{-\gamma}L^2_b}=(u,A^\star v)_{x^{-\gamma}L^2_b},\quad u\in C_c^\infty(\open\M;E),\ v\in C_c^\infty(\open\M;F).
\end{equation*}

\begin{lemma}\label{AdjointOfNormalFamily}
Let $f\in \Ring$ be a strictly positive function and $A\in x^{-m}\Diff^m_e(\M;E,F)$ and $A^\star$ its formal adjoint with respect to $\m_b$ and the weight $x^{2\gamma}$. Then the formal adjoint of $A$ with respect to $f\m_b$ and the weight $x^{2\gamma}$ is
\begin{equation*}
A^\star +Q
\end{equation*}
for some $Q\in x^{-m+1}\Diff^{m-1}_e(\M;F,E)$. It follows that the normal family of $A^\star$ is equal to the formal adjoint of the normal family, $A_\wedge$, of $A$ with respect to $\m^\wedge_b$ and the same weight $x^{2\gamma}$.
\end{lemma}

\begin{proof}
Let $u\in C_c^\infty(\open\M;E),\ v\in C_c^\infty(\open\M;F)$. Then
\begin{align*}
(Au,v)_{x^{-\gamma} L^2(\M;F,f\m_b})
&=\int(Au,v)_F\,fx^{2\gamma}d\m_b\\
&=\int(u,A^\star fv)_E\,x^{2\gamma}d\m_b\\
&=\int(u,A^\star v+\frac{1}{f}[A^\star,f]v)_E\,fx^{2\gamma}d\m_b\\
&=(u,(A^\star v+\frac{1}{f}[A^\star,f]v))_{x^{-\gamma} L^2(\M;E,f\m_b)}
\end{align*}
Let $Q=f^{-1}[A^\star,f]$. Since $f\in \Ring$, $Q\in x^{-m+1}\Diff^{m-1}_e(\M;F,E)$ by \eqref{BracketWR}.

Let $\varphi$ be the tubular neighborhood map used in the previous section. Then $\varphi_*\m^\wedge_b=f\m_b$ in a neighborhood of $\N$ with $f=1$ on $\N$. Thus $f\in \Ring$ (near $\N$) and we have, if $u\in C_c^\infty(\open\N^\wedge;E^\wedge),\ v\in C_c^\infty(\open\N^\wedge;F^\wedge)$ and $g\in C^\infty(\Y)$,
\begin{align*}
\big(\kappa^{-1}_\varrho e^{-\im \varrho \wp_\wedge^*g} \Phi^* A& \Phi_* e^{\im \varrho \wp_\wedge^*g}\kappa_\varrho u ,v\big)_{x^{-\gamma}L^2(\N^\wedge,F^\wedge,\m^\wedge_b)}\\
&=\big(A \Phi_* e^{\im \varrho \wp_\wedge^*g}\kappa_\varrho u , \Phi_* e^{\im \varrho \wp_\wedge^*g} \kappa_\varrho v\big)_{x^{-\gamma}L^2(\M,F,f\m_b)}\\
&=\big(\Phi_* e^{\im \varrho \wp_\wedge^*g}\kappa_\varrho u , (A^\star+Q)\Phi_* e^{\im \varrho \wp_\wedge^*g} \kappa_\varrho v\big)_{x^{-\gamma}L^2(\M,E,f\m_b)}\\
&=\big(u , \kappa^{-1}_\varrho  e^{-\im \varrho \wp_\wedge^*g}\Phi^* (A^\star+Q)\Phi_* e^{\im \varrho \wp_\wedge^*g} \kappa_\varrho v\big)_{x^{-\gamma}L^2(\N^\wedge,E^\wedge,\m^\wedge_b)}
\end{align*}
where $\kappa$  is as in \eqref{KappaOnNWedge}. Now observe that
\begin{equation*}
\lim_{\varrho\to \infty} \varrho^{-m}\kappa^{-1}_\varrho e^{-\im \varrho \wp_\wedge^*g} \Phi^* (A^\star +Q)\Phi_* e^{\im \varrho \wp_\wedge^*g} \kappa_\varrho v=A^\star_\wedge(dg)v
\end{equation*}
because $\lim_{\varrho\to \infty} \varrho^{-m+1}\kappa^{-1}_\varrho e^{-\im \varrho \wp_\wedge^*g} \Phi^* Q\Phi_* e^{\im \varrho \wp_\wedge^*g} \kappa_\varrho =Q_\wedge(dg)$.
\end{proof}

\begin{lemma}\label{xmDmax}
Let $A\in x^{-m}\Diff^m_e(\M;E,F)$ be $w$-elliptic. If $u\in \Dom_{\max}(A)$ then $x^m u\in x^{-\gamma+m}H^m_e(\M;E)$. Consequently $x^m\Dom_{\max}(A)\subset \Dom_{\min}(A)$.
\end{lemma}

\begin{proof}
Suppose $u\in \Dom_{\max}(A)$. So $u\in x^{-\gamma}L^2_b(\M;E)$ and $Au\in x^{-\gamma}L^2_b(\M;F)$. Hence $x^mAu\in x^{-\gamma+m}L^2_b(\M;F)$, therefore $x^mAu\in x^{-\gamma}L^2_b(\M;E)$. Since $x^mA$ is an elliptic edge operator, Theorem~3.8 of \cite{Mazz91} gives that $u\in x^{-\gamma}H^m_e(\M;E)$, so $x^m u \in x^{-\gamma+m}H^m_e(\M;E)$. For the last statement observe that $x^{-\gamma+m}H^m_e(A;E)\subset\Dom_{\min}(A)$.
\end{proof}

The following theorem, of fundamental importance in this paper, attests to the relevancy of these symbols:
\begin{theorem}[{\cite[Theorem 4.4]{GiKrMe10}}]\label{DminTheorem}
Let $A \in x^{-m}\Diff^m_e(\M;E,F)$ be $w$-elliptic. If 
\begin{equation}\label{NoBSpecOnBottom}
\spec_{e}(A) \cap \Y\times \{\sigma \in \C : \Im\sigma = \gamma-m\}= \emptyset
\end{equation}
and
\begin{equation}\label{NormalInjectivity}
A_\wedge(\pmb\eta) : \Dom_{\min}(A_\wedge(\pmb\eta)) \subset x^{-\gamma}L^2_b \to x^{-\gamma}L^2_b\text{ is injective for all }\pmb\eta \in T^*\Y \minus 0
\end{equation}
then
\begin{equation*}
\Dom_{\min}(A) = x^{-\gamma+m}H^m_{e}(\M;E).
\end{equation*}
Moreover, 
\begin{equation*} 
A : \Dom_{\min}(A) \subset x^{-\gamma}L^2_b(\M;E)\to x^{-\gamma}L^2_b(\M;F)
\end{equation*}
is a semi-Fredholm operator with finite-dimensional kernel and closed range.
\end{theorem}

Let $\omega\in C_c^\infty(\R)$ be arbitrary but with $\omega(x)=1$ near $0$. The $w$-ellipticity of $A$ implies that the minimal and maximal domains of 
\begin{equation*}
A_\wedge(\pmb\eta):C_c^\infty(\open\Z_y^\wedge;E)\subset x^{-\gamma}L^2_b(\Z_y^\wedge;E_{\Z_y^\wedge})\to x^{-\gamma}L^2_b(\Z_y^\wedge;F_{\Z_y^\wedge})
\end{equation*}
have the property that
\begin{equation*}
(1-\omega) \Dom_{\min}(A_\wedge(\pmb\eta))=(1-\omega) \Dom_{\max}(A_\wedge(\pmb\eta)) = (1-\omega) x^{\frac{1}{2}\dim\Z_y^\wedge-\gamma} H^m_{\cone}(\Z_y^\wedge;E_{\Z_y^\wedge}).
\end{equation*}
The space $H^m_\cone(\Z^\wedge;E)$ is Schulze's cone Sobolev space, which near $x=\infty$ is identical to Melrose's scattering Sobolev space $\scH^m(\Z\times\lbra0,1\rpar)$ near $x'=0$ resulting from compactifying $\open\Z^\wedge=\Z\times(0,\infty)$ at $\infty$ to $\Z\times\lbra 0,\infty\rpar$ via $x'=1/x$, cf. \cite{Mel95} and \cite{SchuNH}. The vector bundle $E$ is taken here to be the pull-back to $\Z^\wedge$ of one such bundle over $\Z$, with the induced Hermitian metric. 

In the presence of \eqref{NoBSpecOnBottom}, $\omega\Dom_{\min}(A_\wedge(\pmb\eta))$ was identified in \cite[Proposition 3.6]{GiMe01}, see also \cite[Proposition 4.1]{GiKrMe10}, giving
\begin{equation*}
\Dom_{\min}(A_\wedge(\pmb\eta))=\omega x^{-\gamma+m}H^m_b(\Z_y^\wedge;E_{\Z_y^\wedge})+(1-\omega) x^{\frac{1}{2}\dim\Z_y^\wedge-\gamma} H^m_\cone(\Z_y^\wedge;E_{\Z_y^\wedge}).
\end{equation*}
This is a simple consequence of ellipticity coupled with a uniform \emph{a priori} estimate. Thus with the notation
\begin{equation}\label{Kegel}
\K^{s,\gamma}_t(\Z^\wedge;E) = \omega x^{\gamma}H^s_b(\Z^\wedge;E_{\Z_y^\wedge}) + (1-\omega)x^{-t}H_{\cone}^s(\Z^\wedge;E_{\Z_y^\wedge})
\end{equation}
for $s,t,\gamma \in \R$ we have
\begin{equation}\label{DomMinAWedge}
\Dom_{\min}(A_\wedge(\pmb\eta))=\K^{m,-\gamma+m}_{\gamma-\frac{1}{2}\dim\Z_y^\wedge}(\Z_y^\wedge;E_{\Z_y^\wedge}).
\end{equation}
Note that
\begin{equation*}
\K^{0,-\gamma}_{\gamma-\frac{1}{2}\dim\Z_y^\wedge}(\Z_y^\wedge;E_{\Z_y^\wedge}) = x^{-\gamma}L^2_b(\Z_y^\wedge;E_{\Z_y^\wedge}).
\end{equation*}
The spaces $\Dom_{\max}(A_\wedge(\pmb\eta))$ do generally depend on $\pmb\eta$. However
\begin{equation}\label{DomainsDependOnYOnly}
\display{300pt}{Suppose $A$ is a $w$-elliptic first order operator satisfying \eqref{NoBSpecOnBottom}. Then $\Dom_{\min}(A_\wedge(\pmb\eta))$ and $\Dom_{\max}(A_\wedge(\pmb\eta))$ depend  on $\pmb\eta\in T^*\Y$ only through $y=\pi_\Y\pmb\eta$.}
\end{equation}
This follows at once from the observation that, in the first order case and with $0$ being the origin in $T^*_y\Y$, $A_\wedge(\pmb\eta)-A_\wedge(0)$ is a bounded operator 
$x^{-\gamma}L^2_b(\Z_y^\wedge;E_{\Z_y^\wedge})\to x^{-\gamma}L^2_b(\Z_y^\wedge;F_{\Z_y^\wedge})$. See also  \cite[Proposition 4.1]{GiMe01}.

\medskip
It will be useful later on to regard $A_\wedge(\pmb\eta)$ as a family of operators acting fiberwise on (fiberwise subspaces of) a Hilbert space bundle. To this end view 
\begin{equation*}
\Ha_E=\bigsqcup_{y\in \Y}x^{-\gamma}L^2_b(\Z_y^\wedge;E_{\Z_y^\wedge})
\end{equation*}
as a smooth Hilbert space bundle over $\Y$. The local trivializations of $\Ha_E$, carefully described in Section 3 of \cite{GiKrMe10}, are unitary and constructed using trivializations of the boundary fibration $\N\to \Y$ over a neighborhood $U\subset \Y$ of the central point $y_0=\pi_\Y(\pmb\eta_0)\in \Y$, diffeomorphisms $\Z_y\to \Z_{y_0}$, $y\in U$, plus parallel transport with respect to some auxiliary connection on $E$ along ``horizontal'' curves in  $\N$ issuing normally from $\Z_{y_0}$, plus certain factors ensuring unitarity. Such trivializations $\Psi$ preserve Sobolev spaces as well as the spaces $\K^{s,\gamma}_t$ discussed above and have the following feature: 
\begin{equation*}
\display{300pt}{
A family of elements $u(y)$ of $x^{-\gamma}L^2_b(\Z_{y_0}^\wedge;E_{\Z_{y_0}^\wedge})$, $y\in U$, is smooth as section of $E_{\Z^\wedge_{y_0}}$ for $(x,y,z)\in \lbra0,\infty\rpar\times U\times\Z_{y_0}$  iff the corresponding section $\Psi^{-1}u$ of $\Ha_E$ on $U$ is smooth as a section of $E^\wedge$ over $\open\wp_\wedge^{-1}(U)$. 
}
\end{equation*}
See \eqref{MicrolocalBdyFiberBundle} for the definition of the map $\wp_\wedge$. We lift $\Ha_E$ to a Hilbert space bundle over $T^*\Y$ using the map $\wp_{\wedge,\Y}$ defined through the same diagram, likewise the analogously defined vector bundle $\Ha_F$. The normal family of $A$ acts fiberwise, as
\begin{equation*}
A_\wedge:\Dom_{\max}(A_\wedge)\subset \pi_\Y^*\Ha_E\to \mathcal \pi_\Y^*\Ha_F
\end{equation*}
where we have set
\begin{equation}\label{DmaxBundle}
\Dom_{\max}(A_\wedge)=\bigsqcup_{\pmb\eta\in T^*\Y}\Dom_{\max}(A_\wedge(\pmb\eta)).
\end{equation}
With this point of view the collection of fiberwise minimal domains of $A_\wedge$ can be regarded, in the presence of \eqref{NoBSpecOnBottom}, as a a subbundle of $\pi_\Y^*\Ha_E$, a trivial Hilbert space bundle with typical fiber $\K^{m,-\gamma+m}_{\gamma-\frac{1}{2}\dim\Z^\wedge}(\Z^\wedge;E_{\Z^\wedge})$:
\begin{equation}\label{BundleOfMinimalDomains}
\display{300pt}{\begin{center}
\begin{picture}(140,50)
\put(80,0){$T^*\Y$.}
\put(0,45){$\K^{m,-\gamma+m}_{\gamma-1/2\dim \Z^\wedge}\embed \Dom_{\min}(A_\wedge)\embed \Ha_E$}
\put(89,39){\curve(0,0,0,-27)\put(0,-27){\vector(0,-1){0}}}
\end{picture}
\end{center}}
\end{equation}

\section{Trace bundles}\label{sec-TraceBundles}

Let $A\in x^{-m}\Diff^m_e(\M;E,F)$ be $w$-elliptic and assume that for some $\gamma\in \R$,
\begin{equation}\label{NobSpecOnCritLinesOrderm}
\spec_e(A)\cap\big(\Y\times\set{\sigma\in \C:\Im\sigma=\gamma,\ \gamma-m} \big)=\emptyset.
\end{equation}
The trace bundle of $A$ (relative to the  weight $x^{2\gamma}$) was defined in \cite[Section 6]{KrMe12a}. It is the vector bundle $\trb\to\Y$ whose fiber at $y\in \Y$ consists of those distributional sections of $E_{\Z_y^\wedge}$ of the form
\begin{equation}\label{ProtoTrace}
u=\sum_{\substack{\sigma\in \spec_b(\bP_y)\\ \gamma-m<\Im\sigma<\gamma}}\sum_{\ell=0}^{L_\sigma} u_{\sigma,\ell}x^{\im \sigma}\log^\ell x 
\end{equation}
that belong to the kernel of $\bP_y$. The $u_{\sigma,\ell}$ are sections of $E$ along $\Z_y$, smooth because of the ellipticity of $A$. We showed there that 
\begin{equation*}
\trb=\bigsqcup_{y\in\Y}\trb_y
\end{equation*}
with the canonical projection map $\pi_\trb:\trb\to\Y$ is, under the assumption \eqref{NobSpecOnCritLinesOrderm}, naturally a smooth vector bundle. The notion of smoothness is given by the following statement:
\begin{equation}\label{Smoothness}
\display{300pt}{the local smooth sections of $\trb\to\Y$ are elements of the form \eqref{ProtoTrace} depending on $y$ which as sections of $E^\wedge$ over the interior $\open \N^\wedge$ of $\N^\wedge$ are smooth in the usual sense.}
\end{equation}

\begin{proposition}\label{xdxAction}
The operator $x\partial_x$ acts fiberwise on $\trb$ and defines a smooth vector bundle homomorphism.
\end{proposition}

\begin{proof}
Let $y\in \Y$. If $u$ is of the form \eqref{ProtoTrace}, then $x\partial_x u$ has the same form. If further $u\in \ker \bP_y$, then also $\bP_y (x\partial_x u)\in \ker \bP_y$, since $\bP_y$ commutes with $x\partial_x$. Thus $x\partial_x:\trb_y\to\trb_y$. If now $u$ is a smooth section of $\trb$, then by definition it is smooth as a section of $E^\wedge$ over $\open\N^\wedge$, hence the same is true of $x\partial_x u$, and consequently, again by definition, $x\partial_x u$ is also a smooth section of $\trb$. Thus $x\partial_x$ defines a smooth homomorphism $\trb\to\trb$. 
\end{proof}

Later we will be interested in the vector bundle endomorphism
\begin{equation}\label{Generator}
\gen=x\partial_x+\gamma:\trb\to\trb,
\end{equation}
with the shift by $\gamma$ included for compatibility with the unitary action $(\kappa_\varrho u)(\nu)=\varrho^\gamma u(\varrho \nu)$ on $x^{-\gamma}L^2_b$. It is clear from the definition of the fiber $\trb_y$, see \eqref{ProtoTrace}, that, with $\gen_y$ denoting the action of $\gen$ on that fiber that
\begin{equation}\label{SpecGen}
\spec(\gen_y)=\set{\im\sigma+\gamma: \sigma\in \spec_b(\bA_y)\text{ and }\gamma-m<\Im \sigma <\gamma}.
\end{equation}

Sections of the trace bundle (of various regularities) are boundary values of elements in the maximal domain of $A$, see Theorem~\ref{RangeofGamma} in the case of a first order operator. Conversely, these sections, by being actually sections of $E^\wedge$ over $\open \N^\wedge$ they bring with themselves the essence of an extension operator. The following lemma should be viewed in the latter context.

\begin{lemma}\label{CanonicalMap}
The map $\mathfrak i: C^\infty(\Y;\trb)\to C^\infty(\open\N^\wedge;E^\wedge)$ that takes an element $u\in C^\infty(\Y;\trb)$ and regards it as a smooth section of $E^\wedge$ over $\open \N^\wedge$ is continuous.
\end{lemma}

\begin{proof}
Let $U\subset \Y$ be open, the domain of a smooth frame $\tau_1\dotsc\tau_N$ of $\trb$. The map $\mathfrak i$ also makes sense as a map $C^\infty(U;\trb)\to C^\infty(\open W;E^\wedge)$ where $W=\wp_\wedge^{-1}U$. If $u=\sum_\mu u^\mu\tau_\mu$ with smooth $u_\mu$, then $\mathfrak i(u)=\sum_\mu \wp_\wedge^*(u^\mu)\mathfrak i\tau_\mu$. The continuity of $\mathfrak i$ is thus seen to be equivalent to the continuity of the map $C^\infty(U;\C^N)\to C^\infty(\open W;E^\wedge)$, which is obvious.
\end{proof}
The map $\mathfrak i$ will usually be implicit in the sequel.

\medskip
The following observation will be used in the proof of Proposition~\ref{BasicParing}:

\begin{lemma}\label{SmoothnesOnMellinSide}
Let $U\subset \Y$ be open, $u:U\to \trb$ a smooth section of $\trb$. Pick $\omega\in C_c^\infty(\R)$ arbitrarily with $\omega(x)=1$ near $0$. Then the Mellin transform of $\omega u$,
\begin{equation}\label{MellinOmegaU}
\widehat{\omega u}(y,z,\sigma)=\int_0^\infty \omega(x) u(x,y,z) x^{-\im \sigma}\,\frac{dx}{x}
\end{equation}
is smooth in the complement of
\begin{equation*}
\set{(y,z,\sigma)\in \wp^{-1}(U)\times\C:(y,\sigma)\in \spec_e(A)}
\end{equation*}
in $\wp^{-1}(U)\times\C$ and meromorphic in $\C$ for each fixed $(y,z)\in \wp^{-1}(U)$. 
\end{lemma}

Functions like $\omega$ in the statement of the lemma are referred to as cut-off functions.

\medskip
For the sake of notational simplicity when dealing with adjoints we will henceforth assume that
\begin{equation*}
\gamma=m/2, 
\end{equation*}
a situation we are reduced to by replacing the original operator $A$ with 
\begin{equation*}
x^{\gamma-m/2}Ax^{m/2-\gamma}. 
\end{equation*}
Condition \eqref{NobSpecOnCritLinesOrderm} remains valid for the new operator, now with $\gamma=m/2$; of course also the definition of $u$ in \eqref{ProtoTrace} changes accordingly.

\medskip
We use the densities and Hermitian structure introduced above to give $\trb$ a Hermitian structure as follows. Fix once and for all a cut-off function $\omega_0$.  For $y\in \Y$ and  $u$, $v\in \trb_y$ set
\begin{equation*}
(u,v)_y=\int_{\Z_y}(\omega_0 u,\omega_0v)_{L^2(\Z_y)}\,\frac{dx}{x}.
\end{equation*}
The smoothness of this Hermitian structure is proved using the trivializations defined in \cite[Section 3]{GiKrMe10} for the Hilbert space bundle over $\Y$ whose fiber over $y$ is $L^2(\Z_y;E_{\Z_y})$ and the trivializations of $\trb$ obtained in \cite[Section 6]{KrMe12a}.

Let $A^\star$ denote the formal adjoint of $A$ with respect to $x^{-m/2}L^2_b$:
\begin{equation*}
(Au,v)_{x^{-m/2}L^2_b}=(u,A^\star v)_{x^{-m/2}L^2_b},\quad u\in C_c^\infty(\open\N^\wedge;E^\wedge),\ v\in C_c^\infty(\open\N^\wedge;F^\wedge).
\end{equation*}
It is easy to see that if $A=x^{-m}P$ with $P\in \Diff^m_b(\M;E,F)$, then $A^\star=x^{-m}P^\star$ with $P^\star\in \Diff^m_e(\M;F,E)$ the operator satisfying
\begin{equation*}
(Pu,v)_{L^2_b}=(u,P^\star v)_{L^2_b},\quad u\in C_c^\infty(\open\N^\wedge;E^\wedge),\ v\in C_c^\infty(\open\N^\wedge;F^\wedge);
\end{equation*}
observe that here the weight is $x^0$. The $w$-ellipticity of $A$ gives the $w$ ellipticity of $A^\star$, and the relation of the latter with $P^\star$ gives that $A^\star$ also satisfies \eqref{NobSpecOnCritLinesOrderm} with respect to $\gamma=m/2$. Indeed, the formula
\begin{equation*}
\bPhatstar(\sigma)=\bPhat(\overline\sigma)^\star
\end{equation*}
holds. Let then $\trb^\star\to\Y$ denote the trace bundle of $A^\star$ relative to the weight $x^m$.

Note that if $\omega$ is a cut-off function and $u\in \trb_y$, then 
\begin{equation}\label{InDmaxExplicit}
\omega u\in x^{-m/2}L^2_b(\Z_y^\wedge;E^\wedge_{\Z_y^\wedge})\text{ and }\bA_y(\omega u)\in x^{-m/2}L^2_b(\Z_y^\wedge;F^\wedge_{\Z_y^\wedge}),
\end{equation}
in other words, $\omega u\in \Dom_{\max}(\bA_y)$. Indeed, the first of the assertions in \eqref{InDmaxExplicit} is evident form the form of $u$ in \eqref{ProtoTrace} and the smoothness of the coefficients, and the second results from the fact that $\bA_y(\omega u)\in C_c^\infty(\open \Z^\wedge_y;E_{\Z^\wedge_y})$, the latter because $\bP_yu=0$ and $\omega=1$ near $0$. 

Note also that with $u$ and $\omega$ as in the previous paragraph, the Mellin transform of $\omega u$ (given by \eqref{MellinOmegaU}) is, as a $C^\infty(\Z_y)$-valued function, meromorphic in all of $\C$ with poles at $\spec_b(\bA_y)$ with singular part independent of $\omega$, and that 
\begin{equation*}
\bP_y(\omega u)\!\widehat{\phantom{P}}(\sigma)=\bPhat_y(\sigma)\widehat{\omega u}(\sigma)
\end{equation*}
is entire. We will write $\sing_\Omega\widehat u$ for the singular part of $\widehat{\omega u}$ in $\Omega$:
\begin{equation}\label{SingularPart}
\sing\widehat u(\sigma) =\frac{\im}{2\pi}\oint_{\partial\Omega}\frac{\widehat{\omega u}(\zeta)}{\zeta-\sigma}\,d\zeta
\end{equation}
for large $\sigma$ if $\widehat{\omega u}(\sigma)$ has no poles on $\partial\Omega$. The integral is computed with the counter-clockwise orientation.

\begin{proposition}\label{BasicParing}
Define $\beta:\trb\times\trb^\star\to\C$ as follows. Pick cut-off functions $\omega$, $\tilde \omega$ and let
\begin{equation*}
\beta_y(u,v) = [\omega u,\tilde\omega v]_{\bA_y},\ u \in \trb_y,\ v\in \trb^\star_y,\ y\in \Y
\end{equation*}
where
\begin{equation*}
[\omega u,\tilde\omega v]_{\bA_y}=\big(\bA_y(\omega u),\tilde\omega v\big)_{x^{-m/2}L^2_b}-\big(\omega u,\bA^\star_y(\tilde\omega v)\big)_{x^{-m/2}L^2_b}
\end{equation*}
Then $\beta$, which is in fact independent of the particular choice of cut-off functions, is a smooth nondegenerate sesquilinear pairing of $\trb$ and $\trb^\star$. The form $\beta$ establishes an isomorphism between $\trb^\star$ and the anti-dual of $\trb$ (and of $\trb$ and the antidual of $\trb^\star$).
\end{proposition}

\begin{proof}
That $[\omega u,\tilde\omega v]_{\bA_y}$ is independent of the choice of cut-off functions is well known and elementary: the Green form $[ \cdot,\cdot]_{\bA_y}$ vanishes if either argument belongs to the minimal domain. Also well known is the non-degeneracy of the Green form when viewed as acting on $\Dom_{\max}/\Dom_{\min}$. So only the assertion about smoothness needs to be addressed. 

To prove smoothness we exploit an idea in \cite{GiMe01}. Fix $y\in \Y$ and suppose $u\in \trb_y$, $v\in \trb^\star_y$. Plancherel's Theorem gives
\begin{align*}
\big(\bA_y(\omega u),\tilde\omega v&\big)_{x^{-m/2}L^2_b}\\
&= \frac{1}{2\pi}\int_{\set{\Im\sigma=m/2}}  \big(\bPhat_y(\sigma-m \im)\widehat {\omega u}(\sigma-m \im),\widehat {\tilde\omega v}(\sigma) \big)_{L^2(\Z_y)}\,d\sigma\\
&=\frac{1}{2\pi}\int_{-\infty}^\infty  \big(\widehat P(s-\im m/2)\widehat {\omega u}(s-m \im/2),\widehat {\tilde\omega v}(s+m \im/2) \big)_{L^2(\Z_y)}\,ds .
\end{align*}
as well as
\begin{equation*}
\big(\omega u,\bA_y^\star(\tilde\omega v)\big)_{x^{-m/2}L^2_b}
=\frac{1}{2\pi}\int_{\set{\Im\sigma=m/2}}\big(\widehat {\omega u}(\sigma),\bPhatstary(\sigma-m \im)\widehat {\tilde\omega v}(\sigma-m \im)\big)_{L^2(\Z_y)}\,d\sigma.
\end{equation*}
Using $\bPhatstary(\sigma)=\bPhat_y(\overline\sigma)^\star$ in the last integral we get
\begin{align*}
\big(\omega u,\bA_y^\star(\tilde\omega v&)\big)_{x^{-m/2}L^2_b}\\
&=\frac{1}{2\pi}\int_{\set{\Im\sigma=m/2}}\big(\widehat {\omega u}(\sigma),\bPhat_y(\overline \sigma+m \im)^\star\widehat {\tilde\omega v}(\sigma-m \im)\big)_{L^2(\Z_y)}\,d\sigma.
\\
&=\frac{1}{2\pi}\int_{\set{\Im\sigma=m/2}}\big(\bPhat_y(\overline \sigma+m \im) \widehat {\omega u}(\sigma),\widehat {\tilde\omega v}(\sigma-m \im)\big)_{L^2(\Z_y)}\,d\sigma.
\\
&=\frac{1}{2\pi} \int_{-\infty}^\infty\big(\bPhat_y(s+m \im/2) \widehat {\omega u}(s+m \im/2),\widehat {\tilde\omega v}(s-m \im/2)\big)_{L^2(\Z_y)}\,ds
\end{align*}
Thus
\begin{equation*}
[\omega u,\tilde\omega v]_{\bA_y}=\frac{1}{2\pi}\int_{\partial\Sigma} \big(\bPhat_y(\sigma)\widehat{\omega u},\widehat{\tilde\omega v}(\overline \sigma)\big)_{L^2(\Z_y)}\,d\sigma.
\end{equation*}
with $\Sigma=\set{\sigma\in \C:|\Im\sigma|<m/2}$ and the positive orientation for its boundary. Both
\begin{equation*}
\big(\bPhat_y(\sigma)\widehat{\omega u},\widehat{\tilde\omega v}(\overline \sigma)\big)_{L^2(\Z_y)}\text{ and }\big(\bPhat_y(\sigma)\widehat{u},\widehat{v}(\overline \sigma)\big)_{L^2(\Z_y)}
\end{equation*}
are meromorphic in $\sigma\in \C$ with poles at $\Sigma\cap \spec_b(\bA_y)$, and the difference is entire (recall that for instance $\widehat u$ is notation for the singular part of $\widehat{\omega u}$, see \eqref{SingularPart}). Therefore 
\begin{equation}\label{PairingViaMellin}
[\omega u,\tilde\omega v]_{\bA_y}=\frac{1}{2\pi}\int_{\partial R}\big(\bPhat_y(\sigma)\widehat{u},\widehat{v}(\overline \sigma)\big)_{L^2(\Z_y)}\,d\sigma
\end{equation}
where $R$ is an arbitrary rectangle containing $\spec_b(\bA_y)$. By definition of smoothness, if $u$ is a smooth section of $\trb$ and $v$ is one of $\trb^\star$  over some open $U\subset \Y$ such that $\spec_b(\bA_y)\cap \Sigma\subset R$ for $y\in U$, then the integrand in \eqref{PairingViaMellin} depends smoothly on $y\in U$, thus giving the smoothness of $y\mapsto \beta_y(u(y),v(y))$.
\end{proof}

\begin{proposition}\label{generatoradjoint}
The actions of $x\partial_x$ on $\trb$ and $\trb^\star$ are skew-adjoint to each other with respect to $\beta$.
\end{proposition}

\begin{proof}
Let $y\in \Y$, $u\in \trb_y$, $v\in \trb_y^\star$. Taking advantage of \eqref{PairingViaMellin} we have
\begin{align*}
\beta_y(x\partial_x u,v)
&=\frac{1}{2\pi}\int_{\partial R}\big(\bPhat_y(\sigma)(\im \sigma \widehat{u}),\widehat{v}(\overline \sigma)\big)_{L^2(\Z_y)}\,d\sigma\\
&=\frac{1}{2\pi}\int_{\partial R}\big(\bPhat_y(\sigma)( \widehat{u}),-\im \overline \sigma\widehat{v}(\overline \sigma)\big)_{L^2(\Z_y)}\,d\sigma\\
&=\beta_y(u,-x\partial_x v)
\end{align*}
\end{proof}

\section{The trace map for first order operators}\label{sec-TraceMap}

Henceforth we restrict our discussion to the case $m=1$ and weight $x$, i.e., $\gamma=1/2$; $A$ will always be an element of $x^{-1}\Diff^1(\M;E,F)$, and its minimal and maximal domains are subspaces of $x^{-1/2}L^2_b(\M;E)$. The operator $A$  will be assumed to be $w$-elliptic and to satisfy \eqref{NormalInjectivity} and \eqref{NobSpecOnCritLinesOrderm} with $\gamma=1/2$. The trace bundle of $A$ (relative to the weight $x$) continues to be denoted by $\trb$, that of $A^\star$ by $\trb^\star$. The defining function $x$ will remain fixed.

\begin{lemma}\label{AonNwedge}
For any $A\in x^{-1}\Diff^1_e(\M;E,F)$ there is a $w$-elliptic operator $A'\in x^{-1}\Diff^1_e(\N^\wedge;E^\wedge,F^\wedge)$ such that $A'_\wedge=A_\wedge$ and $\kappa_\varrho^{-1} (A'-\bA')\kappa_\varrho=0$.
\end{lemma}

\begin{proof}
Recalling that $\bA=A_\wedge(0)=x^{-1}\bP$, set $q(\pmb\eta)=A_\wedge(\pmb\eta)-\bA$. Thus $q$ is a (smooth) homomorphism $\wp_\wedge^*T^*\Y\to \Hom(E^\wedge, F^\wedge)$ which we may also view as a homomorphism
\begin{equation*}
q:\wp_\wedge^*T^*\Y\otimes E^\wedge\to F^\wedge.
\end{equation*}
Since $q$ is linear, formula \eqref{WedgeSymbolKappaHom} implies that $q$ commutes with $\kappa_\varrho$.

We will use this information to construct a suitable differential operator. The vector bundle $\wp_\wedge^*T^*\Y$ may be viewed, canonically, as a subbundle of $T^*\N^\wedge$ (the annihilator of the vertical tangent bundle). Let $\p:T^*\N^\wedge\to T^*\N^\wedge$ be a projection on $\wp_\wedge^*T^*\Y$ which commutes with the radial action and satisfies $\p(dx)=0$. Pick a connection on $E$, let
\begin{equation*}
\nabla:C^\infty(\N^\wedge;E^\wedge)\to C^\infty(\N^\wedge;T^*\N^\wedge\otimes E^\wedge)
\end{equation*}
denote the connection induced on $E^\wedge\to \N^\wedge$, which automatically commutes with the radial action: $\nabla\kappa_\varrho= \kappa_\varrho\nabla$. Define
\begin{equation}\label{DefQ}
Q_0=-\im q \circ (\p\otimes\Id)\circ \nabla.
\end{equation}
Then, if $g\in C^\infty(\Y)$,
\begin{equation*}
\varrho^{-1}\kappa^{-1}_\varrho e^{-\im \varrho \wp_\wedge^*g} Q_0 e^{\im \varrho \wp_\wedge^*g} \kappa_\varrho u=q\circ(\p\otimes\Id)(d\wp_\wedge^*g\otimes u)-\frac{\im}\varrho q \circ (\p\otimes\Id)\circ \nabla
\end{equation*}
from which it follows, see~\eqref{NormalFamilyAsLimit}, that the normal family of $Q_0$ is $q$. Then $A'_\wedge=\bA+Q_0$ satisfies the required condition since evidently $x Q_0\in \Diff^1_e(\N^\wedge;E^\wedge,F^\wedge)$. Note that $Q_0x=xQ_0$ because $\p(dx)=0$.
\end{proof}

To make the construction of $A'$ symmetric with respect to formal adjoints (which is how we'll use $A'$) we modify $Q_0$ by adding a zeroth order term. Denote by $Q_1$ the similarly defined operator for $A^\star$, i.e., \eqref{DefQ} with $q^\star(\pmb\eta)=A^\star_\wedge(\pmb\eta)-A^\star_\wedge(0)$ and a suitable connection on $F^\wedge$. Then, as before, the normal family of $\bA^\star+Q_1$ is $A_\wedge^\star(\pmb\eta)$. Let $Q_1^{\star}$ be the formal adjoint of
\begin{equation*}
Q_1:C_c^\infty(\open\N^\wedge;F)\subset L^2_b(\N^\wedge;F)\to L^2_b(\N^\wedge;E).
\end{equation*}
Lemma~\ref{AdjointOfNormalFamily} guarantees that the normal family of $A$ is again $\bA+Q_{1,\wedge}^\star$. Let then $Q=\frac{1}{2}(Q_0+Q_1^\star)$. Then
\begin{equation*}
(\bA+Q)_\wedge=A_\wedge
\end{equation*}
and $(\bA+Q)^\star = \bA^\star +Q^\star$ (of course) but now $Q$ and its formal adjoint $Q^\star$ have the same structure. 

\medskip
Suppose $u$ is a smooth section of $\trb$. Viewing $u$ as a section of $E^\wedge$ over $\open\N^\wedge$ (i.e., as $\mathfrak i u$ in the notation of Lemma~\ref{CanonicalMap}) we have, on the one hand,
\begin{equation*}
u\in x^{-1/2}L^2_{b,\loc}(\N^\wedge;E^\wedge)\text{ and }\bA u=0
\end{equation*}
(the latter by definition), and on the other, $Q u\in x^{-1/2}L^2_{b,\loc}(\N^\wedge;F^\wedge)$. Hence, if $\omega$ is a cut-off function on $\N^\wedge$, then $\omega u\in \Dom_{\max}(A')$.

\begin{lemma}\label{prePprime} The  map 
\begin{equation}\label{prePprimeOp}
C^\infty(\Y;\trb)\ni u\mapsto \omega u\in \Dom_{\max}(A')
\end{equation}
is continuous.
\end{lemma}

\begin{proof}
Lemma~\ref{CanonicalMap} gives that
\begin{equation}\label{prePIntoDist}
\omega \mathfrak i:C^\infty(\Y;\trb)\to C^{-\infty}(\N^\wedge;E^\wedge)
\end{equation}
is continuous. As already discussed, the range of $\omega\mathfrak i$ is contained in $\Dom_{\max}(A')$, which is continuously embedded in $C^{-\infty}(\open\N^\wedge;E^\wedge)$. Therefore, since \eqref{prePIntoDist} has closed graph, also \eqref{prePprimeOp} has closed closed graph, hence is continuous.
\end{proof}

Let $\Phi$ be the map used in \eqref{NormalFamilyAsLimit}, assume that $\omega$ is supported in the domain of $\Phi$. 

\begin{lemma}
If $u\in C^\infty(\Y;\trb)$ then $\preP(u)=\Phi_*(\omega u)\in \Dom_{\max}(A)$ and
\begin{equation*}
\preP:C^\infty(\Y;\trb)\to \Dom_{\max}(A)
\end{equation*}
is continuous.
\end{lemma}

We will show this with the aid of the operator
\begin{equation*}
\tilde A=\Phi_*A'\Phi^*\omega + A(1-\omega)\in x^{-1}\Diff^1_e(\M;E,F),
\end{equation*}
where we have used $\omega$ also for $\Phi_*\omega$. Then $A-\tilde A\in \Diff^1_e(\M;E,F)$, since $A_\wedge=\tilde A_\wedge$, and if the support of $\omega$ is close enough to $\N$, then $\tilde A$ is $w$-elliptic. Clearly, if $u\in C^\infty(\Y,\trb)$, then $\Phi_*(\omega u) \in \Dom_{\max}(\tilde A)$, so the previous lemma is a consequence of the next.

\begin{lemma}\label{APairingIsAprimePairing}
Let $A$, $\tilde A\in x^{-1}\Diff^1_e(\M;E,F)$ be $w$-elliptic and suppose $A-\tilde A\in \Diff^1_e(\M;E,F)$. Then $\Dom_{\min}(A)=\Dom_{\min}(\tilde A)$ and $\Dom_{\max}(A)=\Dom_{\max}(\tilde A)$. Furthermore the norms defined by $A$ and $\tilde A$ on $\Dom_{\max}(A)$ are equivalent, and
\begin{equation*}
[u,v]_A=[u,v]_{\tilde A} \text{ for all }u\in \Dom_{\max}(A),\ v\in \Dom_{\max}(A^\star).
\end{equation*}
\end{lemma}

\begin{proof}
Let $u\in \Dom_{\max}(A)$. Lemma~\ref{xmDmax} gives $xu\in x^{1/2}H^1_e(\M;E)$, consequently $u\in x^{-1/2}H^1_e(\M;F)$, thus $(A-\tilde A)u\in x^{-1/2}L^2_b(\M;F)$. Therefore $\tilde Au\in x^{-1/2}L^2_b(\M;E)$, hence $u\in \Dom_{\max}(\tilde A)$. Similarly $\Dom_{\max}(\tilde A)\subset \Dom_{\max}(A)$.

Since $A_\wedge(\pmb\eta) = \tilde A_\wedge(\pmb\eta)$, both \eqref{NormalInjectivity} and \eqref{NobSpecOnCritLinesOrderm} hold for $\tilde A$ iff they hold for $A$. Thus $\Dom_{\min}(\tilde A)=\Dom_{\min}(A)$, since by Theorem~\ref{DminTheorem} they are equal to $x^{1/2}H^1_e(\M;E)$.

Finally, if $u\in \Dom_{\max}(A)$ and $v\in \Dom_{\max}(A^\star)$, then
\begin{align*}
[u,v]_A
&=(Au,v)_{x^{-1/2}L^2_b}-(u,A^\star v)_{x^{-1/2}L^2_b}\\
&=[u,v]_{\tilde A}+\big((A-\tilde A)u,v\big)_{x^{-1/2}L^2_b}-\big(u,(A^\star - {\tilde A}^\star) v\big)_{x^{-1/2}L^2_b}
&=[u,v]_{\tilde A}
\end{align*}
because $A-\tilde A\in \Diff^1_e(\M;E,F)$ and $u,v\in x^{-1/2}H^1_e$.
\end{proof}

All the foregoing can also be applied to $A^\star$ if we assume the analog of \eqref{NormalInjectivity} for this operator:
\begin{equation*}
A_\wedge^\star(\pmb\eta) : \Dom_{\min}(A_\wedge^\star(\pmb\eta)) \subset x^{-\gamma}L^2_b \to x^{-\gamma}L^2_b\text{ is injective for all }\pmb\eta \in T^*\Y \minus 0.
\end{equation*}
Note that this is equivalent to \eqref{NormalSurjectivity} because the operators $A_\wedge(\pmb\eta)$ are Fredholm if $\pmb\eta\ne0$. We let 
\begin{equation*}
\preP^\star:C^\infty(\Y;\trb^\star)\to \Dom_{\max}(A^\star)
\end{equation*}
be the corresponding operator.

\medskip
We now define a map $\gamma_{A^\star}:\Dom_{\max}(A^\star)\to C^{-\infty}(\trb^\star)$ with the property that 
\begin{equation*}
\gamma_{A^\star}\circ \preP^*=\Id.
\end{equation*}
Fix $v\in \Dom_{\max}(A^\star)$. Then the functional $\lambda_v:C^\infty(\Y;\trb)\to\C$ given by 
\begin{equation*}
C^\infty(\Y;\trb)\ni u\mapsto \langle\lambda_v,u\rangle=[\preP u,v]_A\in \C
\end{equation*}
is continuous. Indeed, for general $u\in \Dom_{\max}(A)$ and $v\in \Dom_{\max}(A^\star)$ we have
\begin{equation*}
\big|[u,v]_A\big|\leq\|Au\|_{x^{-1/2}L^2_b}\|v\|_{x^{-1/2}L^2_b}+\|u\|_{x^{-1/2}L^2_b}\|A^\star v\|_{x^{-1/2}L^2_b}\leq 2\|u\|_A\|v\|_{A^\star}
\end{equation*}
so
\begin{equation*}
|[\preP u,v]_A|\leq 2\|\preP u\|_A\|v\|_{A^\star}\leq C\|u\|_{C^1}\|v\|_{A^\star}
\end{equation*}
Thus $\lambda_v$ is a generalized section of $\trb^*\otimes |\Wedge|\Y$. Here $\trb^*$ is the abstract dual bundle of $\trb$ (not $\trb^\star$) and $|\Wedge|\Y$ is the bundle of densities of $\Y$. Trivializing the density bundle using $\m_\Y$ we regard $\lambda_v$ as a generalized section of $\trb^*$. The latter with the opposite complex structure (that is, $\zeta\cdot w=\overline \zeta w$, $\zeta\in \C$, $w\in \trb^*$) is isomorphic to $\trb^\star$ via the pairing $\beta$ defined in Proposition~\ref{BasicParing}. So $\lambda_v$ may be regarded as a generalized section of $\trb^\star$. We write $\gamma_{A_\star}(v)$ for this section of $\trb^\star$. The map $ v\mapsto \gamma_{A^\star}(v)$ is linear (not anti-linear).

Suppose now that $v\in C^\infty(\Y;\trb^\star)$. Then
\begin{equation*}
\langle \lambda_{\preP^\star v},u\rangle = [\preP u,\preP^\star v]_{A}=[\preP u,\preP^\star v]_{A'},
\end{equation*}
the last equality due to Lemma~\ref{APairingIsAprimePairing}. So $\langle \lambda_{\preP^\star v},u\rangle$ is equal to  
\begin{multline*}
\int_{\N^\wedge}\big((\bA+Q)(\omega u),\omega v\big)_{F^\wedge}x\,d\m^\wedge_b-\int_{\N^\wedge}\big(\omega u,(\bA^\star+Q^\star)(\omega v)\big)_{E^\wedge}x\,d\m^\wedge_b\\
=\int_\Y \beta_y(u_y,v_y)\,d\m_\Y(y)+\int_{\N^\wedge}\Big[\big(Q(\omega u),\omega v\big)_{F^\wedge}-\big(u,Q^\star(\omega v)\big)_{E^\wedge}\Big]x\,d\m^\wedge_b.
\end{multline*}
The integral involving $Q$ vanishes because $xQ$ is an edge operator and  $xQ^\star$ is its formal adjoint, since $Q^\star$ is the formal adjoint of $Q$ and $xQ=Qx$. Thus
\begin{equation*}
\langle \lambda_{\preP^\star v},u\rangle=\int_\Y\beta_y(u_y,v_y)\,d\m_\Y\quad \text{for all }u\in C^\infty(\Y;\trb),
\end{equation*}
hence
\begin{equation*}
\gamma_{A^\star} (\preP^\star v)=v
\end{equation*}
for $v\in C^\infty(\Y;\trb^\star)$. The analogously defined map for $A$ is denoted $\gamma_A$. It of course satisfies
\begin{equation}\label{GammaInvertsPreP}
\gamma_A \circ \preP =\Id \text{ on }C^\infty(\Y;\trb)
\end{equation}

\begin{proposition}\label{DminSubsetKerTrace}
$\Dom_{\min}(A)\subset \ker \gamma_A$.
\end{proposition}

\begin{proof}
We prove $\Dom_{\min}(A^\star)\subset \ker \gamma_{A^\star}$. Suppose $v\in \Dom_{\min}(A^\star)$. Since the Hilbert space adjoint of $A$ with its maximal domain is $A^\star$ with its minimal domain, 
\begin{equation*}
\langle\lambda_{v},u\rangle = [\preP u,v]_A = 0
\end{equation*}
for every $u\in C^\infty(\Y;\trb)$ because $\preP u \in \Dom_{\max}(A)$. Consequently $\lambda_{v}=0$ which just means that $\gamma_{A^\star}(v)=0$. 
\end{proof}

\begin{definition}\label{gammaA}
The map $\gamma_A:\Dom_{\max}(A)\to C^{-\infty}(\Y;\trb)$ is the trace map of $A$. 
\end{definition}

The following theorem is a more precise result about the regularity of the traces of elements in $\Dom_{\max}(A)$. Its statement makes use of the Sobolev spaces of sections of of variable smoothness associated with a vector bundle with a given bundle endomorphism; these spaces were introduced \cite{KrMe12b}.

\begin{theorem}\label{RangeofGamma}
The trace map is a continuous operator
\begin{equation*}
\gamma_A : \Dom_{\max}(A) \to H^{-\gen}(\Y;\trb),
\end{equation*}
where $\gen = x\partial_x +1/2 \in C^\infty(\Y;\End(\trb))$.
\end{theorem}

As noted in the introduction the theorem reduces to the classical assertion if $x\partial_x$ acts trivially on $\trb$, that is, if $\trb_y$ is, for each $y$, a space of sections of $E$ over $\Z_y$ independent of $x$, since in this case the image space is just the standard Sobolev space $H^{-1/2}(\Y;\trb)$.

\begin{proof}
We prove that $\gamma_{A^{\star}} : \Dom_{\max}(A^{\star}) \to H^{-\gen}(\Y;\trb^{\star})$. In the proof we make use of the extension operator $\Ext : C^{\infty}(\Y;\trb) \to C^{\infty}(\open\M;E)$ constructed in Section~\ref{sec-Extension}.

Let $v \in \Dom_{\max}(A^{\star})$, and let $u \in C^{\infty}(\Y;\trb)$ be arbitrary. Then
\begin{equation*}
\langle \lambda_v,u \rangle = [\preP u,v]_A = [\Ext u,v]_A
\end{equation*}
because $\preP u - \Ext u \in \dot{C}^{\infty}(\M;E) \subset \Dom_{\min}(A)$ by Proposition~\ref{PoissonOperatorProperties}. Consequently,
\begin{equation*}
|\langle \lambda_v,u \rangle| = |[\Ext u,v]_A| \leq 2\|\Ext u\|_A\|v\|_{A^{\star}} \leq 2 \|v\|_{A^{\star}}\| \Ext\| \|u\|_{H^{1-\gen}(\Y;\trb)}.
\end{equation*}
Here we made use of the continuity of
\begin{equation*}
\Ext : H^{1-\gen}(\Y;\trb) \to \Dom_{\max}(A),
\end{equation*}
see Proposition~\ref{PoissonOperatorProperties}. This shows that $\lambda_v$ extends to a continuous linear functional on $H^{1-\gen}(\Y;\trb)$.

Let $\beta : \trb \times \trb^{\star} \to \C$ be the pairing defined in Proposition~\ref{BasicParing}. The sesquilinear form $\{\cdot,\cdot\} : C^{\infty}(\Y;\trb)\times C^{\infty}(\Y;\trb^{\star}) \to \C$ given by
\begin{equation*}
\{u_1,u_2\} = \int_{\Y}\beta_y(u_1(y),u_2(y))\,d\m_{\Y}.
\end{equation*}
extends by continuity to
\begin{equation*}
\{\cdot,\cdot\} : C^{\infty}(\Y;\trb) \times C^{-\infty}(\Y;\trb^{\star}) \to \C,
\end{equation*}
and by definition of the trace map we have $\langle \lambda_v,u \rangle = \{u,\gamma_{A^{\star}}(v)\}$. From \cite[Theorem~7.4]{KrMe12b} we get that $\{\cdot,\cdot\}$ extends to
\begin{equation*}
\{\cdot,\cdot\} : H^{1-\gen}(\Y;\trb) \times H^{-1+\gen^{\sharp}}(\Y;\trb^{\star}) \to \C,
\end{equation*}
and $H^{1-\gen}(\Y;\trb)' \cong H^{-1+\gen^{\sharp}}(\Y;\trb^{\star})$ via $\{\cdot,\cdot\}$, where
$\gen^{\sharp} \in C^\infty(\Y;\End(\trb^{\star}))$ is the fiberwise adjoint of $\gen$ with respect to the pairing $\beta : \trb \times \trb^{\star} \to \C$. This shows that $\gamma_{A^{\star}}(v) \in H^{-1+\gen^{\sharp}}(\Y;\trb^{\star})$. By Proposition~\ref{generatoradjoint} we have
\begin{equation*}
\gen^{\sharp} = \frac{1}{2} + [x\partial_x]^{\sharp} = \frac{1}{2} - x\partial_x,
\end{equation*}
and consequently $\gamma_{A^{\star}} : \Dom_{\max}(A^{\star}) \to H^{-\gen}(\Y;\trb^{\star})$. The continuity of this map follows from the continuity of $\gamma_{A^{\star}} : \Dom_{\max}(A^{\star}) \to C^{-\infty}(\Y;\trb^{\star})$ and the closed graph theorem.
\end{proof}

In case of a regular elliptic operator $A \in \Diff^1(\M;E,F)$ acting in $L^2$, Theorem~\ref{RangeofGamma} specializes to the familiar result that taking traces on the boundary gives rise to a map $\Dom_{\max}(A) \to H^{-1/2}(\N;E)$.

\medskip
The above construction of the trace map can be applied to each element $A_\wedge(\pmb\eta)$ of the normal family of $A$. The trace bundle of this operator ($\pmb\eta$ fixed) is just $\trb_y$ with $y=\pi_\Y(\pmb\eta)$ and so we have a continuous map
\begin{equation}\label{TraceMapAWedge}
\gamma_{A_\wedge(\pmb\eta)}:\Dom_{\max}(A_\wedge(\pmb\eta))\to\trb_y.
\end{equation}

\begin{lemma}\label{SmoothnessOfTraceWedge}
Let $V\subset T^*\Y$ be open. If $u(\pmb\eta)\in \Dom_{\max}(A_\wedge(\pmb\eta))$, $\pmb\eta\in V$, is smooth as a section of $E^\wedge$ over $\open\wp_{\wedge,\Y}^{-1}(V)$, then $\pmb\eta\mapsto \gamma_{A_\wedge(\pmb\eta)}(u(\pmb\eta))$ is a smooth local section of $\pi_\Y^*\trb$.
\end{lemma}

The map $\wp_{\wedge,\Y}:\wp_\wedge^*T^*\Y\to T^*\Y$ was defined in \eqref{MicrolocalBdyFiberBundle}, $\open\wp_\wedge$ is its restriction to $\open \N^\wedge$, likewise $\open\wp_{\wedge,\Y}$. By definition, see \eqref{Smoothness}, $\pmb\eta\mapsto \tau(\pmb\eta)$ is a smooth section of $\pi_\Y^*\trb$ over $V\subset T^*\Y$ if it smooth as a section of $E^\wedge$ over $\open\wp_{\wedge,\Y}^{-1}(V)$.

\begin{proof}
The singular part, $\Ss_\Omega[ (\omega u(\pmb\eta))\widehat{\ }(\sigma)]$, of the Mellin transform of $\omega u(\pmb\eta)$ in $\Sigma=\set{\sigma:|\Im\sigma|<1/2}$, see \eqref{SingularPart}, is smooth in the complement in $V\times\C$ of
\begin{equation*}
\set{(\pmb\eta,\sigma)\in V\times\C:\sigma\in\spec_b(A_\wedge(\pmb\eta))\cap\Sigma}.
\end{equation*}
Here $\Omega$ is a rectangle with closure contained in $\Sigma$ and containing the points of $\spec_b(A_\wedge(\pmb\eta))\cap \Sigma$ when $\eta\in V$. (Since $\Y$ is compact, such a rectangle exists.) As a consequence,
\begin{equation*}
\tau(\pmb\eta)=\frac{1}{2\pi}\int_{\partial\Omega} x^{\im\sigma} \Ss_\Omega[ (\omega u(\pmb\eta))\widehat{\ }(\sigma)]d\sigma
\end{equation*}
(with the positive orientation for $\partial\Omega$) is a smooth of $E^\wedge$ over $\open\wp_{\wedge,\Y}^{-1}V$. Since in addition $\bA_{\wedge,y}\tau(\pmb\eta)=0$, $y=\pi_\Y(\pmb\eta)$, $\tau(\pmb\eta)$ is a smooth section of $\trb$. However, $\tau(\pmb\eta)$ is precisely $\gamma_{A_\wedge(\pmb\eta)}(u(\pmb\eta))$, so the latter is smooth.
\end{proof}

The kernel of $\gamma_{A_\wedge(\pmb\eta)}$ is $\Dom_{\min}(A_\wedge(\pmb\eta))\subset x^{-1/2}L^2_b(\Z^\wedge_y;E_{\Z^\wedge_y})$, $y=\pi_\Y(\pmb\eta)$. As discussed in Section~\ref{sec-Domains}, in the presence of \eqref{NormalInjectivity} these spaces form a trivial subbundle $\Dom_{\min}(A_\wedge)\embed\Ha_E$ with typical fiber
\begin{equation*}
\K^{1,1/2}_{1/2-1/2\dim Z^\wedge}(\Z^\wedge;E_{\Z^\wedge}).
\end{equation*}
This is \eqref{BundleOfMinimalDomains} with $m=1$, $\gamma=1/2$. 

\section{The kernel bundle of the normal family}\label{sec-KernelBundle}

In addition to the assumptions listed in the first paragraph of Section \ref{sec-TraceMap} we assume in this section the condition dual to \eqref{NormalInjectivity}: 
\begin{equation}\label{NormalSurjectivity}
A_\wedge(\pmb\eta) : \Dom_{\max}(A_\wedge(\pmb\eta)) \subset x^{-\gamma}L^2_b \to x^{-\gamma}L^2_b\text{ is surjective for all }\pmb\eta \in T^*\Y \minus 0
\end{equation}
with $\gamma=1/2$ (this condition will not be needed again until Section \ref{sec-BVP}).

Since $A$ is $w$-elliptic, $A_\wedge(\pmb\eta)$ is elliptic for each $\pmb\eta\in T^*\Y$. Because of \eqref{NormalInjectivity},
\begin{equation*}
\Dom_{\min}(A_\wedge(\pmb\eta))=\K^{1,1/2}_{1/2-\frac{1}{2}\dim\Z_y^\wedge}(\Z_y^\wedge;E_{\Z_y^\wedge}),\quad\pmb\eta\in T^*\Y;
\end{equation*}
this is \eqref{DomMinAWedge} again with $m=1$ and $\gamma=1/2$. Assume $\pmb\eta\ne 0$. The operator $A_\wedge(\pmb\eta)$ is Fredholm with its minimal domain (see \cite[Proposition 4.1, Part ({\it iii})]{GiKrMe10}) hence also with its maximal domain which happens to be
\begin{equation*}
\Dom_{\max}(A_\wedge(\pmb\eta))=\omega \mathfrak i \trb_y +\Dom_{\min}(A_\wedge(\pmb\eta)),\quad y=\pi_\Y(\pmb\eta)
\end{equation*}
(see Lemma \ref{CanonicalMap} for the definition of $\mathfrak i$) for any cut-off function $\omega$; this is obviously a direct sum. Thus
\begin{equation*}
\dim \big(\Dom_{\max}(A_\wedge(\pmb\eta))/\Dom_{\min}(A_\wedge(\pmb\eta))\big)=\dim \trb_y\text{ if }\pmb\eta\in T^*_y\Y. 
\end{equation*}
This dimension is also given by
\begin{equation*}
N=-\Ind(A_{\wedge,\min}(\pmb\eta))+\Ind(A_{\wedge,\max}(\pmb\eta)).
\end{equation*}
We will write $N''=-\Ind (A_{\wedge,\min}(\pmb\eta))$ and $N'=\Ind(A_{\wedge,\max}(\pmb\eta))$. These numbers may depend on the connected component of $T^*\Y\minus 0$ containing $\pmb\eta$, but we will not indicate this in the notation.

It is convenient at this point to give
\begin{equation*}
\Dom_{\max}(A_\wedge)=\bigsqcup_{\pmb\eta\in T^*\Y}\Dom_{\max}(A_\wedge(\pmb\eta))
\end{equation*}
the structure of a Hilbert space bundle. We do this by taking advantage of the fact that $\Dom_{\min}(A_\wedge)$ has already been given such a structure, and specifying that 
\begin{equation*}
\omega\mathfrak i\pi_\Y^*\trb+\Dom_{\max}(A_\wedge)
\end{equation*}
is a smooth direct sum decomposition of $\Dom_{\max}(A_\wedge)$. 

The collection of spaces 
\begin{equation*}
\KK_{\pmb\eta}=\ker(A_{\wedge,\max}(\pmb\eta)),
\end{equation*}
which by \eqref{NormalSurjectivity} have dimension $N'$, plays a fundamental role in the properties of boundary value problems involving $A$. 

\begin{theorem}\label{KernelBundle}
Suppose \eqref{NormalInjectivity}, \eqref{NobSpecOnCritLinesOrderm} and \eqref{NormalSurjectivity} hold with $\gamma=1/2$. Let 
\begin{equation*}
\KK=\bigsqcup_{\eta\in T^*\Y\minus0}\KK_{\pmb\eta},
\end{equation*}
let $\quad \pi_{\KK}:\KK\to T^*\Y$ be the canonical projection. Then $\KK\to T^*\Y\minus 0$ has the structure of a smooth vector bundle, isomorphic to a subbundle $\kerb\subset \pi_\Y^*\trb$ via the map
\begin{equation}\label{TraceOnKerenelBundle}
\gamma_{A_\wedge}\big|_{\KK}:\KK\to\pi_\Y^*\trb
\end{equation}
which is a smooth vector bundle morphism. The smooth sections of $\KK$ are those which are smooth as sections of $E^\wedge$ over $\N^\wedge$.
\end{theorem}

The proof will make use of the following lemma.

\begin{lemma}\label{SmoothnessOfProjection}
For every $\pmb\eta_0\in T^*\Y\minus 0$ there is a neighborhood $V\subset T^*\Y\minus 0$ and a family of projections
\begin{equation*}
\mathfrak p_{\pmb\eta}:\Dom_{\max}(A_\wedge(\pmb\eta))\to \Dom_{\max}(A_\wedge(\pmb\eta))
\end{equation*}
onto $\KK_{\pmb\eta}$, $\pmb\eta\in V$, with the property that if $u(\pmb\eta)\in x^{-1/2}L^2_b(\Z^\wedge_{y};E_{\Z^\wedge_{y}})$, $\pmb\eta \in V$, $y=\pi_\Y(\pmb\eta)$, is smooth as a section of $E^\wedge$ over $\open\wp_{\wedge,\Y}^{-1}(V)$ then so is $\mathfrak p_{\pmb\eta}(u(\pmb\eta))$.
\end{lemma}

We prove the lemma below.

\begin{proof}[Proof of Theorem~\ref{KernelBundle}]
Pick $\pmb\eta_0\in T^*\Y\minus 0$, a neighborhood $V\subset T^*\Y\minus 0$ of $\pmb\eta_0$, and a family of projections $\mathfrak p_{\pmb\eta}$ as in the lemma. Pick a frame $\tau_1,\dotsc,\tau_{N'}$ of $\trb$ near $y_0=\pi_\Y\pmb\eta_0$. Shrinking $V$ if necessary, the lifting of the frame gives a frame of $\pi_{\Y}^*\trb$ over $V$. After redefining the elements of the frame if necessary, assume that $\tau_1(y_0),\dotsc,\tau_{N'}(y_0)$ span $\gamma_{A_\wedge(\pmb\eta_0)}(\KK_{\pmb\eta_0})$. These elements are images by $\gamma_{A_\wedge(\pmb\eta_0)}$ of (unique) elements $\mathfrak n_\mu(\pmb\eta_0)=\omega\mathfrak i\tau_\mu(\pmb\eta_0)+\zeta_\mu(\pmb\eta_0)\in \KK_{\pmb\eta_0}$ where $\zeta_\mu(\pmb\eta_0)\in \Dom_{\min}(A_\wedge(\pmb\eta_0))$. Extend the elements $\zeta_\mu(\pmb\eta_0)$ as smooth sections of $\Dom_{\min}(A_\wedge)$ (see the end of Section~\ref{sec-TraceMap}) in such a way that they are smooth as sections of $E^\wedge$ (for instance by extending as constant with respect to one of the trivializations described in Section~\ref{sec-Domains}). Let 
\begin{multline*}
\mathfrak n_\mu(\pmb\eta)=\mathfrak p_{\pmb\eta}(\omega \mathfrak i\tau_\mu(y)+\zeta_\mu(\pmb\eta)),\
\tilde \tau_\mu(\pmb\eta)=\gamma_{A_\wedge(\pmb\eta)}(\mathfrak n_\mu(\pmb\eta)),\\
y=\pi_\Y(\pmb\eta),\ \pmb\eta\in V,\ \mu=1,\dots,N'.
\end{multline*}
Thus $\mathfrak n_\mu(\pmb\eta)\in \KK_{\pmb\eta}$ for each $\pmb\eta\in V$ and is, by Lemma~\ref{SmoothnessOfProjection}, a smooth section of $E^\wedge$ over $\open\wp_{\wedge,\Y}^{-1}(V)$. By continuity these elements are linearly independent for $\pmb\eta$ in a neighborhood of $\pmb\eta_0$ which we may take to be $V$ itself. Therefore they give a basis of $\KK_{\pmb\eta}$ for each $\pmb\eta \in V$. Furthermore, by Lemma~\ref{SmoothnessOfTraceWedge} the $\tilde \tau_\mu(\pmb\eta)$ are smooth sections of $\pi_\Y^*\trb$ near $\pmb\eta_0$, in $\kerb$, independent there since they agree with the $\tau_\mu$, $\mu=1,\dotsc,N'$, at $\pmb\eta_0$. This implies: every section of $\KK$ over $V$ which is smooth as a section of $E^\wedge$ over $\open\wp_{\wedge,\Y}^{-1}(V)$ is a linear combination of the $\mathfrak n_\mu(\pmb\eta)$ with $C^\infty$ coefficients.

This shows that the family of frames constructed as the frame $\mathfrak n_\mu$ was has smooth transition functions. This defines the $C^\infty$ structure of $\KK$. 

Since $\gamma_{A_\wedge(\pmb\eta)}(\mathfrak n_\mu(\pmb\eta))$ is part of a smooth frame of $\pi_\Y^*\trb$, \eqref{TraceOnKerenelBundle} is a smooth morphism. 
This concludes the proof of Theorem~\ref{KernelBundle}.
\end{proof}

\begin{proof}[Proof of Lemma~\ref{SmoothnessOfProjection}]
Fix $\pmb\eta_0\in T^*\Y\minus 0$. We will construct a family of right inverses 
\begin{equation*}
B_\wedge(\pmb\eta):x^{-1/2}L^2_b\to \Dom_{\max}(A_\wedge(\pmb\eta))
\end{equation*}
defined for $\pmb\eta$ in a neighborhood $V$ of $\pmb\eta_0$ with the property that it maps sections of $\pi_\Y^*\Ha_F$ over $V$ which are smooth as sections of $F^\wedge$ over $\open\wp_{\wedge,\Y}^{-1}V$ to sections of $\Dom_{\max}(A_\wedge)$ over $V$ which are smooth as sections of $E^\wedge$ over the same set. Then $B_\wedge(\pmb\eta)A_\wedge(\pmb\eta)$, which is a projection
\begin{equation*}
x^{-1/2}L^2_b(\Z_{\pmb\eta}^\wedge;E_{\Z_{\pmb\eta}^\wedge})\to x^{-1/2}L^2_b(\Z_{\pmb\eta}^\wedge;E_{\Z_{\pmb\eta}^\wedge})
\end{equation*}
with kernel $\KK_{\pmb\eta}$ will have the analogous property (since $A_\wedge$ has it), and so $\mathfrak p_{\pmb\eta}=\Id-B_\wedge(\pmb\eta)A_\wedge(\pmb\eta)$ will have the required property. We are using the notation $\Z^\wedge_{\pmb\eta}$ for $\wp_{\wedge,\Y}^{-1}(\pmb\eta)$. 

Let $B_\wedge(\pmb\eta_0)$ be a right inverse of $A_\wedge(\pmb\eta_0)$. Let
\begin{equation*}
\pi_{\max},\ \pi_{\min}:\Dom_{\max}(A_\wedge(\pmb\eta_0))\to \Dom_{\max}(A_\wedge(\pmb\eta_0))
\end{equation*}
be the projections associated with the direct sum decomposition 
\begin{equation*}
\Dom_{\max}(A_\wedge(\pmb\eta_0))=\omega \mathfrak i \trb_{y_0} +\Dom_{\min}(A_\wedge(\pmb\eta_0))
\end{equation*}
with respect to some fixed cut-off function $\omega$; we have set $y_0=\pi_\Y(\pmb\eta_0)$. Obviously
\begin{equation*}
\pi_{\max}\circ B_\wedge(\pmb\eta_0),\ \pi_{\min}\circ B_\wedge(\pmb\eta_0):x^{-1/2}L^2_b(\Z^\wedge_{y_0};E_{\Z^\wedge_{y_0}})\to \Dom_{\max}(A_\wedge(\pmb\eta_0))
\end{equation*}
are continuous operators.

The map $\pi_{\max}$ can be written explicitly with the aid of the trace map (see \eqref{TraceMapAWedge}) of $A_\wedge(\pmb\eta_0)$:
\begin{equation*}
\pi_{\max}(u)=\omega\mathfrak i\gamma_{A_\wedge(\pmb\eta_0)}(u).
\end{equation*}
Let $\tau_1,\dots,\tau_N$ be a frame of $\pi_\Y^*\trb$ in a neighborhood $V$ of  $\pmb\eta_0$ and $\tau^1,\dotsc,\tau^N$ the dual frame. Then 
\begin{equation*}
x^{-1/2}L^2_b(\Z^\wedge_{\pmb\eta_0};F_{\Z^\wedge_{\pmb\eta_0}})\ni f\mapsto b^\mu(f)=\langle \tau^\mu(\pmb\eta_0),(\gamma_{A_\wedge(\pmb\eta_0)}\circ B_\wedge(\pmb\eta_0))(f)\rangle\in \C
\end{equation*}
is continuous.

Using the trivializations of $\Ha_E$ and $\Ha_F$ described in Section~\ref{sec-Domains} and taking advantage of the fact that the former restrict to trivializations of $\Dom_{\min}(A_\wedge)$ we regard $\pi_{\min}\circ B_\wedge(\pmb\eta_0)$ as a mapping $\Ha_F \to \Dom_{\min}(A_\wedge)$. As such it is continuous and has the property that it sends sections of $\Ha_F$ which are smooth section of $F^\wedge$ over $\open\wp_{\wedge,\Y}^{-1}V$ to sections of $E^\wedge$ with the same property. The map
\begin{equation*}
\tilde B_\wedge(\pmb\eta):x^{-1/2}L^2_b(\Z^\wedge_{\pmb\eta_0};F_{\Z^\wedge_{\pmb\eta_0}})\to x^{-1/2}L^2_b(\Z^\wedge_{\pmb\eta_0};E_{\Z^\wedge_{\pmb\eta_0}})
\end{equation*}
defined by
\begin{equation*}
\tilde B_\wedge(\pmb\eta)f=\pi_{\min}\circ B_\wedge(\pmb\eta_0)f+\sum_\mu b^\mu(f)\omega\tau_\mu(\pmb\eta),\quad \pmb\eta\in V
\end{equation*}
then has the same property. Note that $\tilde B_\wedge(\pmb\eta)$ maps into $\Dom_{\max}(A_\wedge(\pmb\eta))$ and that $\tilde B_\wedge(\pmb\eta_0)=B_\wedge(\pmb\eta_0)$. Recalling that $\bA_\wedge=\bA=A_\wedge(0)=x^{-1}\bP$, $0\in T^*_{\pi_\Y(\pmb\eta)}\Y$, set
\begin{equation*}
q(\pmb\eta)=A_\wedge(\pmb\eta)-\bA.
\end{equation*}
This is a (smooth) homomorphism $q:\wp_\wedge^*T^*\Y\to \Hom(E^\wedge, F^\wedge)$ which commutes with $x\partial_x$. We will write $\bA_{\wedge,\pmb\eta}$ when referring to $\bA_\wedge$ as an operator on $\Z^\wedge_{\pmb\eta}$. Since $\bA_{\wedge,\pmb\eta}$ is a first order operator,
\begin{equation*}
\bA_{\wedge,\pmb\eta}(h u)=h\bA_{\wedge,\pmb\eta}(u)-\im\sym(\bA_{\wedge,\pmb\eta})(dh)(u)
\end{equation*}
for any smooth function $h$ and generalized section $u$ of $E^\wedge$, where $\sym(\bA_{\wedge,\pmb\eta})$ is the standard principal symbol of $\bA_{\wedge,\pmb\eta}$. Passing to trivializations and working on the fiber over $\pmb\eta_0$ (so that differences of elements originally in different fibers make sense) we have
\begin{align*}
\bA_{\wedge,\pmb\eta}\big(\sum_\mu &b^\mu(f)\omega\tau_\mu(\pmb\eta)\big)\\
&=\sum_\mu b^\mu(f)\omega \bA_{\wedge,\pmb\eta} (\tau_\mu(\pmb\eta))-\im \sum_\mu b^\mu(f)\sym(\bA_{\wedge,\pmb\eta})(d\omega)(\tau_\mu(\pmb\eta))\\
&=-\im \sum_\mu b^\mu(f)\sym(\bA_{\wedge,\pmb\eta})(d\omega)(\tau_\mu(\pmb\eta))\\
&=\bA_{\wedge,\pmb\eta_0}\big(\sum_\mu b^\mu(f)\omega\tau_\mu(\pmb\eta_0)\big)\\
&\qquad-\im\sum_\mu b^\mu(f)\Big(\sym(\bA_{\wedge,\pmb\eta})(d\omega)(\tau_\mu(\pmb\eta))-\sym(\bA_{\wedge,\pmb\eta_0})(d\omega)(\tau_\mu(\pmb\eta_0))\Big)
\end{align*}
taking advantage of the fact that $\bA_{\wedge,\pmb\eta}(\tau_\mu(\pmb\eta))=0$. Proceeding in a similar fashion with all terms we obtain
\begin{multline*}
A_\wedge(\pmb\eta)\tilde B_\wedge(\pmb\eta)(f)= A_\wedge(\pmb\eta_0) B_\wedge(\pmb\eta_0)(f)-(A_\wedge(\pmb\eta_0)-A_\wedge(\pmb\eta))\pi_{\min}B_\wedge(\pmb\eta_0)(f)\\
-\im\sum_\mu b^\mu(f)\Big(\sym(\bA_{\wedge,\pmb\eta})(d\omega)(\tau_\mu(\pmb\eta))-\sym(\bA_{\wedge,\pmb\eta_0})(d\omega)(\tau_\mu(\pmb\eta_0))\Big)\\
+\sum_\mu b^\mu(f) \omega\big(q(\pmb\eta)\tau_\mu(\pmb\eta)-q(\pmb\eta_0)\tau_\mu(\pmb\eta_0)\big),
\end{multline*}
that is,
\begin{equation*}
A_\wedge(\pmb\eta)\tilde B_\wedge(\pmb\eta)=\Id -R(\pmb\eta,\pmb\eta_0)
\end{equation*}
with
\begin{equation*}
R(\pmb\eta,\pmb\eta_0):x^{-1/2}L^2_b(\Z^\wedge_{\pmb\eta_0};F_{\Z^\wedge_{\pmb\eta_0}})\to x^{-1/2}L^2_b(\Z^\wedge_{\pmb\eta_0};F_{\Z^\wedge_{\pmb\eta_0}})
\end{equation*}
which evidently is continuous. We claim that its norm can be made arbitrarily small when $\pmb\eta$ is close to $\pmb\eta_0$. Indeed, since $\sym(\bA_{\wedge,y})(d\omega)$ is a smooth bundle homomorphism depending smoothly on the base variables, 
\begin{equation*}
\sym(\bA_{\wedge,\pmb\eta})(d\omega)(\tau_\mu(\pmb\eta))-\sym(\bA_{\wedge,\pmb\eta_0})(d\omega)(\tau_\mu(\pmb\eta_0))
\end{equation*}
can be made arbitrarily small by taking $\pmb\eta$ sufficiently close to $\pmb\eta_0$. The same argument applies to 
\begin{equation*}
q(\pmb\eta)\tau_\mu(\pmb\eta)-q(\pmb\eta_0)\tau_\mu(\pmb\eta_0).
\end{equation*}
Finally, $A_\wedge(\pmb\eta_0)-A_\wedge(\pmb\eta)$ is a first order operator whose coefficients can be assumed to be arbitrarily small by taking $\pmb\eta$ sufficiently close to $\pmb\eta_0$. So this operator, as an operator $\Dom_{\min}(A_\wedge(\pmb\eta_0))\to x^{-1/2}L^2_b(\Z^\wedge_{\pmb\eta_0};F_{\Z^\wedge_{\pmb\eta_0}})$, has small norm if $\pmb\eta$ is sufficiently close to $\pmb\eta_0$, and the same is then true for 
\begin{equation*}
(A_\wedge(\pmb\eta_0)-A_\wedge(\pmb\eta))\pi_{\min}B_\wedge(\pmb\eta_0):x^{-1/2}L^2_b(\Z^\wedge_{\pmb\eta_0};F_{\Z^\wedge_{\pmb\eta_0}})\to x^{-1/2}L^2_b(\Z^\wedge_{\pmb\eta_0};F_{\Z^\wedge_{\pmb\eta_0}}).
\end{equation*}
Thus we may assume that $\|R(\pmb\eta,\pmb\eta_0)\|_{x^{-1/2}L^2_b}<1$ when $\pmb\eta\in V$ and define
\begin{equation*}
B_\wedge(\pmb\eta)=\tilde B_\wedge(\pmb\eta)\circ(I-R(\pmb\eta,\pmb\eta_0))^{-1}
\end{equation*}
which is a right inverse of $A_\wedge(\pmb\eta)$.

We now show that $B_\wedge(\pmb\eta)$ has the property described at the beginning of the proof. The operator $R(\pmb\eta,\pmb\eta_0)$ depends smoothly on $\pmb\eta$, and has the property that if
\begin{equation*}
f(\pmb\eta)\in x^{-1/2}L^2_b(\Z^\wedge_{\pmb\eta_0};F_{\Z^\wedge_{\pmb\eta_0}})
\end{equation*}
is a smooth section of $F^\wedge$ over $\open \wp_{\wedge,\Y}^{-1}(V)$ then $R(\pmb\eta,\pmb\eta_0)f$ is also smooth as a section of $F^\wedge$ over the same set. This can be seen as follows. If $f$ is smooth as specified, then the  section $B_\wedge(\pmb\eta_0)f(\pmb\eta)$ of $E^\wedge$ along $\Z^\wedge_{y_0}$, is also smooth, by the ellipticity of $A_\wedge(\pmb\eta_0)$. Since in any case $\gamma_{A_\wedge(\pmb\eta_0)}B_\wedge(\pmb\eta_0)f(\pmb\eta)\in \trb_{y_0}$ is a smooth section of $E^\wedge$ over $\open\Z^\wedge_{y_0}$, so is $\pi_{\min}B_\wedge(\pmb\eta_0)f(\pmb\eta)$ since
\begin{equation*}
\pi_{\min}B_\wedge(\pmb\eta_0) = B_\wedge(\pmb\eta_0)-\omega\mathfrak i \gamma_{A_\wedge(\pmb\eta_0)}B_\wedge(\pmb\eta_0)
\end{equation*}
Inspection of the definition of $R(\pmb\eta,\pmb\eta_0)$ shows that these properties are inherited by this operator, hence by $(I-R(\pmb\eta,\pmb\eta_0))^{-1}$, which exists for $\pmb\eta$ in a neighborhood $V\subset T^*\Y\minus 0$ of $\pmb\eta_0$, and in conclusion, by $B_\wedge(\pmb\eta)$. It follows that the projection $\wp_{\pmb\eta}$ has the property described in Lemma~\ref{SmoothnessOfProjection}. This ends the proof of the lemma.
\end{proof}

It follows from \eqref{WedgeSymbolKappaHom} that $\kappa_\varrho$ maps $\KK_{\pmb\eta}$ onto $\KK_{\varrho\pmb\eta}$, in other words, we have an induced $\R_+$-action
\begin{equation*}
\kappa:\KK\to\KK
\end{equation*}
by bundle diffeomorphisms covering the radial action on $T^*\Y\minus 0$. The bundle homomorphism $\gamma_{A_\wedge}$ intertwines this action on $\KK$ and that of $\kappa_\varrho$ on $\pi_\Y^*\trb$. Indeed, suppose $u\in \KK_{\pmb\eta}$. Then $u=\omega\mathfrak i\tau+\zeta$ uniquely with $\tau=\gamma_{A_\wedge(\pmb\eta)}u$ and  $\zeta\in \Dom_{\min}(A_{\wedge}(\pmb\eta))$, and
\begin{align*}
(\gamma_{A_\wedge(\varrho\pmb\eta)}\circ \kappa_\varrho)(\omega\mathfrak i\tau+\zeta)
&=\gamma_{A_\wedge(\varrho\pmb\eta)}(\kappa_{\varrho}\omega\, \kappa_{\varrho}\mathfrak i \tau)+(\gamma_{A_\wedge(\varrho\pmb\eta)}\circ \kappa_\varrho)\zeta\\
&=\gamma_{A_\wedge(\varrho\pmb\eta)}(\omega\,\mathfrak i \kappa_{\varrho}\tau)\\
&=\kappa_\varrho\tau
\end{align*}
using that $\kappa_\varrho$ preserves $\Dom_{\min}(A_\wedge(\pmb\eta))$ and that the latter space is the kernel of $\gamma_{A_\wedge(\pmb\eta)}$. Thus $\gamma_{A_\wedge(\varrho\pmb\eta)}\kappa_\varrho(u)=\kappa_\varrho\gamma_{A_\wedge(\pmb\eta)}u$. 

\begin{proposition}
The map $\gamma_{A_\wedge}$ is an equivariant isomorphism $\KK\to \kerb$.
\end{proposition}

\section{The extension operator}\label{sec-Extension}

We continue with the assumptions listed in the first paragraph of the previous sections.

The following lemma makes strong use of the fact that $A$ is a first order operator; it is false for higher order elliptic $w$-operators. 

\begin{lemma}\label{DmaxDminlocalize}
The spaces $\Dom_{\min}(A)$ and $\Dom_{\max}(A)$, hence also $\Dom_{\max}(A)/\Dom_{\min}(A)$, are modules over $\Ring$.
\end{lemma} 

\begin{proof}
If $f\in\Ring$ then  the commutator $[A,f]=\im\wsym(A)(\wdiff f)$ is a zeroth order operator, see \eqref{wBracket}, in other words, a section of $\Hom(E,F)$ smooth up to the boundary. For such $f$, if $u\in \Dom_{\max}(A)$ then 
\begin{equation*}
A(f u)=fA(u)+[A,f] u
\end{equation*}
in $\open \M$. Since $[A,f] u$ and $A u$ both belong to $x^{-1/2}L^2_b(\M;F)$, $A(f u)$ also belongs to $x^{-1/2}L^2_b(\M;F)$. 

Next, if $u\in \Dom_{\min}(A)$ and $\set{u_k}_{k=1}^\infty$ is a sequence in $C_c^\infty(\open \M;E)$ converging to $u$ in $A$-norm, then $\set{fu_k}_{k=1}^\infty$ is a sequence in $C_c^\infty(\open \M;E)$, and since
\begin{equation*}
\|fu-fu_k\|_A^2\leq 2(\|f(Au-Au_k)\|^2+\|[A,f](u-u_k)\|^2+\|f(u-u_k)\|^2)
\end{equation*}
converges to $0$ as $k\to \infty$, $fu\in \Dom_{\min}(A)$.
\end{proof}

The relevancy of the lemma lies in that it allows for localization. We will take advantage of this to define an extension operator
\begin{equation*}
\Ext:C^\infty(\Y;\trb)\to C^\infty(\open\M;E)
\end{equation*}
using a suitable partition of unity and the local triviality of $\trb$ to reduce to a local definition.

We begin with a local extension operator on generalized functions which we eventually assemble into a global operator with the aid of a partition of unity. 

In the following, $[\cdot] : \R \to \R$ is a $C^\infty$ function such that $[r] \geq 1$ for all $r \in \R$ and $[r]=|r|$ for $|r|$ large. Fix an arbitrary Riemannian metric on $\Y$ and denote by $|\pmb\eta|$ the norm of the covector $\pmb\eta$. Reflecting the properties of $[\cdot]$, $[|\pmb\eta|]\geq 1$ and $[|\pmb\eta|]=|\pmb\eta|$ if $|\pmb\eta|>r_0$. Fix also a cut-off function $\omega\in C_c^\infty(\overline \R_+)$ with support in $x<c_0$ for some $c_0>0$.

\begin{lemma}\label{extsymbol}
Let $y_1\dotsc,y_q$ be coordinates for $\Y$ in an open set $U\subset \Y$, denote the metric on $T^*\Y$ in the local coordinates by $\sum_{k,\ell} g^{k\ell}\eta_k\eta_\ell$ , write
\begin{equation}\label{BracketEtaY}
[\eta]_y  = \Big[\sqrt{\sum_{k,\ell}g^{k\ell}(y)\eta_k\eta_\ell}\Big].
\end{equation}
Let $V=\lbra0,\infty\rpar\times U$. The function 
\begin{equation*}
V\times \R^q\ni (x,y,\eta)\mapsto\omega(x[\eta]_y)\in \R
\end{equation*}
is supported in $x < c_0$, is a symbol in $S^{-\infty}(\open V\times \R^q)$ and its restriction to $x=x_0>0$ tends to $1$ in $S^\eps_{1,0}(U\times\R^q)$ as $x_0\to 0$ for any $\eps>0$.
\end{lemma}

\begin{proof}
Let $\chi\in C_c^\infty(\R)$ be supported in $|x|<c$. Assuming $x>0$, the function $\chi(x[\eta]_y)$ vanishes if $[\eta]_y>c/x$, so
\begin{equation*}
|\chi(x[\eta]_y)|\leq \frac{c^N\sup|\chi|}{x^N}[\eta]_y^{-N}
\end{equation*}
for any $N>0$. In addition, for $|\eta|$ large,
\begin{equation}\label{Derivatives}
\begin{gathered}
x\frac{\partial}{\partial x }\chi(x[\eta]_y)=x[\eta]_y\chi'(x[\eta]_y)\\
\frac{\partial}{\partial y_j}\chi(x[\eta]_y)=x[\eta]_y\chi'(x[\eta]_y)\frac{1}{2[\eta]_y^2}\sum_{k,\ell}\frac{\partial g^{k\ell}}{\partial y_j}(y)\eta_k\eta_\ell\\
\frac{\partial}{\partial \eta_j}\chi(x[\eta]_y)=x[\eta]_y\chi'(x[\eta]_y)\frac{1}{[\eta]_y^2}\sum_\ell g^{j\ell}(y)\eta_\ell.
\end{gathered}
\end{equation}
In all cases the right hand side is of the form $\tilde \chi(x[\eta]_y)a(y,\eta)$ with $\tilde \chi\in C_c^\infty(\R)$ and $a(y,\eta)$ a symbol of type $(1,0)$, of order $0$ in the first two cases and order $-1$ in the last. An argument by induction and the above estimate then gives that $(x,y,\eta)\mapsto\omega(x[\eta]_y)$ is a symbol of order $-\infty$ on $x>0$. We'll leave the proof of the rest of the assertions to the reader. 
\end{proof}

Continuing with the objects and notation in the lemma, let $\varphi$, $\psi\in C_c^\infty(U)$ with $\psi$ regarded as a function on $V$ independent of $x$. Let $f$ be a smooth (scalar) function or distribution on $U$, define ${}_\psi\ext_\varphi f$ by
\begin{equation*}
({}_\psi\ext_\varphi f)(x,y) = \frac{\psi(y)}{(2\pi)^q}\int_{\R^q} e^{iy\cdot\eta}\omega(x[\eta]_y)(\varphi f)\ft(\eta)\,d\eta. 
\end{equation*}
Here $(\varphi f)\ft$ is the Fourier transform of $\varphi f$. By Lemma~\ref{extsymbol}, if $f \in C^{-\infty}(U)$, then
\begin{equation}\label{ClassicalTrace}
{}_\psi\ext_\varphi f \in C^\infty(\open V),\quad ({}_\psi\ext_\varphi f)(x,y)=0\text{ if }x>c_0,\text{ and }\lim_{\eps\searrow 0}{}_\psi\ext_\varphi f\big|_{x=\eps}=\psi\varphi f
\end{equation}
where the limit is in the sense of distributions. Note that the operator ${}_\psi\ext_\varphi$ is generic, not related to any particular differential operator $A$.

\begin{lemma}\label{scalarlocallylikecutoff}
The map
\begin{equation*}
{}_\psi\ext_\varphi:C^{-\infty}(U)\to C^\infty(\open V)
\end{equation*}
has the following properties:
\begin{enumerate}
\item If $f \in C^\infty(U)$ then
\begin{equation*}
{}_\psi\ext_\varphi f - \psi\varphi f \in \dot C^\infty(V).
\end{equation*}
\item ${}_\psi\ext_\varphi:L^2_{\loc}(U) \to x^{-1/2}H_e^\infty(V)$.
\item If $\supp\varphi\cap \supp\psi=\emptyset$, then ${}_\psi\ext_\varphi$ maps $C^{-\infty}(U)$ into $\dot C^\infty(V)$.
\end{enumerate}
\end{lemma}
Recall from \cite{Mel81} that if $\M$ is a manifold with boundary then $\dot C^\infty(\M)$ consists of all elements of $C^\infty(\M)$ that vanish to infinite order at the boundary. The term $\psi\phi f$ in (1) is to be regarded as a function on $V$ independent of $x$.

\begin{proof}
The Fourier inversion formula gives 
\begin{equation*}
({}_\psi\ext_\varphi f)(x,y) - (\psi\varphi f)(x,y) = \frac{\psi(y)}{(2\pi)^q}\int_{\R^q} e^{iy\cdot\eta}(\omega(x[\eta]_y)-1)(\varphi f)^\wedge(\eta)\, d\eta
\end{equation*}
on $x>0$. We have, with an arbitrary nonnegative integer $k$, that
\begin{equation*}
\frac{1}{x^{k}}\big(\omega(x[\eta]_y)-1\big)=[\eta]_y^k h_k(x[\eta]_y)
\end{equation*}
with $h_k(x)=(\omega(x)-1)/x^{k}$. This is a smooth function since it vanishes near $x=0$. The additional increase in $\eta$ resulting from taking an arbitrary number of derivatives of
\begin{equation*}
e^{iy\cdot\eta}[\eta]_y^k h_k(x[\eta]_y)
\end{equation*}
with respect to $x$ or the $y_j$ can be absorbed by the rapid decrease of $(\varphi f)\ft(\eta)$, resulting in uniform estimates for 
\begin{equation*}
\frac{1}{x^{k}}\big({}_\psi\ext_\varphi f - \psi\varphi f)
\end{equation*}
and all its derivatives up to $x=0$. This proves the first assertion.

For the second assertion, suppose $f\in L^2_{\loc}(U)$. Using induction and the formulas \eqref{Derivatives} we find that $(xD_y)^\alpha x^k D_x^k{}_\psi\ext_\varphi$ is, for arbitrary $k$ and $\alpha$, a linear combination with constant coefficients of operators
\begin{multline*}
f\mapsto E^{k,\alpha}_{\alpha',\alpha'',\ell}(f)(x,y)=\\x^{|\alpha|}\frac{\psi^{(\alpha')}(y)}{(2\pi)^q}\int_{\R^q} e^{iy\cdot\eta}\eta^{\alpha''}\omega_{k+\ell}(x[\eta]_y)g^\alpha_{\alpha',\alpha'',\ell}(y,\eta)(\varphi f)\ft(\eta)\,d\eta
\end{multline*}
with $\alpha'+\alpha''\leq \alpha$ (componentwise), $\ell\leq |\alpha-\alpha'-\alpha''|$, $\omega_k(x)=x^k\omega^{(k)}(x)$ and $g^\alpha_{\alpha',\alpha'',\ell}\in S^0(U\times \R^q)$. Each of these is, for fixed $x> 0$, a pseudodifferential operator with symbol
\begin{equation*}
(y,y',\eta)\mapsto \psi^{(\alpha')}(y)\varphi(y') \Big(\frac{\eta}{[\eta]_y}\Big)^{\alpha''}\!\!(x[\eta]_y)^{|\alpha|}\omega_{k+\ell}(x[\eta]_y)\frac{g^\alpha_{\alpha',\alpha'',\ell}(y,\eta)}{[\eta]_y^{|\alpha-\alpha''|}}
\end{equation*}
in $S^0_{1,0}(U\times U\times\R^q)$. The seminorms of these symbols are uniformly bounded independently of $x>0$. Consequently, there is a constant $C_{k,\alpha}$ such that 
\begin{equation*}
\|(xD_y)^\alpha x^k D_x^k({}_\psi\ext_\varphi f)(x,\cdot)\|_{L^2(U)}\leq C_{k,\alpha}\|\chi f\|_{L^2(U)}
\end{equation*}
for all $x$, where $\chi\in C_c^\infty(U)$ is any function with $\chi=1$ on $\supp\varphi$. Thus
\begin{equation*}
\int_0^\infty\int_U\big|x^{1/2}(xD_y)^\alpha x^k D_x^k({}_\psi\ext_\varphi f)(x,y)\big|^2\,dy\,\frac{dx}{x}\leq c_0C_{k,\alpha}^2\|\chi f\|^2_{L^2(U)},
\end{equation*}
taking advantage also of the support property in \eqref{ClassicalTrace}.

The third assertion is most easily proved using the limit property in \eqref{ClassicalTrace} and that the Schwartz kernel of ${}_\psi\ext_\varphi$, the distribution 
\begin{equation*}
\frac{\psi(y)\varphi(y')}{(2\pi)^q}\int_{\R^q} e^{\im(y-y')\cdot \eta}\omega(x[\eta]_y)\,d\eta,
\end{equation*}
is smooth up to $x=0$ and vanishes there if $\supp\varphi\cap \supp\psi=\emptyset$. The details are left to the reader. 
\end{proof}

We now reinterpret the operators ${}_\psi\ext_\varphi$ as operators 
\begin{equation*}
{}_\psi\ext_\varphi:C^{-\infty}(U)\to C^\infty(\open \wp_\wedge^{-1}(U)).
\end{equation*}
The functions $\varphi$ and $\psi$ are in $C^\infty(\Y)$, compactly supported in $U$ and we write also $\psi$ for $\wp_\wedge^*\psi$. If $\nu\in \open \N^\wedge$ then $x$ has a well defined value at $\nu$ (we fixed $x$ at the beginning of Section~\ref{sec-TraceMap}), and if $\nu\in \wp_\wedge^{-1}(U)$ then via the coordinates we have a point in $\R^q$ corresponding to $\wp_\wedge(\nu)$. So if $f\in C^{-\infty}(\Y)$ there is a well defined value for ${}_\psi\ext_\varphi f$ at $\nu$. This defines ${}_\psi\ext_\varphi f$ as a smooth function on $\open \N^\wedge$. The function ${}_\psi\ext_\varphi f$ is defined as zero in the complement of $\wp_\wedge^{-1}(U)$ in $\open\N^\wedge$. The operator depends on the local coordinates used to define it, but that is unimportant.

If $f$ is smooth, then, by Part (1) of Lemma~\ref{scalarlocallylikecutoff}, ${}_\psi\ext_\varphi f\in C^\infty(\N^\wedge)$ and its restriction to $x=0$ (that is, to $\N$) is equal to $\wp^*(\psi \varphi f)$:
\begin{equation}\label{MiniExt}
f\in C^\infty(\Y)\implies {}_\psi\ext_\varphi f\in \Ring^\wedge.
\end{equation}
Recall that $\Ring^\wedge$ consists of all smooth functions $f$ on $\N^\wedge$ such that $f\big|_\N\in \wp_\wedge^*C^\infty(\Y)$.

\medskip
Suppose now that $U$ is so small that the trace bundle $\trb\to \Y$ is trivial over $U$ and let $\tau_\mu$, $\mu=1,\dotsc,N$ be a frame of $\trb$ over $U$. Let $\psi$, $\varphi\in C_c^\infty(U)$. Define
\begin{equation*}
{}_\psi \Ext_\varphi:C^{-\infty}(\Y;\trb)\to C^\infty(\open\N^\wedge;E^\wedge)
\end{equation*}
by
\begin{equation}\label{localExt}
{}_\psi \Ext_\varphi u = \sum_{\mu=1}^N {}_\psi \ext_\varphi(u^\mu)\tau_\mu\quad\text{if }u = \sum_{\mu=1}^N u^\mu \tau_\mu.
\end{equation}
Note that ${}_\psi \Ext_\varphi u$ is always a smooth global section of $E^\wedge$ over $\open \N^\wedge$ supported in $\wp_\wedge^{-1}(U)\minus\N$. Note further that ${}_\psi\Ext_\varphi u=0$ if $x>c_0$ because of the support property in \eqref{ClassicalTrace}, so if $c_0$ is small enough (assume this is the case) we can also regard ${}_\psi \Ext_\varphi u$ as defined on $\M$ via $\Phi$. Finally, observe that
\begin{equation}\label{MicolocalPoisson}
u\in C^\infty(\Y;\trb)\implies {}_\psi \Ext_\varphi u\in \Dom_{\max}(A')
\end{equation}
because of \eqref{MiniExt} and Lemmas~\ref{prePprime} and \ref{DmaxDminlocalize}.

\medskip
Cover $\Y$ by finitely many open subsets $U_s$, $s=1,\dotsc,S$, of the kind used in defining \eqref{localExt}. Choose a subordinate partition of unity $\set{\varphi_s}$ and functions $\psi_s \in C_c^\infty(U_s)$ such that $\psi_s = 1$ in a neighborhood of $\supp\varphi_s$.

\begin{definition}
The operator
\begin{equation*}
\Ext = \Phi_*\sum_{s=1}^S {}_{\psi_s}\Ext_{\varphi_s} : C^{-\infty}(\Y;\trb) \to C^\infty(\open \M;E)
\end{equation*}
is an \emph{extension operator} associated with $A$.
\end{definition}

Recall that the map $\Phi_*$ transfers sections of $E^\wedge$ (or $F^\wedge$) defined in a neighborhood of $\N$ in $\N^\wedge$ to sections of $E$ (or $F$) defined in a neighborhood of $\N$ in $\M$ with the aid of a tubular neighborhood map and connections (see Section~\ref{sec-SummaryWedgeOps}, for instance \eqref{NormalFamilyAsLimit}). Since $\omega(x[\eta]_y)$ vanishes for sufficiently large $x$ independently of $\eta$, the support of elements in the image of ${}_{\psi_s}\Ext_{\varphi_s}$ is close to $\N$ (in $\N^\wedge$), so we may view elements in the image of $\Ext$ as defined on all of $\M$ (and supported near $\N$). The operator $\Ext$ is an extension operator adapted to $A$:
\begin{equation}\label{ExtisGood}
\gamma_A\circ\Ext :C^\infty(\Y;\trb)\to C^\infty(\Y;\trb)\text{ is the identity operator}. 
\end{equation}
Indeed, if $u\in C^\infty(\Y;\trb)$, then Lemma~\ref{scalarlocallylikecutoff} gives
\begin{equation}\label{ExtOnSmooth}
{}_{\psi_s} \Ext_{\varphi_s} u-\psi_s\preP(\varphi_s u)\in \dot C^\infty(\N^\wedge;E)
\end{equation}
from which $\gamma_A(\Ext u)= u$ follows using \eqref{GammaInvertsPreP} and Proposition~\ref{DminSubsetKerTrace}. We prove stronger results than \eqref{ExtisGood} about $\Ext$ later in this section.

Another immediate property of $\Ext$ is that, because of \eqref{MicolocalPoisson}, it maps $C^\infty(\Y;\trb)$ into the maximal domain of $A$, thus
\begin{equation*}
\Ext:C^\infty(\Y;\trb)\to \Dom_{\max}(A)\cap C^\infty(\open \M;E).
\end{equation*}
This of course need not be true of the operators ${}_{\psi_s} \ext_{\varphi_s}$, since these are generally unrelated to $A$.

\medskip
In the rest of this section we analyze the maps ${}_{\psi_s}\Ext_{\varphi_s}$ is detail in order to improve on the last two listed properties of $\Ext$. We drop the index $s$ from the notation. 

Recall from Proposition~\ref{xdxAction} that $x\partial_x$ defines an endomorphism of $\trb$. We will write $\gen$ for the endomorphism $x\partial_x+1/2$, that is, \eqref{Generator} with $\gamma=1/2$. Pick a positive $\delta<1$. We will additionally assume that
\begin{equation}\label{UIsDeltaAdmissible}
\display{300pt}{the open set $U\subset \Y$ is so small that $\trb$ has a $\delta$-admissible frame $\tau_1,\dotsc,\tau_N$ over $U$ with respect to $\gen$.
}
\end{equation}
This means, see \cite[Definition 2.1]{KrMe12b}, that the eigenvalues of the fiberwise action of $\gen$ remain within a disjoint family of sets (clusters) in $\C$ of diameter less than $\delta$ as $y$ varies in $U$, and that the frame consists of linear combinations of elements in the generalized eigenspaces associated to the clusters. The statement about the clustering of the eigenvalues has a direct translation to the boundary spectrum of $A$ in view of \eqref{SpecGen}. The compactness of $\Y$ and \eqref{NobSpecOnCritLinesOrderm} (with $\gamma=1/2)$ imply the existence of a number $0 <\delta_0< \frac{1}{2}$ such that
\begin{equation}\label{delta0}
\spec_b(\bA_y)\cap\set{\sigma:-1/2<\Im\sigma<1/2} \subset \Set{\sigma \in \C: -\frac{1}{2}+\delta_0<\Im\sigma<\frac{1}{2}-\delta_0}
\end{equation}
for every $y\in \Y$.

We shall further assume that the boundary fibration $\N\to\Y$ is trivial over $U$ in addition to it being the domain of a coordinate system as already indicated above. We then make the identifications
\begin{equation*}
\wp^{-1}(U)=U\times\Z,\quad\wp_\wedge^{-1}(U)=U\times \Z^\wedge,
\end{equation*}
$\Z^\wedge=\lbra0,\infty\rpar\times \Z$ and let 
\begin{equation*}
W=\lbra0,\eps\rpar\times U\times \Z.
\end{equation*}
The open set $U$, the functions $\varphi$, $\psi$, $\omega$, and the $\delta$-admissible frame $\tau_1,\dots,\tau_N$ will remain fixed until further notice. We will write $E$ instead of $E_{\Z^\wedge}$ for the vector bundle over $\Z^\wedge$ and $\tau_\mu(y)$ for the section $(x,z)\mapsto \tau_\mu(x,y,z)$ of $E$.

\medskip
Let $H_1$ and $H_2$ be Hilbert spaces equipped with strongly continuous group actions $\kappa_{j,\varrho}$, $\varrho > 0$, $j=1,2$. Denote by $\L(H_1,H_2)$ the space of continuous linear maps $H_1\to H_2$. Recall from \cite{SchuNH} that the class of twisted operator-valued symbols
$S^m(U\times\R^q;H_1,H_2)$ is defined as the space of all 
\begin{equation*}
p(y,\eta) \in C^\infty(U\times\R^q,\L(H_1,H_2))
\end{equation*}
such that for all $\alpha,\beta \in \NN_0^q$ and $K\Subset U$ there exists $C_{\alpha,\beta} \geq 0$ such that
\begin{equation}\label{ConstantOrder}
\|\kappa_{2,[\eta]}^{-1}\,\bigl(D_y^{\alpha}\partial_{\eta}^{\beta}p(y,\eta)\bigr)\,\kappa_{1,[\eta]}\|_{\L(H_1,H_2)} \leq C_{\alpha,\beta} [\eta]^{m-|\beta|}
\end{equation}
for all $(y,\eta) \in K\times\R^q$. Here $[\eta]$ is given by \eqref{BracketEtaY} but with the Euclidean metric (we omit the reference to $y$ in this case). We review some of the properties of the pseudodifferential calculus associated with such symbols in Appendix~\ref{app-PseudosOpValued}.

A map $p\in C^\infty(U\times\R^q,\L(H_1,H_2))$ is homogeneous of degree $m$ for large $|\eta|$ if
\begin{equation*}
p(y,\varrho\eta)=\varrho^m\kappa_{2,\varrho}\,p(y,\eta)\,\kappa_{1,\varrho}^{-1}
\end{equation*}
holds when $\varrho\geq 1$ for all sufficiently large $|\eta|$, uniformly on compact subsets of $U$ (see \cite[Definition 5, Section 3.2]{SchuNH}). The reader may verify that such a map belongs to $S^m(U\times\R^q;H_1,H_2)$. An example of this is the function $\omega(x[\eta]_y)$ with $[\eta]_y$ given by \eqref{BracketEtaY} and the following setup: let $H_1=\C$ with the trivial $\R_+$-action and $H_2=x^{-1/2}L^2_b(\Z^\wedge)$ with the action 
\begin{equation}\label{ActionOneHalf}
\kappa_\varrho u(x,z) = \varrho^{1/2}u(\varrho x,z)\text{ for }\varrho > 0,
\end{equation}
and let $p:U\times\R^q\to \L(H_1,H_2)$ be defined, for $(y,\eta)\in U\times\R^q$, as multiplication by $\omega(x[\eta]_y)$. Then $\varrho^{-1/2}\kappa_\varrho p(y,\eta) =\omega(\varrho x[\eta]_y)$ which is equal to $\omega(x [\varrho \eta]_y)$ if $\varrho\geq 1$ and $|\eta|$ is sufficiently large. Thus $p$ is homogeneous of degree $-1/2$ for large $|\eta|$.

The spaces $\K^{s,\gamma}_t(\Z^\wedge;E)$ (see \eqref{Kegel}) will be relevant for us because $\tau_\mu\in \K^{s,-1/2+\delta_0}_{t}$ for any $s$ and for any $t<-\dim\Z^\wedge -1/2+\delta_0$. These objects appear in ${}_\psi\Ext_\varphi$ in the form $\omega(x[\eta]_y)\tau_\mu$, so more relevant for our purposes is the fact that
\begin{equation*}
\omega(x[\eta]_y/[\eta])D_y^\alpha\tau_\mu\in \K^{s,-1/2+\delta_0}_{t}\quad\text{ for all }s,\ t\in \R,\text{ and }\alpha\in \NN_0^q.
\end{equation*}
Indeed, now the behavior of $\tau_\mu$ at infinity is irrelevant, while the behavior of the derivatives near $0$ is a consequence of the meaning of smoothness, of Lemma~\ref{SmoothnesOnMellinSide}, and the Mellin inversion formula, since the location of the poles of the expression in \eqref{MellinOmegaU} does not change as the expression is differentiated. That $\partial_y^\alpha\tau_\mu$ may not be a section of $\trb$ if $\alpha\ne 0$ is also immaterial.

Equip the spaces $\K^{s,\gamma}_t(\Z^\wedge;E)$ with the strongly continuous group action \eqref{ActionOneHalf} and the spaces $\C^N$ with the trivial group action.

\begin{lemma}\label{Poissonsymbol}
Define ${}_\psi\extsym_\varphi:U\times U\times \R^q\to \L\big(\C^N,C^\infty(\open \Z^\wedge;E)\big)$ by
\begin{equation*}
{}_\psi\extsym_\varphi(y,y',\eta): \zeta\mapsto \psi(y)\varphi(y') \omega(x[\eta]_y)  \sum_{j=1}^N \zeta^\mu \tau_\mu(y).
\end{equation*}
Then ${}_\psi\extsym_\varphi$ belongs to 
\begin{equation*}
S^{-\delta_0}(U\times U\times\R^q;\C^N, \K^{\infty,-1/2+\delta_0}_\infty) =
\bigcap_{s,t \in \R} S^{-\delta_0}(U\times U\times\R^q,\C^N, \K^{s,-1/2+\delta_0}_t)
\end{equation*}
for any $\delta_0$ as in \eqref{delta0}. Furthermore,
\begin{equation*}
D_y^{\alpha}\partial_{\eta}^{\beta}{}_\psi\extsym_\varphi\in
S^{-|\beta|-\delta_0}(U\times U\times\R^q;\C^N, \K^{\infty,\infty}_\infty)
\end{equation*}
for all $\alpha$, $\beta \in \NN_0^q$ with $|\beta| > 0$.
\end{lemma}

\begin{proof}
Consider the group action \eqref{ActionOneHalf} on general sections of $E^\wedge$  on $\Z^{\wedge}$. Its infinitesimal generator restricted to elements of $\trb$ is just $\gen$, obviously. Since the $\tau_\mu$ form a smooth frame of $\trb$,
\begin{equation}\label{ActionOnFrame}
\kappa_{\varrho}^{-1}\tau_\mu(y)=\sum_\nu c^\nu_\mu(y,\varrho)\tau_\nu(y)
\end{equation}
for some $c^\nu_\mu\in C^\infty(U\times\R_+)$. Because of \eqref{delta0} and the specific structure of the $\tau_\mu$, 
\begin{equation}\label{cOfNegOrder}
\text{the functions }c^\nu_\mu\text{ are symbols of order }\leq -\delta_0.
\end{equation}
To illustrate this consider $\tau=x^{\im\sigma}$ with one of the relevant numbers $\sigma$. Then $\kappa_{\varrho}^{-1}\tau=\varrho^{-\im \sigma-1/2}\tau$, and estimates like $|\varrho^{-\im \sigma-1/2}|=\varrho^{\Im\sigma-1/2}$ for large $\varrho$ coupled with the fact that
\begin{equation*}
\spec(\gen_y) \subset \set{\lambda \in \C: \delta_0<\Re\lambda<1-\delta_0}
\end{equation*}
for every $y\in \Y$ give that $\varrho^{-\im \sigma-1/2}$ is a symbol of order $<-\delta_0$. It follows that $(y,\eta)\mapsto c^\ell_j(y,[\eta])$ is an element of $S^{-\delta_0}(U\times\R^q)$. A rigorous proof of \eqref{cOfNegOrder} may be obtained using the Dunford integral representation
\begin{equation*}
\kappa^{-1}_{\varrho} = \frac{1}{2\pi i } \int_{\Gamma} \varrho^{-\lambda}(\lambda-\gen_y)^{-1}\,d\lambda : \trb_y \to \trb_y,
\end{equation*}
where $\Gamma \subset \set{\lambda \in \C: \delta_0<\Re\lambda<1-\delta_0}$ is a fixed contour enclosing $\spec(\gen_y)$ for all $y \in \Y$.

Returning to ${}_\psi\extsym_\varphi$, we have
\begin{equation*}
\kappa_{[\eta]}^{-1}D_y^\alpha\partial_\eta^\beta \big(\omega(x[\eta]_y)\tau_\mu\big)=\sum_{\alpha'\leq \alpha}\binom{\alpha}{\alpha'}[\eta]^{1/2}\big(\kappa_{[\eta]}^{-1}D_y^{\alpha'}\partial_\eta^\beta \omega(x[\eta]_y)\big)D_y^{\alpha-\alpha'}\kappa_{[\eta]}^{-1}\tau_\mu.
\end{equation*}
since $\kappa_{[\eta]}^{-1}$ commutes with $D_y^\alpha$. Furthermore, using formulas such as \eqref{Derivatives} (or $\kappa$-homogeneity in the large as discussed above) and induction we see that 
\begin{equation*}
D_y^{\alpha'}\partial_\eta^\beta\omega(x[\eta]_y)=\sum_{\substack{\alpha''\leq \alpha'\\\beta'\leq \beta}}q^{\alpha',\beta}_{\alpha'',\beta'}(y,\eta)\omega^{\alpha',\beta}_{\alpha'',\beta'}(x[\eta]_y)
\end{equation*}
with $q^{\alpha,\alpha'}_{\alpha'',\beta'}\in S^{-|\beta|}(U\times \R^q)$ and $\omega^{\alpha',\beta}_{\alpha'',\beta'}\in C^\infty_c(\R)$ vanishing near $0$ if $\beta\ne 0$. We thus have
\begin{equation*}
\kappa_{[\eta]}^{-1}D_y^\alpha\partial_\eta^\beta \big(\omega(x[\eta]_y)\tau_\mu\big)=\sum q^{\alpha',\beta}_{\alpha'',\beta'}(y,\eta)\omega^{\alpha',\beta}_{\alpha'',\beta'}(x[\eta]_y/[\eta])D_y^{\alpha-\alpha'}(
c^\nu_\mu(y,[\eta])\tau_\nu),
\end{equation*}
hence
\begin{equation*}
\|\kappa_{[\eta]}^{-1}D_y^\alpha\partial_\eta^\beta \big(\omega(x[\eta]_y)\tau_j\big)\|_{\K^{s,\gamma}_t} \leq  C[\eta]^{-\delta_0-|\beta|}
\end{equation*}
for $\gamma=-1/2+\delta_0$ and any $s$, $t$, and, if $\beta\ne 0$, also for all $\gamma$.
Thus the lemma follows.
\end{proof}

\begin{lemma}\label{MappingFromL2}
The mapping ${}_\psi\Ext_\varphi : C^{-\infty}(U;\trb_U) \to C^\infty(\open W;E)$ restricts to a map
\begin{equation*}
{}_\psi\Ext_\varphi:L^2_{\loc}(U;\trb_U) \to x^{-1/2+\delta_0}H_e^\infty(W;E)
\end{equation*}
with $\delta_0$ as in \eqref{delta0}. If $\supp \varphi \cap \supp \psi = \emptyset$, then
\begin{equation*}
{}_\psi\Ext_\varphi:C^{-\infty}(U;\trb_U) \to \dot C^\infty(W;E).
\end{equation*}
\end{lemma}

\begin{proof}
The second statement follows immediately from Part (3) of Lemma~\ref{scalarlocallylikecutoff}. To prove the first statement, let $P\in \Diff_e(W;E)$ be an arbitrary edge-differential operator on $W$. We may write $P$ in the form
\begin{equation*}
(P v)(y)=\frac1{(2\pi)^q}\int_{U\times \R^q} e^{\im (y-y')\cdot\eta} p(y,\eta) v(y')\,dy'\,d\eta
\end{equation*}
with $p(y,\eta)=p_1(y,x\eta)$, $p_1(y,\eta)$ a polynomial in $\eta$ with coefficients depending smoothly on $y$ with values in $\Diff_b(\Z^\wedge;E)$. The variables $x$ and $z$ are implicit. Because of the factor $x$ accompanying $\eta$, $p(y,\eta)$ may be regarded as an operator valued symbol of order $0$ acting between the $\K$-spaces, see Proposition~\ref{EdgeOpsOpValued}. As such, the usual Leibniz rule applies: if $u\in C^{-\infty}(U;\trb)$, then $Q=P\circ {}_\psi\Ext_\varphi$ has the form
\begin{equation*}
Q u = \frac{1}{(2\pi)^q}\int_{U\times \R^q} e^{\im (y-y')\cdot\eta} q(y,y',\eta) \tilde u(y')\,dy'\,d\eta
\end{equation*}
where $\tilde u$ is the column vector whose components are the coefficients of $u$ with respect to the frame $\tau_\mu$, and 
\begin{equation*}
q=p\leibniz {}_\psi \extsym_\varphi \in S^{-\delta_0}(U\times U\times\R^q;\C^N, \K^{\infty,-1/2+\delta_0}_{\infty})
\end{equation*}
in view of Lemma~\ref{Poissonsymbol}.
Consequently,
\begin{equation*}
x^{-\delta_0}q\in S^0(U\times U\times\R^q;\C^N, x^{-1/2}L^2_b(\Z^{\wedge};E)),
\end{equation*}
and thus, in particular, $Q$ maps $L^2_\loc(U;\trb_U)$ into $L^2(U,x^{-1/2+\delta_0}L^2_b(\Z^{\wedge};E))$. Since $P$ is arbitrary we obtain the asserted mapping property
\begin{equation*}
{}_\psi{\Ext}_{\varphi} : L^2_{\loc}(U;\trb_U) \to x^{-1/2+\delta_0}H_e^{\infty}(W;E).
\end{equation*}
\end{proof}

The combination of the previous and following lemmas gives, in particular, that $\Ext$ maps $H^{1-\gen}(\Y;\trb)$ into $\Dom_{\max}(A)$. Here $\gen = x\partial_x+1/2$ is the section of $\End(\trb_U)$ defined in \eqref{Generator} (with $\gamma=1/2$, of course). The space $H^{1-\gen}(\Y;\trb)$ and other similar spaces are the Sobolev spaces of variable order defined in \cite[Section 5]{KrMe12b}. The following lemma is stated in terms of $A'$, the operator defined in Lemma~\ref{AonNwedge}, but the result applies equally well to $A$ because of Lemma~\ref{MappingFromL2} and the fact that elements of $\Diff^1_e(\M;E,F)$ such as $A-A'$ map $x^{-1/2+\delta_0}H_e^\infty(W;E)$ into 
$x^{-1/2+\delta_0}H_e^\infty(W;F)$.

\begin{lemma}\label{MappingintoDmax}
The composition
\begin{equation*}
A'\circ {}_\psi\Ext_{\varphi} : C^{-\infty}(U;\trb_U) \to C^{\infty}(\open W;F)
\end{equation*}
restricts to a continuous map
\begin{equation*}
A'\circ {}_\psi\Ext_{\varphi} : H^{1-\gen}_{\loc}(U;\trb_U) \to x^{-1/2}H_e^{\infty}(W;F).
\end{equation*}

\end{lemma}
\begin{proof}
Let $\pmb\tau=\begin{bmatrix}\tau_1,\dots,\tau_N\end{bmatrix}$ be the frame of $\trb$ used in \eqref{localExt} to define $_\psi\Ext_{\varphi}$, define $\mathfrak k:U\times \R^q\to \L(\C^N,\trb_U)$ by
\begin{equation*}
\mathfrak k(y,\eta)\zeta=\kappa_{[\eta]}\pmb\tau(y)\zeta.
\end{equation*}
So
\begin{equation*}
\mathfrak k(y,\eta)\zeta = \sum_{\mu,\nu} \tilde c^\nu_\mu(y,\eta)\zeta^\mu\tau_\nu(y)
\end{equation*}
where $\tilde c^\nu_\mu(y,\eta)=c^\nu_\mu(y,1/[\eta])$ since $\kappa_{[\eta]}\pmb\tau=\pmb\tau c(y,1/[\eta])$, see \eqref{ActionOnFrame}. We claim that 
\begin{equation}\label{mathfrakkinS0}
\mathfrak k\in S^m_{1,\delta}(U\times\R^q;(\C^N,0),(\trb_U,-\gen))
\end{equation}
with $m=0$. For general $m\in \R$ this means that, trivializing $\trb_U$ using the frame $\pmb\tau$ (it is here where the fact that the frame is $\delta$-admissible with respect to $\gen$, as stated in \eqref{UIsDeltaAdmissible}, is used for the first time) $\mathfrak k$ satisfies estimates of the form
\begin{equation}\label{VariableOrder}
\display{300pt}{
for all $\alpha, \beta\in \NN_0^q$ and $K\Subset U$ there is $C$ such that
\begin{equation*}
\big\|\langle \eta \rangle^{-\gen(y)}\big(D_y^{\alpha}\partial_{\eta}^{\beta}\mathfrak k(y,\eta)\big)\langle \eta \rangle^{0}\big\| \leq C \langle \eta \rangle^{m - |\beta| + \delta|\alpha|}\text{ if }y\in K,\ \eta\in \R^q
\end{equation*}
}
\end{equation}
with $\langle\eta\rangle=\sqrt{1+|\eta|^2}$, see \cite[Definitions 3.1 and 6.1]{KrMe12b}. The homomorphism $\langle\eta\rangle^{-\gen}$ is just $\kappa_{1/\langle\eta\rangle}$, and of course $\langle\eta\rangle^0$ is the identity homomorphism. To verify the claim we use that by \cite[Remark 3.9]{KrMe12b} the condition is equivalent to the statement that $\langle\eta\rangle^{-\gen}\mathfrak k\langle\eta\rangle^0$ belongs to the standard H\"ormander class $S^m_{1,\delta}$, which is trivially true here since 
$\kappa_{1/\langle\eta\rangle} = \kappa_{[\eta]/\langle\eta\rangle}\kappa^{-1}{[\eta]}$ so that
\begin{equation*}
\langle\eta\rangle^{-\gen}\mathfrak k\langle\eta\rangle^0 = \kappa_{[\eta]/\langle\eta\rangle}\kappa^{-1}_{[\eta]}\kappa_{[\eta]}\pmb\tau=\kappa_{[\eta]/\langle\eta\rangle}\pmb\tau.
\end{equation*}
The fact that $[\eta]/\langle\eta\rangle$ is a symbol of order $0$ gives that $\mathfrak k$ satisfies \eqref{mathfrakkinS0} with $m=0$. We note in passing that the condition \eqref{VariableOrder} is structurally the same as \eqref{ConstantOrder} except that in the former the action is independent of the space variable $y$ whereas here we have the action by $-\gen$ which may depend on $y$.

Define $\op(\mathfrak k):C_c^{-\infty}(U;\C^N)\to C^{-\infty}(U;\trb_U)$
\begin{equation*}
\big(\op(\mathfrak k)v\big)(y)=\frac{1}{(2\pi)^q}\int_{\R^n} e^{\im y\cdot \eta} \mathfrak k(y,\eta) \widehat v(\eta)\,d\eta
\end{equation*}
Then $\op(\mathfrak k)\in \Psi^{0}_{1,\delta}(U;(\C^N,0),(\trb_U,-\gen))$ reflecting the properties of $\mathfrak k$. By \cite[Proposition 5.2]{KrMe12b}, 
\begin{equation*}
\op(\mathfrak k):H^1_{\comp}(U;\C^N)\to H^{1-\gen}_{\loc}(U;\trb_U).
\end{equation*}
Observe that $\mathfrak k$ is an elliptic symbol in the sense of \cite{KrMe12b}. Indeed, the inverse of $\mathfrak k(y,\eta)$ is 
\begin{equation*}
\mathfrak l:\trb_y\ni u \mapsto [\tau^1(y)\kappa_{[\eta]}^{-1}u,\dotsc, \tau^N(y)\kappa_{[\eta]}^{-1}u] \in \C^N
\end{equation*}
where $\tau^1,\dotsc,\tau^N$ is the frame dual to the frame $[\tau_1,\dotsc,\tau_N]$; this defines an element of $S^0_{1,\delta}(U\times\R^q;(\trb_U,-\gen),(\C^N,0))$. Consequently $\op(\mathfrak k)$ has a properly supported right parametrix ${\mathfrak L} \in \Psi^{0}_{1,\delta}(U;(\trb_U,-\gen),(\C^N,0))$ modulo a smoothing operator $\mathfrak R \in \Psi^{-\infty}(U;\trb_U)$,
\begin{equation*}
\op(\mathfrak k)\circ {\mathfrak L}=\Id-\mathfrak R,
\end{equation*}
giving in particular a continuous operator
\begin{equation}\label{mathfrakLCont}
{\mathfrak L}:H^{1-\gen}_{\comp}(U;\trb_{U})\to H^1_{\comp}(U,\C^N).
\end{equation}
We will show in a moment that $A'\circ {}_\psi\Ext_{\varphi} \circ \op(\mathfrak k)$ restricts to a continuous operator 
\begin{equation}\label{auxeq1}
A'\circ {}_\psi\Ext_{\varphi} \circ \op(\mathfrak k):H^1_{\comp}(U;\C^N) \to x^{-1/2}H^{\infty}_e(W;F).
\end{equation}
Assuming this for the moment, we use 
\begin{equation*}
A'\circ {}_\psi\Ext_{\varphi}=A'\circ {}_\psi\Ext_{\varphi} \circ \op(\mathfrak k) \circ{\mathfrak L} +A'\circ {}_\psi\Ext_{\varphi}\circ\mathfrak R
\end{equation*}
together with \eqref{mathfrakLCont} and \eqref{auxeq1}, plus the regularization property
\begin{equation*}
\mathfrak R:H^{1-\gen}_{\comp}(U;\trb_U)\to C^\infty(U;\trb_U)
\end{equation*}
and the fact that, because $\bA$ annihilates the $\tau_j$, $A'\circ {}_\psi\Ext_\varphi$ maps $C^\infty(U;\trb_U)\to x^{-1/2}H^\infty_e(W;F)$, to conclude that indeed
\begin{equation*}
A'\circ {}_\psi\Ext_\varphi:H^{1-\gen}_{\comp}(U;\trb_U)\to x^{-1/2}H^\infty_e(W;F).
\end{equation*}
The thesis of the lemma follows from this by noting the presence of factor $\varphi$ in ${}_\psi\Ext_\varphi$.

We now show \eqref{auxeq1}. Recall that $\mathfrak k(y,\eta) = \kappa_{[\eta]}\pmb\tau(y)=\pmb\tau(y)\tilde c(y,\eta)$, $\tilde c=[\tilde c^\ell_j]$, as a map $\C^N\to \trb_y$. Note that 
\begin{equation}\label{ctildeOfSmallOrder}
\text{the functions }\tilde c^\ell_j\text{ are symbols of order}\leq 1-\delta_0. 
\end{equation}
This can be justified with an argument similar to the one given to lend credence to the assertion that the $c^\ell_j$ in \eqref{ActionOnFrame} belong to $S^{-\delta_0}(U\times\R_+)$. Let $\op(\tilde c):C_c^{-\infty}(U;\C^N)\to C^{-\infty}(U;\C^N)$ be the pseudodifferential operator defined by $\tilde c$. Then $\op(\mathfrak k)=\pmb\tau \op(\tilde c)$. The definition of ${}_\psi\Ext_\varphi$, see \eqref{localExt}, gives
\begin{align*}
\big({}_\psi\Ext_\varphi \circ \op(\mathfrak k)\big)(v) 
&= \sum_{\nu}\Big(\sum_{\mu=1}^N \big({}_\psi\ext_\varphi\circ \op(\tilde c^\nu_\mu)\big)(v^\mu)\Big)\tau_\nu\\
&= \op({}_\psi\extsym_\varphi\leibniz \tilde c)(v)
\end{align*}
We have
\begin{equation*}
r = {}_\psi\extsym_\varphi \leibniz \tilde c - {}_\psi\extsym_\varphi|_{\diag}\, \tilde c \in S^{-2\delta_0}(U\times U\times\R^q;\C^N,\K^{\infty,\infty}_{\infty})
\end{equation*}
by Lemma~\ref{Poissonsymbol} ($\diag$ means the diagonal in $U\times U$), see also Corollary~\ref{CorSmallerSpace}. Consequently, $x^{-1-\delta_0}r \in S^{1-\delta_0}(\R^q\times\R^q\times\R^q;\C^N,\K^{\infty,\infty}_{\infty})$, and by an argument analogous to that in the proof of Lemma~\ref{MappingFromL2} we now get that
\begin{equation}\label{lot}
\op(r) : H^1_\comp(U;\C^N) \to x^{1/2+\delta_0}H^{\infty}_e(W;E).
\end{equation}
In particular, 
\begin{equation*}
A'\circ \op(r) :  H^1_\comp(U;\C^N) \to x^{-1/2+\delta_0}H^{\infty}_e(W;F).
\end{equation*}

We now prove the desired mapping property for $A'\circ \op(p)$ with $p={}_\psi\extsym_\varphi|_{\diag}\, \tilde c.$ By definition,
\begin{equation*}
p(y,\eta) = \kappa_{[\eta]} \big(\omega(x[\eta]_y/[\eta])\psi(y)\varphi(y)  \pmb \tau\big) : \C^N \to C^{\infty}(\Z^{\wedge};E),
\end{equation*}
which shows that $p \in S^0(U\times\R^q;\C^N,\K_{\infty}^{\infty,-1/2+\delta_0})$.  Assuming $U$ is small enough, we may view the operator $A'$ on $\open W$ as being of the form $A'=\op(a)$ with $a(y,\eta) : C^{\infty}(\Z^{\wedge};E) \to C^{\infty}(\Z^{\wedge};F)$ given by 
\begin{equation*}
a(y,\eta)=\bA_y+q(y,\eta).
\end{equation*}
This is the coordinate version of the decomposition of $A_\wedge$ used in the proof of Lemma~\ref{AonNwedge}. Recall that $\bA_y$ is independent of $\eta$ and that $q$ is linear in $\eta$ since $A'$ is a first order wedge operator. In particular, $\partial_{\eta}^{\alpha}a(y,\eta)$ is a smooth bundle homomorphism if $|\alpha| = 1$ and the zero operator if $|\alpha|>1$. Thus
\begin{equation}\label{lotOrder}
\sum_{|\alpha| \geq 1}\frac{1}{\alpha !}\partial_{\eta}^{\alpha}a\, D_y^{\alpha} p=\sum_{|\alpha| = 1}\frac{1}{\alpha !}\partial_{\eta}^{\alpha}a\, D_y^{\alpha} p\in S^0(U\times\R^q;\C^N,\K_{\infty}^{\infty,-1/2+\delta_0}).
\end{equation}
Now let $\tilde{\omega} \in C^{\infty}(\lbra0,\eps\rpar)$ be such that $\tilde{\omega} = 1$ in a neighborhood of the support of $\omega$. Then
\begin{equation*}
a(y,\eta)p(y,\eta)=\tilde{\omega}(x[\eta]_y)\big(\bA_y(p(y,\eta)) +
q(y,\eta)p(y,\eta)\big).
\end{equation*}
However,
\begin{multline*}
\tilde{\omega}(x[\eta]_y)\bA_y(p(y,\eta)) =\\
[\eta] \kappa_{[\eta]} [\varphi(y)\psi(y)\bA_y\omega(x[\eta]_y/[\eta]) \tau_1 , \cdots , \varphi(y)\psi(y)\bA_y\omega(x[\eta]_y/[\eta]) \tau_N],
\end{multline*}
which gives
\begin{equation*}
\tilde{\omega}(x[\eta]_y)\bA_y(p(y,\eta)) \in S^1(U\times\R^q;\C^N,\K_{\infty}^{\infty,\infty})
\end{equation*}
because $\bA_y\tau_j = 0$. On the other hand,
\begin{equation*}
\tilde{\omega}(x[\eta]_y)qp \in S^1(U\times\R^q;\C^N,\K_{\infty}^{\infty,-1/2+\delta_0}).
\end{equation*}
Therefore
\begin{equation}\label{topOrder}
ap \in S^1(U\times\R^q;\C^N,\K_{\infty}^{\infty,-1/2+\delta_0}).
\end{equation}
Combining \eqref{lotOrder} and \eqref{topOrder} gives
\begin{equation*}
\sum_{\alpha}\frac{1}{\alpha !}\partial_{\eta}^{\alpha}a\, D_y^{\alpha} p \in 
S^1(U\times\R^q;\C^N,\K_{\infty}^{\infty,-1/2+\delta_0})
\end{equation*}
which proves that
\begin{equation*}
A'\circ \op (p) : H^1_{\comp}(U;\C^N) \to x^{-1/2}H^{\infty}_e(W;F),
\end{equation*}
and consequently, in conjunction with \eqref{lot}, that 
\begin{equation}
A'\circ {}_\psi \Ext_\varphi \circ \op(\mathfrak k) : H^{1}_{\comp}(U;\C^N) \to x^{-1/2}H_e^{\infty}(W;F)
\end{equation}
as asserted.
\end{proof}

\begin{proposition}\label{PoissonOperatorProperties}
The operator $\Ext: C^{-\infty}(\Y;\trb) \to C^\infty(\open \M;E)$ has the following mapping properties.
\begin{enumerate}
\item $\Ext u - \preP u \in \dot{C}^{\infty}(\M;E)$ for all $u \in C^{\infty}(\Y;\trb)$.
\item If $\varphi \in C^{\infty}(\Y)$ and $\psi \in \Ring$ are such that $\psi$ vanishes to infinite order on $\supp(\wp^*\varphi)$, then $\psi \Ext \varphi$ maps $C^{-\infty}(\Y;\trb)$ into $\dot{C}^{\infty}(\M;E)$.
\item $\gamma_A\circ\Ext :C^\infty(\Y;\trb)\to C^\infty(\Y;\trb)$ is the identity operator.
\item $\Ext : L^2(\Y;\trb) \to x^{-1/2+\delta_0}H^{\infty}_{e}(\M;E)$ for some $\delta_0 > 0$.
\item $A\circ \Ext : H^{1-\gen}(\Y;\trb) \to x^{-1/2}H^\infty_e(\M;F)$, where $\gen$ is the bundle homomorphism $x\partial_x +1/2:\trb\to\trb$.
\end{enumerate}
In particular, $\Ext : H^{1-\gen}(\Y;\trb) \to \Dom_{\max}(A)$ continuously, and $\gamma_A\circ\Ext = \Id$ on $H^{1-\gen}(\Y;\trb)$. 
\end{proposition}
\begin{proof}
Statements (1) and (2) follow from Lemma~\ref{scalarlocallylikecutoff} (see also \eqref{ExtOnSmooth} for (1)) while statement (3) is just \eqref{ExtisGood}. Statements (4) and (5) follow, respectively, from Lemmas~\ref{MappingFromL2} and \ref{MappingintoDmax}. By (3) and (4) we have $\Ext : H^{1-\gen}(\Y;\trb) \to \Dom_{\max}(A)$. The composition
\begin{equation*}
\gamma_A\circ \Ext : H^{1-\gen}(\Y;\trb) \to C^{-\infty}(\Y;\trb)
\end{equation*}
is continuous, and by \eqref{ExtisGood} we have $(\gamma_A\circ\Ext)(u)=u$ if $u \in C^{\infty}(\Y;\trb)$. By continuity, $\gamma_A\circ\Ext$ must coincide with the inclusion map $H^{1-\gen}(\Y;\trb) \embed C^{-\infty}(\Y;\trb)$ on $H^{1-\gen}(\Y;\trb)$ which proves the claim.
\end{proof}

The following proposition determines the normal family of $\Ext$ (compare with \eqref{NormalFamilyAsLimit}). 

\begin{proposition}
Let $g\in C^\infty(\Y)$ be real valued and $u\in C^\infty(\Y;\trb)$ with $dg\ne0$ on $\supp u$. Then
\begin{equation}\label{NormalFamilyOfExtAsLimit}
\Ext_\wedge (dg)u= \lim_{\varrho\to \infty} \kappa^{-1}_\varrho e^{-\im \varrho \wp_\wedge^*g} \Phi^* \Ext \Phi_* e^{\im \varrho \wp_\wedge^*g} \kappa_\varrho u,
\end{equation}
exists and is equal to $\omega(x|dg|)\mathfrak i u$.
\end{proposition}

The map $\mathfrak i$ was defined in \eqref{CanonicalMap}. The normal family of $\Ext$ on $T^*\Y\minus 0$ is defined by \eqref{NormalFamilyOfExtAsLimit}:
\begin{gather*}
\Ext_{\wedge}(\pmb\eta) : \pi_{\Y}^*\trb_y \to C^{\infty}(\Z^{\wedge}_y;E_{\Z^{\wedge}_y}), \\
\Ext_{\wedge}(\pmb\eta)\tau = \omega(x|\pmb\eta|)\tau(x,z),
\end{gather*}
where $y = \pi_{\Y}\pmb\eta$. The normal family is twisted homogeneous of degree 0 in the sense that
\begin{equation*}
\Ext_{\wedge}(\varrho\pmb\eta) = \kappa_{\varrho}\Ext_{\wedge}(\pmb\eta)\kappa_{\varrho}^{-1}
\end{equation*}
for all $\pmb\eta \in T^*\Y\minus 0$ and $\varrho  > 0$.

\begin{proof}
Using the definition of $\Ext$ we see that we only need to show that the limit
\begin{equation*}
\lim_{\varrho\to \infty} \kappa^{-1}_\varrho e^{-\im \varrho \wp_\wedge^*g} {}_\psi\Ext_\varphi e^{\im \varrho \wp_\wedge^*g} \kappa_\varrho u
\end{equation*}
exists. We may assume that $dg\ne 0$ on $U$. So suppose $\tau_1,\dots,\tau_N$ is a smooth frame of $\trb$ as usual and $u=\sum u^\mu\tau_\mu$ with smooth coefficients $u^\mu$. Using the notation \eqref{ActionOnFrame} we have
\begin{equation*}
\kappa_\varrho u(y) = \sum_\nu c^\nu_\mu(y,1/\varrho) u^\mu(y)\tau_\nu(y).
\end{equation*}
Using \eqref{ActionOnFrame} again we get that $(2\pi)^q\kappa_\varrho^{-1} e^{-\im \varrho \wp_\wedge^*g} {}_\psi\Ext_\varphi e^{\im \varrho \wp_\wedge^*g} \kappa_\varrho u$ is the sum of
\begin{equation*}
\iint e^{\im(y-y')\cdot \eta-\im\varrho(g(y)- g(y'))}\psi(y)\varphi(y')\omega(x[\eta]_y/\varrho) c^\lambda_\nu(y,\varrho)c^\nu_\mu(y',1/\varrho)u^\mu(y')\,dy'\,d\eta\,\tau_\lambda(y)
\end{equation*}
over all indices $\mu$, $\nu$, and $\lambda$. Because of \eqref{cOfNegOrder} and \eqref{ctildeOfSmallOrder} the entries $C^\lambda_\mu(y,y',\varrho)$ of $c(y,\varrho)c(y',\varrho)^{-1}$ are symbols of various orders $\leq 1 -2\delta_0$. Obviously $C^\lambda_\mu(y,y,\varrho)=\delta^\lambda_\mu$. A change of variables in the integral gives
\begin{equation*}
\varrho^{-q}\iint e^{\im\varrho[(y-y')\cdot \eta-(g(y)- g(y'))]}\psi(y)\varphi(y')\omega(x[\varrho\eta]_y/\varrho) C^\lambda_\mu(y,y',\varrho)u^\mu(y')\,dy'\,d\eta\,\tau_\lambda(y).
\end{equation*}
For $y\in U$ and $x>0$ the critical point of the phase is nondegenerate and occurs at $y'=y$, $\eta=\nabla g(y')$. By stationary phase the integral is equal to
\begin{equation*}
(2\pi)^q \psi(y)\varphi(y)\omega(x[\varrho\nabla g(y)]_y/\varrho)u^\lambda(y)\tau_\lambda(y)
\end{equation*}
modulo lower order terms. This gives
\begin{equation*}
\lim_{\varrho\to \infty} \kappa^{-1}_\varrho e^{-\im \varrho \wp_\wedge^*g} {}_\psi\Ext_\varphi e^{\im \varrho \wp_\wedge^*g} \kappa_\varrho u=\psi(y)\varphi(y)\omega(x[dg]_y) u
\end{equation*}
since $[\varrho\nabla g(y)]_y/\varrho = |dg(y)|$ for large $\varrho$ because $dg(y)\ne 0$.
\end{proof}

\begin{remark}\label{PerturbationOfGreen}
The following observation will be important in the proof of Theorem~\ref{FredholmTheorem}. Let $0 < \delta < 1$, and let $K \in \Psi^{0}_{1,\delta}(\Y;(\trb,0),(\trb,-\gen))$ have twisted homogeneous principal symbol $\sym(K)(\pmb\eta)$ on $T^*\Y\minus 0$ of order $0$, i.e.,
\begin{equation*}
\sym(K)(\varrho\pmb\eta) = \kappa_{\varrho}\sym(K)(\pmb\eta)
\end{equation*}
for $\varrho > 0$, see \cite{KrMe12b}. Define 
\begin{equation}\label{GreenPerturbPrinc}
G_{\wedge}(\pmb\eta) = A_{\wedge}(\pmb\eta)\circ\Ext_{\wedge}(\pmb\eta)\circ\sym(K)(\pmb\eta)
\end{equation}
Then $G_{\wedge}(\varrho\pmb\eta) = \varrho^1\kappa_{\varrho}G_{\wedge}(\pmb\eta)$ in view of the homogeneities of the normal families $A_{\wedge}(\pmb\eta)$ and $\Ext_{\wedge}(\pmb\eta)$ and of the principal symbol $\sym(K)(\pmb\eta)$ of $K$. By the proof of Lemma~\ref{MappingintoDmax}, $G_{\wedge}(\pmb\eta)$ satisfies the mapping properties for principal parts of Green symbols with respect to the weights $(\gamma_1,\gamma_2)=(\frac{1}{2},-\frac{1}{2})$, see \eqref{Greensymbol}. Therefore, there is a Green operator $G$ of order $1$ with respect to these weights whose normal family is given by \eqref{GreenPerturbPrinc}. By the proof of Lemma~\ref{MappingintoDmax} again and the material reviewed in Appendices~\ref{app-PseudosOpValued} and \ref{GreenOperators}, we can deduce that
\begin{equation}\label{GreenPerturb}
A\circ\Ext\circ K - G : H^1(\Y;\trb) \to x^{-1/2+\varepsilon}H^{\infty}_e(\M;F),
\end{equation}
which shows, in particular, that $A\circ\Ext\circ K - G : H^1(\Y;\trb) \to x^{-1/2}L^2_b(\M;F)$ is compact. 
\end{remark}

\section{The space $H^1_{\trb}$}\label{sec-H1A}

We continue with the assumptions and notational conventions of the previous sections. In particular, the normal family $A_\wedge(\pmb\eta)$ is injective on its minimal domain, \eqref{NormalInjectivity} with $\gamma=1/2$.

Define
\begin{equation*}
C^{\infty}_{\trb} (\M;E) = \dot{C}^{\infty}(\M;E) + \preP(C^{\infty}(\Y;\trb)).
\end{equation*}
We have $C^{\infty}_{\trb}(\M;E) \subset x^{-1/2}L^2_b(\M;E)$, and the operator $A$ induces a map
\begin{equation*}
A : C^{\infty}_{\trb}(\M;E) \to x^{-1/2}L^2_b(\M;F).
\end{equation*}
Consequently, $C^{\infty}_{\trb}(\M;E) \subset \Dom_{\max}(A)$, and the trace map $\gamma_A$ defines an operator
\begin{equation*}
\gamma_A : C^{\infty}_{\trb}(\M;E) \to C^{\infty}(\Y;\trb).
\end{equation*}
Recall that
\begin{equation}\label{PsudoDminSubsetKerTrace}
\dot{C}^{\infty}(\M;E) \subset \ker\gamma_A
\end{equation}
because of Proposition~\ref{DminSubsetKerTrace}.

\begin{definition}\label{H1ADef}
Let $H^1_{\trb}(\M;E)$ be the completion of $C^{\infty}_{\trb}(\M;E)$ with respect to the norm $\|\cdot\|_{H^1_{\trb}}$ given by
\begin{equation*}
\|u\|_{H^1_{\trb}}^2 = \|u\|_A^2 + \|\gamma_A u \|^2_{H^{1-\gen}(\Y;\trb)}.
\end{equation*}
\end{definition}

\begin{proposition}\label{H1ARepresentation}
Let $\Ext:C^{-\infty}(\Y;\trb)\to C^\infty(\open\M;E)$ be an extension operator associated to $A$ and $\iota:\Dom_{\min}(A) \hookrightarrow \Dom_{\max}(A)$ the inclusion map. The map
\begin{equation*}
\begin{bmatrix} \iota & \Ext \end{bmatrix} :
\begin{matrix} \Dom_{\min}(A) \\ \oplus \\ H^{1-\gen}(\Y;\trb) \end{matrix} \to \Dom_{\max}(A),
\end{equation*}
induces a topological isomorphism
\begin{equation*}
\begin{bmatrix} \iota & \Ext \end{bmatrix} :
\begin{matrix} \Dom_{\min}(A) \\ \oplus \\ H^{1-\gen}(\Y;\trb) \end{matrix} \to H^1_{\trb}(\M;E).
\end{equation*}
The inverse is
\begin{equation*}
\begin{bmatrix} \iota & \Ext \end{bmatrix}^{-1}u = \begin{bmatrix} u - (\Ext\circ\gamma_A)u \\ \gamma_A u \end{bmatrix}
\text{ for } u \in  H^1_{\trb}(\M;E).
\end{equation*}
\end{proposition}
\begin{proof}
We first show that the map $\begin{bmatrix} \iota & \Ext \end{bmatrix}$ induces an algebraic isomorphism
\begin{equation*}
\begin{bmatrix} \iota & \Ext \end{bmatrix} :
\begin{matrix} \dot{C}^{\infty}(\M;E) \\ \oplus \\ C^{\infty}(\Y;\trb) \end{matrix} \to
C^{\infty}_{\trb}(\M;E).
\end{equation*}
This operator is well-defined because for $u \in \dot{C}^{\infty}(\M;E)$ and $v \in C^{\infty}(\Y;\trb)$ we have
\begin{equation*}
u + \Ext v =(u+\Ext v-\preP v)+\preP v
\end{equation*}
with $u+\Ext v-\preP v\in  \dot{C}^{\infty}(\M;E)$ since by Proposition~\ref{PoissonOperatorProperties}, $\Ext v - \preP v \in \dot{C}^{\infty}(\M;E)$. The operator is also injective because if $u + \Ext v = 0$ with $u$ and $v$ as before, then $u = -\Ext v$, and therefore $\gamma_A u = - (\gamma_A\circ\Ext)v = -v$ by the same proposition, so $v=0$ since $\gamma_A u = 0$ by \eqref{PsudoDminSubsetKerTrace}. Consequently $u = -\Ext v = 0$. Finally, the operator is surjective because if $u + \preP v \in C^{\infty}_{\trb}(\M;E)$ with $u$ and $v$ as above, then
\begin{equation*}
u + \preP v = \begin{bmatrix} \iota & \Ext \end{bmatrix}\begin{bmatrix} u + \bigl(\preP v - \Ext v\bigr) \\ v \end{bmatrix}.
\end{equation*}
This also shows that the inverse $\begin{bmatrix} \iota & \Ext \end{bmatrix}^{-1}$ has the stated form.

By density of the spaces of smooth functions it remains to show that both
\begin{equation*}
\begin{bmatrix} \iota & \Ext \end{bmatrix} :
\begin{matrix} \big( \dot{C}^{\infty}(\M;E),\|\cdot\|_A\big)
\\
\oplus
\\ \big(C^{\infty}(\Y;\trb),\|\cdot\|_{H^{1-\gen}}\big) \end{matrix}
\to
\big(C^{\infty}_{\trb}(\M;E),\|\cdot\|_{H^1_{\trb}}\big).
\end{equation*}
and its inverse
\begin{equation*}
\begin{bmatrix} \Id - \Ext\circ\gamma_A \\ \gamma_A \end{bmatrix} :
\big(C^{\infty}_{\trb}(\M;E),\|\cdot\|_{H^1_{\trb}}\big) \to
\begin{matrix} \big( \dot{C}^{\infty}(\M;E),\|\cdot\|_A\big)
\\
\oplus
\\ \big(C^{\infty}(\Y;\trb),\|\cdot\|_{H^{1-\gen}}\big) \end{matrix}
\end{equation*}
are continuous.

Let $u \in  \dot{C}^{\infty}(\M;E)$ and $v \in C^{\infty}(\Y;\trb)$. By Proposition~\ref{PoissonOperatorProperties}, the operator $\Ext : H^{1-g}(\Y;\trb) \to \Dom_{\max}(A)$ is continuous. Consequently,
\begin{align*}
\|u + \Ext v\|_{A} &\leq \|u\|_A + \|\Ext v\|_A \leq \|u\|_A + \|\Ext\|\|v\|_{H^{1-g}} \\
&\leq (1 + \|\Ext\|)(\|u\|_A + \|v\|_{H^{1-g}}),
\end{align*}
while
\begin{equation*}
\|\gamma_A(u + \Ext v)\|_{H^{1-g}} = \|v\|_{H^{1-g}} \leq \|u\|_A + \|v\|_{H^{1-g}}.
\end{equation*}
This shows the continuity of $\begin{bmatrix} \iota & \Ext \end{bmatrix}$ with respect to the indicated norms.

Now let $u \in \big(C^{\infty}_{\trb}(\M;E),\|\cdot\|_{H^1_{\trb}}\big)$. Then the definition of $\|\cdot\|_{H^1_{\trb}}$ gives
\begin{equation*}
\|\gamma_A u\|_{H^{1-g}} \leq \|u\|_{H^1_{\trb}},
\end{equation*}
while
\begin{equation*}
\|u - \Ext\circ\gamma_A u\|_A \leq \|u\|_A + \|\Ext\circ\gamma_A u\|_A \leq \|u\|_A + \|\Ext\|\|\gamma_A u\|_{H^{1-g}} \leq (1 + \|\Ext\|)\|u\|_{H^1_{\trb}}.
\end{equation*}
This shows the continuity of the inverse $\begin{bmatrix} \iota & \Ext \end{bmatrix}^{-1}$ and finishes the proof of the proposition.
\end{proof}

\begin{corollary}\label{compactembedding}
The embedding $H^1_{\trb}(\M;E) \embed x^{-1/2}L^2_b(\M;E)$ is compact.
\end{corollary}
\begin{proof}
By Proposition~\ref{H1ARepresentation} the map
\begin{equation*}
\begin{bmatrix} \iota & \Ext \end{bmatrix}^{-1} : H^1_{\trb}(\M;E) \to
\begin{matrix} \Dom_{\min}(A) \\ \oplus \\ H^{1-g}(\Y;\trb) \end{matrix}
\end{equation*}
is continuous. We have $\Dom_{\min}(A) = x^{1/2}H^1_e(\M;E)$ by \cite{GiKrMe10}, and the embedding $x^{1/2}H^1_e(\M;E) \hookrightarrow x^{-1/2}L^2_b(\M;E)$ is compact, see \cite[Propositon 3.29]{Mazz91}. By Proposition~\ref{PoissonOperatorProperties}, the operator $\Ext : L^2(\Y;\trb) \to x^{-1/2+\delta_0}H^{\infty}_e(\M;E)$ is continuous (see \eqref{delta0} for the definition of $\delta_0$), and we have $H^{1-g}(\Y;\trb) \hookrightarrow L^2(\Y;\trb)$. Because $x^{-1/2+\delta_0}H^{m}_e(\M;E) \hookrightarrow x^{-1/2}L^2_b(\M;E)$ is compact for $m > 0$ again by \cite{Mazz91}, we obtain that $\Ext : H^{1-g}(\Y;\trb) \to x^{-1/2}L^2_b(\M;E)$ is compact. Consequently,
\begin{equation*}
\begin{bmatrix} \iota & \Ext \end{bmatrix} :
\begin{matrix} \Dom_{\min}(A) \\ \oplus \\ H^{1-g}(\Y;\trb) \end{matrix} \to x^{-1/2}L^2_b(\M;E)
\end{equation*}
is compact which proves the compactness of the embedding
\begin{equation*}
\begin{bmatrix} \iota & \Ext \end{bmatrix}\begin{bmatrix} \iota & \Ext \end{bmatrix}^{-1} :
H^1_{\trb}(\M;E) \embed x^{-1/2}L^2_b(\M;E)
\end{equation*}
as stated.
\end{proof}

\begin{remark}
For a regular elliptic operator $A \in \Diff^1(\M;E,F)$ on a compact manifold with boundary acting as an unbounded operator $L^2(\M;E) \to L^2(\M;F)$, we obtain that the space $H^1_{\trb}(\M;E)$ coincides with the standard Sobolev space $H^1(\M;E)$ on $\M$.
\end{remark}

\section{Boundary value problems}\label{sec-BVP}

We are now ready to state our main theorem. We begin by reminding the reader that $A \in x^{-1}\Diff^1_e(\M;E,F)$ is $w$-elliptic, i.e.,
\begin{equation}\label{Awelliptic}
\wsym(A) \text{ is invertible everywhere on } \wT^*\M \minus 0,
\end{equation}
that
\begin{equation}\label{speceAssumption}
\spec_{e}(A) \cap \bigl(\Y\times\{\sigma \in \C;\; \Im(\sigma) = \pm 1/2\}\bigr) = \emptyset.
\end{equation}
and that
\begin{equation}\label{normalAssumption}
\begin{gathered}
A_{\wedge}(\pmb\eta) :\Dom_{\min}(A_{\wedge}(\pmb\eta)) \subset x^{-1/2}L^2_b(\M;E) \to x^{-1/2}L^2_b(\M;F) \text{ is injective}, \\
A_{\wedge}(\pmb\eta) :\Dom_{\max}(A_{\wedge}(\pmb\eta)) \subset x^{-1/2}L^2_b(\M;E) \to x^{-1/2}L^2_b(\M;F) \text{ is surjective}
\end{gathered}
\end{equation}
for all $\pmb\eta \in T^*\Y\minus 0$. Recall, see \eqref{DomainsDependOnYOnly}, that both $\Dom_{\min}(A_{\wedge}(\pmb\eta))$ and $\Dom_{\max}(A_{\wedge}(\pmb\eta))$ depend on $\pmb\eta \in T^*\Y$ only through $y = \pi_{\Y}\pmb\eta$ in the first order case, and the minimal domain is explicitly given by \eqref{DomMinAWedge} (with $m=1$ and $\gamma=1/2$ there). The trace bundle $\trb\to\Y$ was reviewed in Section~\ref{sec-TraceBundles}; it comes equipped with a smooth endomorphism $\gen=x\partial_x+1/2:\trb\to\trb$, see Proposition~\ref{xdxAction}.

Let $\KK\to T^*\Y\minus 0$ be the vector bundle whose fiber at $\pmb\eta$ is the kernel of
\begin{equation*}
A_{\wedge}(\pmb\eta) : \Dom_{\max}(A_{\wedge}(\pmb\eta))\subset x^{-1/2}L^2_b(\Z^\wedge_y;E_{\Z^\wedge_y}) \to x^{-1/2}L^2_b(\Z^\wedge_y;F_{\Z^\wedge_y}).
\end{equation*}
The image $\gamma_{A_\wedge}(\KK)=\kerb\subset \pi^*_\Y\trb$ is a subbundle, see Theorem~\ref{KernelBundle}.

Let $\target$ be a vector bundle over $\Y$ together with a smooth endomorphism $\agen:\target\to\target$ and 
\begin{equation*}
B : C^{\infty}(\Y;\trb) \to C^{\infty}(\Y;\target)
\end{equation*}
a pseudodifferential operator of order $\mu \in \R$ in the class $\Psi^{\mu}_{1,\delta}(\Y;(\trb,-\gen),(\target,{\agen}))$ for a fixed $0 < \delta < 1$. These classes of pseudodifferential operators were defined in \cite{KrMe12b}. Assume further that $B$ has twisted homogeneous principal symbol $\sym(B)$ on $T^*\Y\minus 0$, which means that
\begin{equation*}
\sym(B)(\varrho\pmb\eta) = \varrho^{\mu}\varrho^{-{\agen_y}}\sym(B)(\pmb\eta)\varrho^{-\gen_y}
\end{equation*}
for all $\varrho > 0$, where $y = \pi_{\Y}\pmb\eta$. Here $\varrho^{-\gen_y}$ is of course the action $\kappa_{\varrho}^{-1}$ on $\trb_y$. Finally, let
\begin{equation*}
\Pi = \Pi^2 \in \Psi^0_{1,\delta}(\Y;(\target,{\agen}),(\target,{\agen}))
\end{equation*}
be a projection and assume that $\Pi$ has twisted homogeneous principal symbol $\sym(\Pi)$. Because $\sym(\Pi) = \sym(\Pi)^2$, the range of $\sym(\Pi)$ is a subbundle $\target_{\Pi}$ of $\pi_{\Y}^*\target$ over $T^*\Y\minus 0$.

Let $\gamma_A : H^1_{\trb}(\M;E) \to H^{1-\gen}(\Y;\trb)$ be the trace map associated with $A$. We consider the boundary value problem
\begin{equation}\label{ABVPEq}
\left\{\begin{aligned}
&Au = f \in x^{-1/2}L^2_b(\M;F)\\
&\Pi B(\gamma_A u) = g \in \Pi H^{1-\mu+{\agen}}(\Y;\target)
\end{aligned}
\right.
\end{equation}
and seek solutions $u \in H^1_{\trb}(\M;E)$. Here $\Pi H^{1-\mu+{\agen}}(\Y;\target)$ is the range of the projection $\Pi \in \L(H^{1-\mu+{\agen}}(\Y;\target))$, a closed subspace of $H^{1-\mu+{\agen}}(\Y;\target)$.

The ellipticity assumption about this boundary value problem is that the principal symbol of the boundary condition induces a vector bundle isomorphism
\begin{equation}\label{BoundaryConditionKernelIso}
\sym(\Pi B) : \kerb \overset{\cong}{\longrightarrow} \target_{\Pi}
\quad \text{on } T^*\Y\minus 0.
\end{equation}
The main theorem is the following.

\begin{theorem}\label{FredholmTheorem}
Under the assumptions \eqref{Awelliptic}, \eqref{speceAssumption}, \eqref{normalAssumption}, \eqref{BoundaryConditionKernelIso}, the boundary value problem \eqref{ABVPEq} is well-posed, i.e., the operator
\begin{equation}\label{ABVP}
\begin{bmatrix} A \\ \Pi B \gamma_A \end{bmatrix} : H^1_{\trb}(\M;E) \to
\begin{array}{c} x^{-1/2}L^2_b(\M;F) \\ \oplus \\ \Pi H^{1-\mu+{\agen}}(\Y;\target) \end{array}
\end{equation}
is a Fredholm operator.
\end{theorem}

\begin{example}\label{APSBoundaryValueProblem}
Consider any operator $A \in x^{-1}\Diff^1_e(\M;E,F)$ that satisfies \eqref{Awelliptic}, \eqref{speceAssumption}, and \eqref{normalAssumption}. Let $S^*\Y$ be the cosphere bundle with respect to some Riemannian metric on $\Y$. Pick a projection $q = q^2 \in \End(\pi_{\Y}^*\trb)$ onto the subbundle $\kerb \subset \pi_{\Y}^*\trb$ restricted to $S^*\Y$, and extend $q$ by twisted homogeneity of degree zero to $T^*\Y\minus 0$ with respect to the action $-\gen$ on $\trb$. Then $q$ is a projection onto $\kerb$ everywhere on $T^*\Y\minus 0$, and by \cite[Lemma~8.1]{KrMe12b} there exists a projection $\Pi = \Pi^2 \in \Psi^0_{1,\delta}(\Y;(\trb,-\gen),(\trb,-\gen))$ with $\sym(\Pi) = q$. By construction, \eqref{BoundaryConditionKernelIso} holds for $\Pi$ (and $B = \Id$), and therefore
\begin{equation}\label{generalizedAPS}
\begin{bmatrix} A \\ \Pi \gamma_A \end{bmatrix} : H^1_{\trb}(\M;E) \to
\begin{array}{c} x^{-1/2}L^2_b(\M;F) \\ \oplus \\ \Pi H^{1-\gen}(\Y;\trb) \end{array}
\end{equation}
is a Fredholm operator by Theorem~\ref{FredholmTheorem}.

The boundary value problem \eqref{generalizedAPS} is a generalization of the Atiyah-Patodi-Singer boundary value problem in the classical case, see \cite{APS}. The projection $\Pi$ could also be regarded as a generalization of the Calder{\'o}n-Seeley projection \cite[Chapter~VI]{Seeley1969} to our more general situation (this point of view applies in particular to the level of principal symbols which is further underscored by Lemma~\ref{KplusKstar} below). 

This example shows that it is always possible to pose generalized Atiyah-Patodi-Singer boundary conditions to obtain a well-posed boundary value problem for any operator $A \in x^{-1}\Diff^1_e(\M;E,F)$ that satisfies \eqref{Awelliptic}, \eqref{speceAssumption}, and \eqref{normalAssumption}. In the classical theory, one is often just interested in the homogeneous boundary value problem. This corresponds here to considering the unbounded operator
\begin{equation*}
A : \Dom(A) \subset x^{-1/2}L^2_b(\M;E) \to x^{-1/2}L^2_b(\M;F)
\end{equation*}
with domain
\begin{equation*}
\Dom(A) = \{u \in H^1_{\trb}(\M;E) : \Pi(\gamma_A u) = 0\}.
\end{equation*}
Observe that $\Dom_{\min}(A) \subset \Dom(A) \subset \Dom_{\max}(A)$ because $\gamma_A$ vanishes on $\Dom_{\min}(A)$ by Proposition~\ref{DminSubsetKerTrace}.
The operator $A$ with domain $\Dom(A)$ is closed, densely defined, and Fredholm by Theorem~\ref{FredholmTheorem} (closedness is a consequence of Fredholmness by elementary functional analytic arguments, see for example \cite[Proposition~4.4]{GiKrMe10} for a proof).
\end{example}

The case when the projection $\Pi$ in \eqref{ABVPEq} and \eqref{ABVP} can be taken to be the identity map corresponds classically to boundary value problems satisfying the \emph{Shapiro-Lopatinskii condition}. In this case, the ellipticity condition \eqref{BoundaryConditionKernelIso} becomes
\begin{equation}\label{BoundaryConditionKernelIso1}
\sym(B) : \kerb \overset{\cong}{\longrightarrow} \pi^*_{\Y}\target
 \text{ on } T^*\Y\minus 0.
\end{equation}
By twisted homogeneity of $\sym(B)$, the invertibility of \eqref{BoundaryConditionKernelIso1} everywhere on $T^*\Y\minus 0$ is equivalent to the invertibility just on the cosphere bundle $S^*\Y$ with respect to some Riemannian metric on $\Y$. The bundle $\kerb \to S^*\Y$ is therefore isomorphic to the pull-back $\pi_{\Y}^*\target$ of the bundle $\target \to \Y$ via $\sym(B)$, i.e.,
\begin{equation}\label{AtBott}
\kerb \in \pi_{\Y}^*\Vect(\Y) \text{ on } S^*\Y
\end{equation}
up to vector bundle isomorphism. Conversely, if \eqref{AtBott} holds (up to isomorphism), then there exists a twisted homogeneous section of $\Hom(\pi_{\Y}^*\trb,\pi_{\Y}^*\target)$ on $T^*\Y\minus 0$ that realizes \eqref{BoundaryConditionKernelIso1}. This can be seen by composing a bundle isomorphism $\kerb \cong \pi^*_{\Y}\target$ with a projection $\pi_\Y^*\trb \to \kerb$ on $S^*\Y$ and extending the resulting morphism by twisted homogeneity. This and the calculus of operators twisted by endomorphisms developed in \cite{KrMe12b} show that a boundary value problem \eqref{ABVPEq} with $\Pi = \Id$ that satisfies the ellipticity condition \eqref{BoundaryConditionKernelIso1} exists if and only if \eqref{AtBott} holds.
Condition \eqref{AtBott} is a topological obstruction for the operator $A$ to admit boundary conditions $B$ that are elliptic in the sense that the Shapiro-Lopatinskii condition \eqref{BoundaryConditionKernelIso1} is satisfied. In the context of the classical $L^2$-theory of regular elliptic boundary value problems, this topological obstruction was identified by Atiyah and Bott \cite{AtiyahBott}. In particular, as the classical theory shows, one can in general not expect that \eqref{AtBott} holds for the bundle $\kerb$. In fact, the $K$-theory class
\begin{equation*}
[\kerb] = \Ind_{S^*\Y}(A_{\wedge} : \Dom_{\max}(A_{\wedge}) \to x^{-1/2}L^2_b) \in K(S^*\Y)
\end{equation*}
($\Dom_{\max}(A_\wedge)$ was defined in \eqref{DmaxBundle}) does in general not belong to the subgroup $\pi_{\Y}^*K(\Y)$ (as is the case, for example, for the $\overline{\partial}$-operator on the unit disk, and more generally Dirac operators on even-dimensional manifolds). However, as discussed in Example \ref{APSBoundaryValueProblem}, it is always possible to set up a well-posed boundary value problem subject to generalized Atiyah-Patodi-Singer boundary conditions.

\subsubsection*{Proof of Theorem~\ref{FredholmTheorem}}

We begin the proof with several reductions, each substantiated further below:

\begin{enumerate}
\item \label{Reduct_1} We may assume that the order $\mu$ of the boundary operator $B$ is zero.
\item \label{Reduct_2} It suffices to consider the case that $E = F$, i.e., $A \in x^{-1}\Diff^1_e(\M;E,E)$. Moreover, it suffices to consider operators $A = A^{\star}$ that are symmetric on $C_c^{\infty}(\open\M;E)$.
\item \label{Reduct_3} It suffices to consider the case that $A$ has vanishing Atiyah-Bott obstruction. More precisely, one can assume that there exists a vector bundle $\target_0 \in \Vect(\Y)$ such that $[\K] = [\pi_{\Y}^*\target_0]$ in $K(S^*\Y)$.
\end{enumerate}

Reduction \eqref{Reduct_1} is possible because
\begin{align*}
B \in \Psi^{\mu}_{1,\delta}(\Y;(\trb,-\gen),(\target,{\agen})) &= \Psi^{0}_{1,\delta}(\Y;(\trb,-\gen),(\target,{\agen}-\mu)), \\
\Pi \in \Psi^{0}_{1,\delta}(\Y;(\target,{\agen}),(\target,{\agen})) &= \Psi^{0}_{1,\delta}(\Y;(\target,{\agen} - \mu),(\target,{\agen}-\mu)).
\end{align*}
Consequently, by changing from ${\agen}$ to ${\agen}-\mu$, we may without loss of generality assume in the sequel that $\mu = 0$, i.e., $B \in \Psi^{0}_{1,\delta}(\Y;(\trb,-\gen),(\target,{\agen}))$.

Reduction \eqref{Reduct_2} is of little relevance for the proof and provides merely notational convenience. However, reduction \eqref{Reduct_3} is of primary importance. We need the following lemma to establish both.

\begin{lemma}\label{KplusKstar}
Suppose $A \in x^{-1}\Diff^1_e(\M;E,F)$ satisfies \eqref{Awelliptic}, \eqref{speceAssumption}, and \eqref{normalAssumption}. Then so does the formal adjoint $A^{\star} \in x^{-1}\Diff^1_e(\M;F,E)$. Let $\trb^{\star}$ be the trace bundle associated with $A^{\star}$ acting in $x^{-1/2}L^2_b$, and let $\kerb^{\star} \subset \pi_{\Y}^*\trb^{\star}$ be the subbundle obtained by passing from the kernel bundle $\KK^{\star}$ of the normal family $A_{\wedge}^{\star}(\pmb\eta) : \Dom_{\max}(A^{\star}_{\wedge}(\pmb\eta)) \to x^{-1/2}L^2_b$ to $\pi_{\Y}^*\trb^{\star}$ by fiberwise applying $\gamma_{A^{\star}_{\wedge}(\pmb\eta)}$, see Theorem~\ref{KernelBundle}. Both $\kerb$ and $\kerb^{\star}$ are restricted to $S^*\Y$.
Then
\begin{equation}\label{kerbkerbstar}
[\kerb] + [\kerb^{\star}] = [\pi_{\Y}^*\trb] \in K(S^*\Y).
\end{equation}
\end{lemma}
\begin{proof}[Proof of Lemma~\ref{KplusKstar}]
That $A^{\star}$ satisfies \eqref{Awelliptic}, \eqref{speceAssumption}, and \eqref{normalAssumption} is clear. We only have to prove \eqref{kerbkerbstar}. Consider the Fredholm functions $A_{\wedge,\max}(\pmb\eta) : \Dom_{\max}(A_{\wedge}(\pmb\eta)) \to x^{-1/2}L^2_b$ and $A_{\wedge,\max}^{\star}(\pmb\eta) : \Dom_{\max}(A^{\star}_{\wedge}(\pmb\eta)) \to x^{-1/2}L^2_b$ on $S^*\Y$. By assumption we have
\begin{align*}
[\kerb] &= \Ind_{S^*\Y}(A_{\wedge,\max}) \in K(S^*\Y), \\
[\kerb^{\star}] &= \Ind_{S^*\Y}(A^{\star}_{\wedge,\max}) \in K(S^*\Y).
\end{align*}
Since $A^{\star}_{\wedge,\max}(\pmb\eta) : \Dom_{\max}(A^{\star}_{\wedge}(\pmb\eta)) \subset  x^{-1/2}L^2_b \to x^{-1/2}L^2_b$ is the adjoint of the Fredholm function $A_{\wedge,\min}(\pmb\eta) : \Dom_{\min}(A_{\wedge}(\pmb\eta))\subset  x^{-1/2}L^2_b \to x^{-1/2}L^2_b$ we get that
\begin{equation*}
\Ind_{S^*\Y}(A^{\star}_{\wedge,\max}) = -\Ind_{S^*\Y}(A_{\wedge,\min}).
\end{equation*}
On the other hand, $A_{\wedge,\min} = A_{\wedge,\max} \circ \iota$, where $\iota(\pmb\eta) : \Dom_{\min}(A_{\wedge}(\pmb\eta)) \to \Dom_{\max}(A_{\wedge}(\pmb\eta))$ is the inclusion map. Additivity of the index gives
\begin{equation*}
\Ind_{S^*\Y}(A_{\wedge,\min}) = \Ind_{S^*\Y}(A_{\wedge,\max}) + \Ind_{S^*\Y}(\iota).
\end{equation*}
Because $\Ind_{S^*\Y}(\iota) = -[\pi_{\Y}^*\trb]$ we obtain
\begin{align*}
[\kerb] + [\kerb^{\star}] &= \Ind_{S^*\Y}(A_{\wedge,\max}) + \Ind_{S^*\Y}(A^{\star}_{\wedge,\max}) \\
&= \Ind_{S^*\Y}(A_{\wedge,\max}) - \Ind_{S^*\Y}(A_{\wedge,\min}) \\
&= \Ind_{S^*\Y}(A_{\wedge,\max}) - \big(\Ind_{S^*\Y}(A_{\wedge,\max}) + \Ind_{S^*\Y}(\iota)\big) \\
&= -\Ind_{S^*\Y}(\iota) = [\pi_{\Y}^*\trb]
\end{align*}
as claimed. See \cite{Waterstraat} for background on the index in $K$-theory for Fredholm morphisms between general Banach bundles.
\end{proof}

We continue the proof of Theorem~\ref{FredholmTheorem} by establishing reductions \eqref{Reduct_2} and \eqref{Reduct_3}. Consider the operator
\begin{equation*}
{\mathcal A} = \begin{bmatrix} 0 & A^{\star} \\ A & 0 \end{bmatrix} \in x^{-1}\Diff^1_e\biggl(\M;\begin{array}{c} E \\ \oplus \\ F \end{array},\begin{array}{c} E \\ \oplus \\ F \end{array}\biggr),
\end{equation*}
acting as an unbounded operator in $x^{-1/2}L^2_b\biggl(\M;\begin{array}{c} E \\ \oplus \\ F \end{array}\biggr)$. This operator is symmetric, and satisfies all assumptions \eqref{Awelliptic}, \eqref{speceAssumption}, and \eqref{normalAssumption}. The trace bundle $\trb_{{\mathcal A}}$ of ${\mathcal A}$ is the direct sum $\trb \oplus \trb^{\star}$ of the trace bundles of $A$ and $A^{\star}$. Moreover, the bundle $\kerb_{{\mathcal A}} \subset \pi_{\Y}^*\trb_{{\mathcal A}}$ built from the kernel bundle $\KK_{{\mathcal A}}$ of the family
\begin{equation*}
{\mathcal A}_{\wedge}(\pmb\eta) = \begin{bmatrix} 0 & A_{\wedge}^{\star}(\pmb\eta) \\ A_{\wedge}(\pmb\eta) & 0 \end{bmatrix} : \begin{array}{c} \Dom_{\max}(A_{\wedge}(\pmb\eta)) \\ \oplus \\ \Dom_{\max}(A_{\wedge}^{\star}(\pmb\eta)) \end{array} \to \begin{array}{c} x^{-1/2}L^2_b(\Z_y^{\wedge},E_{\Z_y^{\wedge}}) \\ \oplus \\ x^{-1/2}L^2_b(\Z_y^{\wedge},F_{\Z_y^{\wedge}}) \end{array}
\end{equation*}
by applying the fiberwise trace map $\gamma_{{\mathcal A}_{\wedge}(\pmb\eta)} = \gamma_{A_{\wedge}(\pmb\eta)} \oplus \gamma_{A_{\wedge}^{\star}(\pmb\eta)}$
is the direct sum $\kerb \oplus \kerb^{\star}$ of the corresponding bundles $\kerb \subset \pi_{\Y}^*\trb$ and $\kerb^{\star} \subset \pi_{\Y}^*\trb^{\star}$ associated with $A$ and $A^{\star}$, respectively. In particular $[\kerb_{{\mathcal A}}] = [\kerb] + [\kerb^{\star}] = [\pi_{\Y}^*\trb]$ as elements in $K(S^*\Y)$ by Lemma~\ref{KplusKstar}. Consequently, the operator ${\mathcal A}$ has vanishing Atiyah-Bott obstruction and satisfies \eqref{Reduct_3}.

Let $\Pi_{A^{\star}} = \Pi_{A^{\star}}^2 \in \Psi^0_{1,\delta}(\Y;(\trb^{\star},-\gen),(\trb^{\star},-\gen))$ have twisted homogeneous principal symbol $\sym(\Pi_{A^{\star}})$ that is a projection in $\pi_{\Y}^*\trb^{\star}$ onto the subbundle $\kerb^{\star}$ on $T^*\Y\minus 0$ (see Example~\ref{APSBoundaryValueProblem}). Define
\begin{align*}
{\mathcal B} = \begin{bmatrix} B & 0 \\ 0 & \Id \end{bmatrix} &\in
\Psi^0_{1,\delta}\biggl(\Y;\biggl(\begin{array}{c} \trb \\ \oplus \\ \trb^{\star} \end{array},\begin{bmatrix} -\gen & 0 \\ 0 & -\gen \end{bmatrix}\biggr),\biggl(\begin{array}{c} \target \\ \oplus \\ \trb^{\star} \end{array},\begin{bmatrix} {\agen} & 0 \\ 0 & -\gen \end{bmatrix}\biggr)\biggr), \\
\tilde{\Pi} = \begin{bmatrix} \Pi & 0 \\ 0 & \Pi_{A^{\star}} \end{bmatrix} &\in
\Psi^0_{1,\delta}\biggl(\Y;\biggl(\begin{array}{c} \target \\ \oplus \\ \trb^{\star} \end{array},\begin{bmatrix} {\agen} & 0 \\ 0 & -\gen \end{bmatrix}\biggr),\biggl(\begin{array}{c} \target \\ \oplus \\ \trb^{\star} \end{array},\begin{bmatrix} {\agen} & 0 \\ 0 & -\gen \end{bmatrix}\biggr)\biggr).
\end{align*}
The boundary value problem represented by the operator
\begin{equation}\label{ReducedBVP}
\begin{bmatrix} {\mathcal A} \\ \tilde{\Pi}{\mathcal B}\gamma_{{\mathcal A}} \end{bmatrix} :
H^1_{\trb_{{\mathcal A}}}\biggl(\M;\begin{array}{c} E \\ \oplus \\ F \end{array}\biggr)
\to
\begin{array}{c}
x^{-1/2}L^2_b\biggl(\M;\begin{array}{c} E \\ \oplus \\ F \end{array}\biggr)
\\ \oplus \\
\tilde{\Pi}H^{1+\tilde{{\agen}}}\biggl(\Y;\begin{array}{c} \target \\ \oplus \\ \trb^{\star} \end{array}\biggr)
\end{array},
\end{equation}
where $\tilde{{\agen}} = \begin{bmatrix} {\agen} & 0 \\ 0 & -\gen \end{bmatrix} \in \End\biggl(\begin{array}{c} \target \\ \oplus \\ \trb^{\star} \end{array}\biggr)$, satisfies all assumptions of Theorem~\ref{FredholmTheorem}, and in addition it satisfies the extra assumptions \eqref{Reduct_1}, \eqref{Reduct_2}, and \eqref{Reduct_3} stated at the beginning of the proof. Consequently, if Theorem~\ref{FredholmTheorem} is proved for boundary value problems satisfying these extra assumptions, we can apply the conclusion of the theorem and obtain that \eqref{ReducedBVP} is a Fredholm operator. By construction of the operator \eqref{ReducedBVP} this then implies that the original operator \eqref{ABVP} is Fredholm, thus proving the theorem in full generality.

\medskip

We now proceed with proving Theorem~\ref{FredholmTheorem} for the operator \eqref{ABVP} under the additional assumptions \eqref{Reduct_1}, \eqref{Reduct_2}, and \eqref{Reduct_3}.

To begin we make use of the isomorphism
\begin{equation*}
\begin{bmatrix} \iota & \Ext \end{bmatrix} :
\begin{array}{c} x^{1/2}H^1_e(\M;E) \\ \oplus \\ H^{1-\gen}(\Y;\trb) \end{array} \to
H^1_{\trb}(\M;E)
\end{equation*}
in Proposition~\ref{H1ARepresentation}. Recall that $\Dom_{\min}(A) = x^{1/2}H^1_e(\M;E)$ by Theorem~\ref{DminTheorem}.

Clearly, then, the operator \eqref{ABVP} is Fredholm if and only if the operator matrix
\begin{equation}\label{ABVPBlock}
\begin{bmatrix} A & A\Ext \\ 0 & \Pi B \end{bmatrix} =
\begin{bmatrix} A \\ \Pi B \gamma_A \end{bmatrix} \circ \begin{bmatrix} \iota & \Ext \end{bmatrix} :
\begin{array}{c} x^{1/2}H^1_e(\M;E) \\ \oplus \\ H^{1-\gen}(\Y;\trb) \end{array} \to
\begin{array}{c} x^{-1/2}L^2_b(\M;E) \\ \oplus \\ \Pi H^{1+\agen}(\Y;\target) \end{array}
\end{equation}
is Fredholm. Recall that $\gamma_A \equiv 0$ on $\Dom_{\min}(A)=x^{1/2}H^1_e(\M;E)$, and $\gamma_A\circ\Ext = \Id$ on $H^{1-\gen}(\Y;\trb)$, see Proposition~\ref{PoissonOperatorProperties}.

By assumption \eqref{BoundaryConditionKernelIso} and \eqref{Reduct_3} we have $[\target_{\Pi}] = [\kerb] = [\pi_{\Y}^*\target_0] \in K(S^*\Y)$ for some $\target_0 \in \Vect(\Y)$. In other words, there is a vector bundle isomorphism
\begin{equation}\label{csymbol}
c : \begin{array}{c} \target_{\Pi} \\ \oplus \\ \C^k \end{array} \overset{\cong}{\longrightarrow} \pi_{\Y}^*\biggl(\begin{array}{c} \target_0 \\ \oplus \\ \C^k \end{array}\biggr)
\end{equation}
on $S^*\Y$ for some $k \in \NN_0$. By replacing $c$ by $c\begin{bmatrix} \sym(\Pi) & 0 \\ 0 & \Id \end{bmatrix}$ (but keeping the notation $c$) we obtain
\begin{equation*}
c \in \Hom\biggl(\pi_{\Y}^*\biggl(\begin{array}{c} \target \\ \oplus \\ \C^k \end{array}\biggr), \pi_{\Y}^*\biggl(\begin{array}{c} \target_0 \\ \oplus \\ \C^k \end{array} \biggr)\biggr)
\end{equation*}
on $S^*\Y$, which we then extend by twisted homogeneity of degree one to $T^*\Y\minus 0$ with respect to the actions generated by the endomorphisms
\begin{equation*}
\begin{bmatrix} {\agen} & 0 \\ 0 & 0 \end{bmatrix} \in \End\biggl(\begin{array}{c} \target \\ \oplus \\ \C^k \end{array}\biggr) \text{ and }
\begin{bmatrix} 0 & 0 \\ 0 & 0 \end{bmatrix} \in \End\biggl(\begin{array}{c} \target_0 \\ \oplus \\ \C^k \end{array}\biggr).
\end{equation*}
Let
\begin{equation*}
C = \begin{bmatrix} C_{11} & C_{12} \\ C_{21} & C_{22} \end{bmatrix} \in \Psi_{1,\delta}^1\biggl(\Y;\biggl(\begin{array}{c} \target \\ \oplus \\ \C^k \end{array},\begin{bmatrix} {\agen} & 0 \\ 0 & 0 \end{bmatrix}\biggr),\biggl(\begin{array}{c} \target_0 \\ \oplus \\ \C^k \end{array},\begin{bmatrix} 0 & 0 \\ 0 & 0 \end{bmatrix}\biggr)\biggr)
\end{equation*}
with $\sym(C) = c$. Because \eqref{csymbol} is an isomorphism on $S^*\Y$ we obtain from \cite[Theorem~8.2]{KrMe12b} that the operator
\begin{equation*}
\begin{bmatrix} C_{11} & C_{12} \\ C_{21} & C_{22} \end{bmatrix}\begin{bmatrix} \Pi & 0 \\ 0 & \Id \end{bmatrix} : \begin{array}{c} \Pi H^{1+{\agen}}(\Y;\target) \\ \oplus \\ H^1(\Y;\C^k) \end{array} \to \begin{array}{c} L^2(\Y;\target_0) \\ \oplus \\ L^2(\Y;\C^k) \end{array}
\end{equation*}
is Fredholm. Consequently, the operator matrix \eqref{ABVPBlock} (and therefore \eqref{ABVP}) is Fredholm if and only if
\begin{equation}\label{ABVPBlock0}
\underbrace{\begin{bmatrix}
1 & 0 & 0 \\ 0 & C_{11}\Pi & C_{12} \\ 0 & C_{21}\Pi & C_{22}
\end{bmatrix}
\begin{bmatrix}
A & A \Ext & 0 \\ 0 & \Pi B & 0 \\ 0 & 0 & \Id
\end{bmatrix}}_{= \begin{bmatrix}
A & A \Ext & 0 \\ 0 & C_{11}\Pi B & C_{12} \\ 0 & C_{21}\Pi B  & C_{22}
\end{bmatrix}}
: \begin{array}{c} x^{1/2}H^1_e(\M;E) \\ \oplus \\ H^{1-\gen}(\Y;\trb) \\ \oplus \\ H^1(\Y;\C^k) \end{array} \to
\begin{array}{c} x^{-1/2}L^2_b(\M;E) \\ \oplus \\ L^2(\Y;\target_0) \\ \oplus \\ L^2(\Y;\C^k) \end{array}
\end{equation}
is Fredholm.

Choose an elliptic operator $K \in \Psi^{0}_{1,\delta}(\Y;(\trb,0),(\trb,-\gen))$ that has twisted homogeneous principal symbol and is invertible with $K^{-1} \in \Psi^{0}_{1,\delta}(\Y;(\trb,-\gen),(\trb,0))$. Such an operator exists by \cite[Theorem~6.9]{KrMe12b}. Then $K : H^1(\Y;\trb) \to H^{1-\gen}(\Y;\trb)$ is an isomorphism, and therefore \eqref{ABVPBlock0} is Fredholm if and only if
\begin{equation}\label{ABVPBlock1}
\begin{bmatrix}
A & A \Ext K & 0 \\ 0 & C_{11}\Pi B K & C_{12} \\ 0 & C_{21}\Pi B K & C_{22}
\end{bmatrix} :
\begin{array}{c} x^{1/2}H^1_e(\M;E) \\ \oplus \\ H^{1}(\Y;\trb) \\ \oplus \\ H^1(\Y;\C^k) \end{array} \to
\begin{array}{c} x^{-1/2}L^2_b(\M;E) \\ \oplus \\ L^2(\Y;\target_0) \\ \oplus \\ L^2(\Y;\C^k) \end{array}
\end{equation}
is Fredholm.

From the composition theorem \cite[Theorem~6.5]{KrMe12b} we get that
\begin{equation*}
\begin{bmatrix} C_{11}\Pi B K & C_{12} \\ C_{21}\Pi B K & C_{22}\end{bmatrix} =
\begin{bmatrix} C_{11} & C_{12} \\ C_{21} & C_{22} \end{bmatrix}\begin{bmatrix} \Pi B & 0 \\ 0 & \Id \end{bmatrix}\begin{bmatrix} K & 0 \\ 0 & \Id \end{bmatrix}
\end{equation*}
belongs to
\begin{equation*}
\Psi_{1,\delta}^1\biggl(\Y;\begin{array}{c} \trb \\ \oplus \\ \C^k \end{array},\begin{array}{c} \target_0 \\ \oplus \\ \C^k \end{array}\biggr),
\end{equation*}
the standard H{\"o}rmander class of pseudodifferential operators of type $(1,\delta)$, and this operator has principal symbol given by
\begin{equation}\label{prsymlbk}
c(\pmb\eta) \begin{bmatrix} \sym(\Pi B)(\pmb\eta) & 0 \\ 0 & 1 \end{bmatrix}\begin{bmatrix} \sym(K)(\pmb\eta) & 0 \\ 0 & \Id \end{bmatrix},
\end{equation}
a homogeneous section of degree one of the homomorphism bundle on $T^*\Y\minus 0$ in the ordinary (untwisted) sense. Now choose an operator
\begin{equation*}
\tilde{B} = \begin{bmatrix} \tilde{B}_{11} & \tilde{B}_{12} \\ \tilde{B}_{21} & \tilde{B}_{22} \end{bmatrix} \in \Psi_{\cl}^1\biggl(\Y;\begin{array}{c} \trb \\ \oplus \\ \C^k \end{array},\begin{array}{c} \target_0 \\ \oplus \\ \C^k \end{array}\biggr)
\end{equation*}
in the standard $(1,0)$-calculus of classical pseudodifferential operators with $\sym(\tilde{B})(\pmb\eta)$ given by \eqref{prsymlbk}. Then
\begin{equation*}
\begin{bmatrix} C_{11}\Pi B K & C_{12} \\ C_{21}\Pi B K & C_{22}\end{bmatrix} - \tilde{B} :
\begin{array}{c} H^{1}(\Y;\trb) \\ \oplus \\ H^1(\Y;\C^k) \end{array} \to
\begin{array}{c} L^2(\Y;\target_0) \\ \oplus \\ L^2(\Y;\C^k) \end{array}
\end{equation*}
is a compact operator.

Next choose a Green operator
\begin{equation*}
G_0 : C^{\infty}(\Y;\trb) \to C^{\infty}(\open\M;E)
\end{equation*}
of order one with respect to the weights $(\frac{1}{2},-\frac{1}{2})$ whose normal family is given by
\begin{equation}\label{Gwedgeprsym}
G_{0,\wedge}(\pmb\eta) = A_{\wedge}(\pmb\eta)\circ\Ext_{\wedge}(\pmb\eta)\circ\sym(K)(\pmb\eta)
\end{equation}
on $T^*\Y\minus 0$, see Remark~\ref{PerturbationOfGreen}. Then
\begin{equation*}
A \Ext K - G_0 : H^1(\Y;\trb) \to x^{-1/2}L^2_b(\M;E)
\end{equation*}
is compact by that same remark because it has the mapping property \eqref{GreenPerturb}. Consequently, the operator \eqref{ABVPBlock1} (and therefore \eqref{ABVP}) is Fredholm if and only if
\begin{equation}\label{ABVPBlock2}
\begin{bmatrix} A & G \\ 0 & \tilde{B} \end{bmatrix} :
\begin{array}{c} x^{1/2}H^1_e(\M;E) \\ \oplus \\ H^{1}\biggl(\Y;\begin{array}{c} \trb \\ \oplus \\ \C^k\end{array}\biggr) \end{array} \to
\begin{array}{c} x^{-1/2}L^2_b(\M;E) \\ \oplus \\ L^2\biggl(\Y;\begin{array}{c} \target_0 \\ \oplus \\ \C^k\end{array}\biggr)
\end{array}
\end{equation}
is Fredholm, where $G = \begin{bmatrix} G_0 & 0 \end{bmatrix}$. The operator matrix
\begin{equation*}
\begin{bmatrix} 0 & G \\ 0 & \tilde{B} \end{bmatrix} :
\begin{array}{c} x^{1/2}H^1_e(\M;E) \\ \oplus \\ H^{1}\biggl(\Y;\begin{array}{c} \trb \\ \oplus \\ \C^k\end{array}\biggr) \end{array} \to
\begin{array}{c} x^{-1/2}L^2_b(\M;E) \\ \oplus \\ L^2\biggl(\Y;\begin{array}{c} \target_0 \\ \oplus \\ \C^k\end{array}\biggr)
\end{array}
\end{equation*}
is a Green operator of order one with respect to the weights $(\frac{1}{2},-\frac{1}{2})$ as discussed in Appendix~\ref{GreenOperators}, and consequently the operator \eqref{ABVPBlock2} can be analyzed with the methods of Schulze's edge calculus \cite{SchuNH}. In view of our standing assumptions \eqref{Awelliptic} and \eqref{speceAssumption}, Schulze's theory implies that the operator \eqref{ABVPBlock2} is Fredholm provided that
\begin{equation}\label{SchulzeWedgeInv}
\begin{bmatrix} A_{\wedge}(\pmb\eta) & G_{\wedge}(\pmb\eta) \\ 0 & \sym(\tilde{B})(\pmb\eta) \end{bmatrix} :
\begin{array}{c}
\K^{1,\frac{1}{2}}_{\frac{1}{2}-\frac{1}{2}\dim\Z^{\wedge}_y}(\Z^{\wedge}_y;E_{\Z^{\wedge}_y}) \\ \oplus \\
\left[\begin{array}{c} \pi_{\Y}^*\trb_y \\ \oplus \\ \C^k \end{array}\right]
\end{array}
\to
\begin{array}{c}
x^{-1/2}L^2_b(\Z^{\wedge}_y;E_{\Z^{\wedge}_y}) \\ \oplus \\
\left[\begin{array}{c} \pi_{\Y}^*\target_{0,y}  \\ \oplus \\ \C^k \end{array}\right]
\end{array},
\end{equation}
is invertible for all $\pmb\eta \in T^*\Y\minus 0$, where $y = \pi_{\Y}\pmb\eta$ (see also Appendix~\ref{GreenOperators} for more details on ellipticity and the Fredholm property in Schulze's edge calculus).

Because $\sym(\tilde{B})(\pmb\eta)$ is given by \eqref{prsymlbk} and $G_{\wedge}(\pmb\eta) = \begin{bmatrix} G_{0,\wedge}(\pmb\eta) & 0 \end{bmatrix}$ with $G_{0,\wedge}(\pmb\eta)$ as in \eqref{Gwedgeprsym}, and both $c(\pmb\eta)$ in \eqref{csymbol} and $\sym(K)(\pmb\eta)$ are invertible, we readily obtain that \eqref{SchulzeWedgeInv} is invertible if and only if
\begin{equation}\label{APrWSym}
\begin{bmatrix} A_{\wedge}(\pmb\eta) \\ \sym(\Pi B)(\pmb\eta)\circ\gamma_{A_\wedge(\pmb\eta)} \end{bmatrix} : \Dom_{\max}(A_{\wedge}(\pmb\eta)) \to \begin{array}{c} x^{-1/2}L^2_b(\Z_y^{\wedge};E_{\Z_y^{\wedge}}) \\ \oplus \\ \target_{\Pi,\pmb\eta} \end{array}
\end{equation}
is invertible. This can be seen by applying the exact same reasoning that we utilized to convert problem \eqref{ABVP} to problem \eqref{ABVPBlock1} via \eqref{ABVPBlock} and \eqref{ABVPBlock0} on the level of the operators to the normal families. Recall that $\gamma_{A_{\wedge}(\pmb\eta)} \equiv 0$ on $\Dom_{\min}(A_{\wedge}(\pmb\eta)) = \K^{1,\frac{1}{2}}_{\frac{1}{2}-\frac{1}{2}\dim\Z^{\wedge}_y}(\Z^{\wedge}_y;E_{\Z^{\wedge}_y})$ (see \eqref{DomMinAWedge}), and $\gamma_{A_{\wedge}(\pmb\eta)}\circ\Ext_{\wedge}(\pmb\eta) = \Id$.

By \eqref{normalAssumption}, the assumed ellipticity condition \eqref{BoundaryConditionKernelIso} is equivalent to the invertibility of \eqref{APrWSym} on $T^*\Y\minus 0$. Consequently, the operator \eqref{ABVPBlock2} is Fredholm, and therefore \eqref{ABVP} is Fredholm. This finishes the proof of the theorem.

\begin{remark}
The idea to analyze a boundary value problem with nonvanishing Atiyah-Bott obstruction by adding on a complementary boundary value problem to obtain an enlarged system with vanishing Atiyah-Bott obstruction, and subsequently to take advantage of a suitable pseudodifferential calculus of operator matrices capable of analyzing the enlarged problem, was applied by Grubb and Seeley in \cite{GrubbSeeley} to analyze the Atiyah-Patodi-Singer boundary value problem. This was taken up and further developed in \cite{Schu2001} (for classical pseudodifferential boundary value problems based on Boutet de Monvel's calculus) and \cite{SchuSei2006} (for pseudodifferential edge operator matrices in Schulze's edge calculus).
\end{remark}

\appendix

\section{Pseudodifferential operators with operator-valued symbols that are twisted by group actions}\label{app-PseudosOpValued}

This appendix provides a summary of definitions and basic results on pseudodifferential operators with operator-valued symbols that are twisted by strongly continuous group actions. This calculus was introduced by Schulze to facilitate the description of the structure near the boundary of operators in his edge calculus and in Boutet de Monvel's calculus that are smoothing in the interior but not necessarily compact. The ideals of smoothing operators of such kind are generally referred to as Green operators, singular Green operators, or generalized singular Green operators in the literature. We will review the definition and some of the properties of Green operators in the edge calculus in Appendix~\ref{GreenOperators}.  A standard reference in this context is the monograph \cite{SchuNH}, see also \cite{EgorovSchulze,SchuWiley}. The results on boundedness and compactness of pseudodifferential operators with operator-valued symbols in abstract edge Sobolev spaces as presented here in Theorem~\ref{WSpaceContinuity} were obtained in \cite{SeilerCont}.

\begin{definition}
Let $H$ and $\tilde{H}$ be Hilbert spaces equipped with strongly continuous group actions $\kappa_{\varrho}$ on $H$ and $\tilde{\kappa}_{\varrho}$ on $\tilde{H}$ for $\varrho > 0$.
\begin{enumerate}[(a)]
\item For $\mu \in \R$ let $S^{\mu}(\R^q\times\R^q;H,\tilde{H})$ be the space of all $a(y,\eta) \in C^{\infty}(\R^q\times\R^q,\L(H,\tilde{H}))$ such that for all $\alpha,\beta \in \NN_0^q$ there exists a constant $C_{\alpha,\beta} > 0$ such that
\begin{equation}\label{symbolestimate}
\|\tilde{\kappa}_{\langle \eta \rangle}^{-1}\bigl(D_y^{\alpha}\partial_{\eta}^{\beta}a(y,\eta)\bigr)\kappa_{\langle \eta \rangle}\|_{\L(H,\tilde{H})} \leq C_{\alpha,\beta}\langle \eta \rangle^{\mu-|\beta|}
\end{equation}
for all $(y,\eta) \in \R^q\times\R^q$, where as usual $\langle \eta \rangle = (1 + |\eta|^2)^{1/2}$.
\item With every $a(y,\eta) \in S^{\mu}(\R^q\times\R^q;H,\tilde{H})$ we associate a pseudodifferential operator
\begin{equation*}
\op(a) : \Sch(\R^q,H) \to \Sch(\R^q,\tilde{H})
\end{equation*}
via
\begin{equation*}
\bigl[\op(a)u\bigr](y) = \frac{1}{(2\pi)^q}\iint e^{i(y-y')\eta}a(y,\eta)u(y')\,dy'\,d\eta.
\end{equation*}
\end{enumerate}
\end{definition}

The usual classes of operator-valued symbols are recovered in this definition when choosing the trivial group actions $\kappa_{\varrho} \equiv \Id_H$ and $\tilde{\kappa}_{\varrho} \equiv \Id_{\tilde{H}}$ on the Hilbert spaces.

In general, the strong continuity of the group action $\kappa_{\varrho}$ implies that there exist $c,M \geq 0$ such that
\begin{equation}\label{estimategroupaction}
\|\kappa_{\varrho}\|_{\L(H)} \leq c \max\{\varrho,\varrho^{-1}\}^M
\end{equation}
for all $\varrho > 0$. Consequently, a symbol $a(y,\eta)$ that satisfies the symbol estimates \eqref{symbolestimate} ``twisted'' by general group actions $\kappa_{\varrho}$ and $\tilde{\kappa}_{\varrho}$ of order $\mu$ also satisfies the standard ``untwisted'' symbol estimates of order $\mu + M + \tilde{M}$, where $M$ and $\tilde{M}$ are the exponents in the estimate \eqref{estimategroupaction} for the group actions $\kappa_{\varrho}$ and $\tilde{\kappa}_{\varrho}$, respectively, and vice versa. Thus both
\begin{align*}
S^{-\infty}(\R^q\times\R^q;H,\tilde{H}) &= \bigcap_{\mu \in \R}S^{\mu}(\R^q\times\R^q;H,\tilde{H}) \\
\intertext{and}
S^{\infty}(\R^q\times\R^q;H,\tilde{H}) &= \bigcup_{\mu \in \R}S^{\mu}(\R^q\times\R^q;H,\tilde{H})
\end{align*}
are independent of the group actions. In particular, 
\begin{equation*}
\display{320pt}{the standard elements of pseudodifferential calculus (asymptotic expansions, compositions, formal adjoints, etc.) are applicable to operators with symbols that satisfy the twisted estimates \eqref{symbolestimate},}
\end{equation*}
see below for the pertinent statements. In essence, twisting by nontrivial group actions just induces a filtration by order that is different from the standard filtration.

\begin{lemma}
Let $H$, $\tilde{H}$, and $H_j$, $j=1,2,3$, be equipped with strongly continuous group actions.
\begin{enumerate}[(a)]
\item $S^{\mu}(\R^q\times\R^q;H,\tilde{H})$ is a Fr{\'e}chet space with the topology induced by the seminorms given by the best constants in the estimates \eqref{symbolestimate} for all $\alpha,\beta \in \NN_0^q$.
\item Multiplication of operator functions induces a bilinear map
\begin{equation*}
S^{\mu_1}(\R^q\times\R^q;H_2,H_3) \times S^{\mu_2}(\R^q\times\R^q;H_1,H_2) \to
S^{\mu_1+\mu_2}(\R^q\times\R^q;H_1,H_3).
\end{equation*}
\item Differentiation induces a map
\begin{equation*}
D_y^{\alpha}\partial_{\eta}^{\beta} : S^{\mu}(\R^q\times\R^q;H,\tilde{H}) \to S^{\mu-|\beta|}(\R^q\times\R^q;H,\tilde{H})
\end{equation*}
for all $\alpha,\beta \in \NN_0^q$.
\item Let $a_j \in S^{\mu_j}(\R^q\times\R^q;H,\tilde{H})$ with $\mu_j \to -\infty$ as $j \to \infty$, and let $\overline{\mu} = \max\limits_{j}\mu_j$. Then there exists $a \in S^{\overline{\mu}}(\R^q\times\R^q;H,\tilde{H})$ such that
$a \sim \sum\limits_{j=0}^{\infty}a_j$.
\end{enumerate}
\end{lemma}

\begin{proposition}\label{composition}
Let $a_1 \in S^{\mu_1}(\R^q\times\R^q;H_2,H_3)$ and $a_2 \in S^{\mu_2}(\R^q\times\R^q;H_1,H_2)$. Then
\begin{equation*}
\op(a_1)\circ\op(a_2) = \op(a_1 \leibniz a_2) : \Sch(\R^q,H_1) \to \Sch(\R^q,H_3)
\end{equation*}
with $a_1 \leibniz a_2 \in S^{\mu_1+\mu_2}(\R^q\times\R^q;H_1,H_3)$. We have
\begin{equation*}
\bigl(a_1 \leibniz a_2)(y,\eta) = \frac{1}{(2\pi)^q}\iint e^{-iy'\eta'}a_1(y,\eta+\eta')a_2(y+y',\eta)\,dy'\,d\eta'.
\end{equation*}
Moreover,
\begin{equation*}
a_1 \leibniz a_2 \sim \sum\limits_{\alpha \in \NN_0^q} \frac{1}{\alpha !} \bigl(\partial_{\eta}^{\alpha}a_1\bigr)\bigl(D_y^{\alpha}a_2\bigr),
\end{equation*}
where more precisely, for every $N \in \NN$,
\begin{equation*}
r_N = a_1 \leibniz a_2  - \sum\limits_{|\alpha| \leq N-1} \frac{1}{\alpha !} \bigl(\partial_{\eta}^{\alpha}a_1\bigr)\bigl(D_y^{\alpha}a_2\bigr)
\end{equation*}
is given by the oscillatory integral
{\small
\begin{equation*}
N \sum\limits_{|\alpha|=N}\int_0^1 \frac{(1-\theta)^{N-1}}{\alpha!}
\frac{1}{(2\pi)^q}\iint e^{-iy'\eta'}\bigl(\partial_\eta^{\alpha}a_1\bigr)(y,\eta+\theta\eta')\bigl(D_y^{\alpha}a_2\bigr)(y+y',\eta)\,dy'\,d\eta'\,d\theta.
\end{equation*}
}
\end{proposition}

The following corollary is particularly useful in the case that the embedding $H_4 \hookrightarrow H_3$ is compact in conjunction with the boundedness and compactness Theorem~\ref{WSpaceContinuity}.

\begin{corollary}\label{CorSmallerSpace}
Let $a_1 \in S^{\mu_1}(\R^q\times\R^q;H_2,H_3)$ and $a_2 \in S^{\mu_2}(\R^q\times\R^q;H_1,H_2)$. Let $H_4$ be a Hilbert space with $H_4 \hookrightarrow H_3$, and assume that the group action on $H_3$ restricts to a strongly continuous group action on $H_4$.

Suppose there exists an $N \in \NN$ such that
\begin{equation*}
D_y^{\alpha}\partial_{\eta}^{\beta}a_1 \in S^{\mu_1-|\beta|}(\R^q\times\R^q;H_2,H_4)
\end{equation*}
for all $\alpha \in \NN_0^q$ and all $|\beta| \geq N$.

Then
\begin{equation*}
a_1 \leibniz a_2  - \sum\limits_{|\alpha| \leq N-1} \frac{1}{\alpha !} \bigl(\partial_{\eta}^{\alpha}a_1\bigr)\bigl(D_y^{\alpha}a_2\bigr) \in S^{\mu_1+\mu_2-N}(\R^q\times\R^q;H_1,H_4).
\end{equation*}
\end{corollary}

\begin{definition}[Abstract edge Sobolev space]\label{WSpace}
For $s \in \R$ let ${\mathcal W}^s(\R^q,H)$ be the completion of $\Sch(\R^q,H)$ with respect to the norm
\begin{equation*}
\|u\|_{{\mathcal W}^s}^2 = \int_{\R^q}\langle \eta \rangle^{2s}\|\kappa^{-1}_{\langle \eta \rangle}\bigl({\mathscr F}u\bigr)(\eta)\|_H^2\,d\eta.
\end{equation*}
\end{definition}

\begin{theorem}\label{WSpaceContinuity}
Let $a \in S^{\mu}(\R^q\times\R^q;H,\tilde{H})$.
\begin{enumerate}[(a)]
\item The operator
\begin{equation*}
\op(a) : {\mathcal W}^s(\R^q,H) \to {\mathcal W}^{s-\mu}(\R^q,\tilde{H})
\end{equation*}
is continuous for every $s \in \R$.
\item If $a(y,\eta) : H \to \tilde{H}$ is compact for every $(y,\eta) \in \R^q$ and $a(y,\eta) \equiv 0$ for $|y| \geq R$ large enough, then
\begin{equation*}
\op(a) : {\mathcal W}^s(\R^q,H) \to {\mathcal W}^{s-\mu_1}(\R^q,\tilde{H})
\end{equation*}
is compact for all $s \in \R$ and $\mu_1 > \mu$.
\end{enumerate}
\end{theorem}

Let $\iota : H \hookrightarrow \tilde{H}$ be compact, and assume that the group action on $H$ is induced by restricting the group action from $\tilde{H}$ to $H$. Let $\varphi \in C_c^{\infty}(\R^q)$, and consider $a(y,\eta) = \varphi(y)\iota : H \to \tilde{H}$ in Theorem~\ref{WSpaceContinuity}(b). The theorem then implies that the multiplication operator $\op(a) : u \mapsto \varphi u$ is compact in ${\mathcal W}^s(\R^q,H) \to {\mathcal W}^t(\R^q,\tilde{H})$ for $t < s$.

\bigskip

The calculus of pseudodifferential operators with operator-valued symbols that are twisted by group actions has a subclass of operators with classical symbols. The notion of classical symbol is built on a suitable notion of twisted homogeneity, which is the following.

\begin{definition}\label{twistedhomogeneous}
A function $a_{(\mu)}(y,\eta) \in C^{\infty}(\R^q\times\bigl(\R^q\minus\{0\}\bigr),\L(H,\tilde{H}))$ is twisted homogeneous (in the variable $\eta \in \R^q\minus\{0\}$) of degree $\mu \in \R$ if
\begin{equation*}
a_{(\mu)}(y,\varrho\eta) = \varrho^{\mu}\tilde{\kappa}_{\varrho}a_{(\mu)}(y,\eta)\kappa_{\varrho}^{-1}
\end{equation*}
holds for all $\varrho > 0$ and $(y,\eta) \in \R^q\times\bigl(\R^q\minus\{0\}\bigr)$.
\end{definition}

Let $\chi \in C^{\infty}(\R^q)$ be a function such that $\chi \equiv 0$ near zero and $\chi(\eta) \equiv 1$ for $|\eta|$ large enough. If $a_{(\mu)}(y,\eta)$ is as in Definition~\ref{twistedhomogeneous} and all derivatives $D_y^{\alpha}\partial_{\eta}^{\beta}a_{(\mu)}(y,\eta)$ are bounded for $y \in \R^q$ and $|\eta| = 1$, then clearly
\begin{equation*}
\chi(\eta)a_{(\mu)}(y,\eta) \in S^{\mu}(\R^q\times\R^q;H,\tilde{H}).
\end{equation*}
A classical symbol of order $\mu$ is a symbol $a(y,\eta) \in S^{\mu}(\R^q\times\R^q;H,\tilde{H})$ that has an asymptotic expansion
\begin{equation}\label{HomCompExp}
a(y,\eta) \sim \sum\limits_{j=0}^{\infty}\chi(\eta)a_{(\mu-j)}(y,\eta),
\end{equation}
where $a_{(\mu-j)}(y,\eta)$ is twisted homogeneous in $\eta \in \R^q\minus \{0\}$ of degree $\mu - j$, $j \in \NN_0$. Let $S^{\mu}_{\cl}(\R^q\times\R^q;H,\tilde{H})$ denote the space of classical symbols of order $\mu \in \R$.

Standard arguments show that the class of operators with classical symbols is closed with respect to the usual operations in pseudodifferential calculus.

\subsection*{Example: Boundary symbols of edge differential operators}

Let
\begin{equation*}
P = \sum\limits_{j+|\alpha| \leq m} \varphi_{j,\alpha}(y,x)(xD_y)^{\alpha}(xD_x)^j : C_c^{\infty}(\R^q\times\open\Z^{\wedge};E) \to C^{\infty}(\R^q\times\open\Z^{\wedge};F)
\end{equation*}
with $\varphi_{j,\alpha} \in C_c^{\infty}(\R^q\times\overline{\R}_+,\Diff^{m-j-|\alpha|}(\Z;E,F))$ be the representation of an edge differential operator on $\M$ in adapted coordinates and trivializations near the boundary. Here $E$ and $F$ are vector bundles on $\Z^{\wedge}$ that are pull-backs of corresponding bundles over $\Z$ with the induced Hermitian metrics. Formally,
\begin{equation*}
P = \op(p) : C_c^{\infty}(\R^q,C_c^{\infty}(\open\Z^{\wedge};E)) \to C^{\infty}(\R^q,C^{\infty}(\open\Z^{\wedge};F))
\end{equation*}
with an operator-valued symbol
\begin{equation*}
p(y,\eta) = \sum\limits_{j+|\alpha| \leq m}\varphi_{j,\alpha}(y,x)(x\eta)^{\alpha}(xD_x)^{j} : C_c^{\infty}(\open\Z^{\wedge};E) \to C^{\infty}(\open\Z^{\wedge};F).
\end{equation*}
The symbol $p(y,\eta)$ is referred to as the complete boundary or edge symbol of $P$ in the chosen coordinates and trivializations.
For purposes of constructing parametrices, it is important to understand the estimates that $p(y,\eta)$ satisfies, and to identify the leading term of $p(y,\eta)$. The calculus of operator-valued symbols twisted by group actions is a suitable tool for this and we proceed here to make this statement precise, see Proposition~\ref{EdgeOpsOpValued} below. More generally, similar statements hold for edge pseudodifferential operators near the boundary, see \cite{EgorovSchulze,SchuNH,SchuWiley} for further details on the general case.

Let $\omega \in C^{\infty}(\overline{\R}_+)$ be a cut-off function, and define the cone Sobolev space
\begin{equation}\label{Kegelspace}
\K^{s,\gamma}_t(\Z^\wedge;E) = \omega x^{\gamma}H^s_b(\Z^\wedge;E) + (1-\omega)x^{-t}H_{\cone}^s(\Z^\wedge;E)
\end{equation}
for $s,t,\gamma \in \R$. The space $H^s_\cone(\Z^\wedge;E)$ is near $x=\infty$ identical to Melrose's scattering Sobolev space $\scH^s(\Z\times\lbra0,1\rpar)$ near $x'=0$ resulting from compactifying $\open\Z^\wedge=\Z\times(0,\infty)$ at $\infty$ to $\Z\times\lbra 0,\infty\rpar$ via $x'=1/x$, cf. \cite{Mel95} and \cite{SchuNH}. We will abbreviate $\K^{s,\gamma}_t(\Z^\wedge;E)$ to $\K^{s,\gamma}_t$. Note that
\begin{equation*}
\K^{0,-\gamma}_{\gamma-\frac{1}{2}\dim\Z^\wedge} = x^{-\gamma}L^2_b(\Z^\wedge;E).
\end{equation*}
Moreover, $\K^{s_1,\gamma_1}_{t_1} \hookrightarrow \K^{s_2,\gamma_2}_{t_2}$ for all $s_1 \geq s_2$, $t_1 \geq t_2$, and $\gamma_1 \geq \gamma_2$. This embedding is compact if $s_1 > s_2$, $t_1 > t_2$, and $\gamma_1 > \gamma_2$.

We equip $\K^{s,\gamma}_t$ with the strongly continuous group action
\begin{equation}\label{dilationkappa}
\kappa_{\varrho}u(x,z) = \varrho^{1/2}u(\varrho x,z), \quad \varrho > 0,
\end{equation}
for all $s,t,\gamma \in \R$.

\begin{proposition}\label{EdgeOpsOpValued}
Let
\begin{equation*}
p(y,\eta) = \sum\limits_{j+|\alpha| \leq m}\varphi_{j,\alpha}(y,x)(x\eta)^{\alpha}(xD_x)^{j} : C_c^{\infty}(\open\Z^{\wedge};E) \to C^{\infty}(\open\Z^{\wedge};F)
\end{equation*}
with $\varphi_{j,\alpha} \in C^{\infty}(\R^q\times{\overline \R}_+,\Diff^{m-j-|\alpha|}(\Z;E,F))$ such that $\varphi_{j,\alpha}(y,1/x)$ is smooth up to $x = 0$ and $\varphi_{j,\alpha}(y,x) \equiv 0$ for $|y|$ large. Then
\begin{equation*}
p(y,\eta) \in S^0(\R^q\times\R^q;\K^{s,\gamma}_t,\K^{s-m,\gamma}_{t-m}).
\end{equation*}
If the coefficients $\varphi_{j,\alpha}(y,x)$ are independent of $x$, then $p(y,\eta)$ is twisted homogeneous of degree zero. For every $N \in \NN_0$ a Taylor expansion of the coefficients at $x = 0$ induces a finite expansion
\begin{equation*}
p(y,\eta) = \sum\limits_{j=0}^{N-1}x^jp_j(y,\eta) + x^N\tilde{p}(y,\eta)
\end{equation*}
with
\begin{equation*}
x^jp_j(y,\eta) \in S_{\cl}^{-j}(\R^q\times\R^q;\K^{s,\gamma}_t,\K^{s-m,\gamma+j}_{t-m-j})
\end{equation*}
twisted homogeneous of degree $-j$, and
\begin{equation*}
x^N\tilde{p}(y,\eta) \in S^{-N}(\R^q\times\R^q;\K^{s,\gamma}_t,\K^{s-m,\gamma+N}_{t-m-N}).
\end{equation*}
In this sense, we have $p(y,\eta) \sim \sum\limits_{j=0}^{\infty}x^jp_j(y,\eta)$.
\end{proposition}

Now consider again an edge differential operator $P$ of order $m$ near the boundary in coordinates and trivializations as given by
\begin{equation*}
P = \op(p) : C_c^{\infty}(\R^q,C_c^{\infty}(\open\Z^{\wedge};E)) \to C^{\infty}(\R^q,C^{\infty}(\open\Z^{\wedge};F))
\end{equation*}
with edge symbol $p(y,\eta)$ as in Proposition~\ref{EdgeOpsOpValued}. The normal family of $P$ is
$P_{\wedge}(y,\eta) = p_0(y,\eta)$ with $p_0(y,\eta)$ as in Proposition~\ref{EdgeOpsOpValued}, and we have
\begin{equation*}
p(y,\eta) - P_{\wedge}(y,\eta) \in S^{-1}(\R^q\times\R^q;\K^{s,\gamma}_t,\K^{s-m,\gamma+1}_{t-m-1}).
\end{equation*}
In this sense, $P_{\wedge}(y,\eta)$ is the principal part of the complete edge symbol $p(y,\eta)$ of $P$.

Correspondingly, near the boundary of $\M$ a wedge differential operator of order $m$ is of the form $A = x^{-m}P$ with $P \in \Diff_{e}^m(\R^q\times\Z^{\wedge};E,F)$ as above. We have $A = \op(x^{-m}p)$ with the operator-valued symbol
\begin{equation*}
a(y,\eta) = x^{-m}p(y,\eta) \in S^{m}(\R^q\times\R^q;\K^{s,\gamma}_t,\K^{s-m,\gamma-m}_t),
\end{equation*}
and we have
\begin{equation*}
a(y,\eta) - A_{\wedge}(y,\eta) \in S^{m-1}(\R^q\times\R^q;\K^{s,\gamma}_t,\K^{s-m,\gamma-m+1}_{t-1});
\end{equation*}
$a(y,\eta)$ is the complete edge (or boundary) symbol of $A$ in the chosen coordinates and trivializations near the boundary, and the normal family is given by $A_{\wedge}(y,\eta) = x^{-m}P_{\wedge}(y,\eta)$. $A_{\wedge}(y,\eta)$ is twisted homogeneous of degree $m$.

\begin{remark}\label{WandeSpaces}
For $m \in \NN_0$ let
\begin{equation*}
H^m_e(\M;E) = \{u;\; Pu \in L^2_b(\M;E) \text{ for all } P \in \Diff^m_e(\M;E)\},
\end{equation*}
and define the spaces $H^s_e(\M;E)$ for $s \in \R$ by duality and interpolation, see \cite{Mazz91} (in Mazzeo's paper the definition of these spaces is based on $L^2(\M;E)$, but basing these spaces on $L^2_b(\M;E)$ is more convenient for our purposes).

The following relationship between these $e$-Sobolev spaces and the abstract edge Sobolev spaces of Definition~\ref{WSpace} near the boundary of the manifold $\M$ is useful:

Let ${\mathcal U} \subset \M$ be any open subset near the boundary that is diffeomorphic to $\Omega \times \Z \times \lbra0,\delta\rpar$ with $\delta > 0$, and assume that all bundles have adapted edge trivializations in ${\mathcal U}$ covering the diffeomorphism as described in detail in \cite[Sections~2 and 3]{GiKrMe10}. Let $u \in x^{\gamma}H_e^s(\M;E)$ have compact support in ${\mathcal U}$. Then the push-forward of $u$ to $\Omega\times\Z\times\lbra0,\delta\rpar \subset \R^q\times\Z^{\wedge}$ belongs to the space
\begin{equation}\label{HselocalWs}
{\mathcal W}^{\gamma+1/2}(\R^q,\K^{s,\gamma}_{s-\gamma-\frac{1}{2}\dim\Z^{\wedge}}(\Z^{\wedge};E)).
\end{equation}
Conversely, if $u$ belongs to the space \eqref{HselocalWs} and has compact support in $\Omega\times\Z\times\lbra0,\delta\rpar$, then the pull-back of $u$ with respect to the diffeomorphism and the adapted trivializations belongs to $x^{\gamma}H_e^s(\M;E)$. In short, the Sobolev space \eqref{HselocalWs} is a local model for $x^{\gamma}H^s_e(\M;E)$ near the boundary. The $\frac{1}{2}$-shift from $\gamma$ to $\gamma+\frac{1}{2}$ in the exponent of the ${\mathcal W}$-space in \eqref{HselocalWs} is due to our choice of normalization of the dilation group action in \eqref{dilationkappa}.
\end{remark}

\section{Green operators and ellipticity in the edge calculus}\label{GreenOperators}

In this appendix we review the definition and some of the properties of Green operators and the concept of ellipticity and the Fredholm property in Schulze's edge calculus \cite{EgorovSchulze,SchuNH,SchuWiley}. We will utilize some minor modifications from the presentation in Schulze's books and papers here that appear in our opinion to be better adapted for our purposes.

Green operators are matrix operators of the form
\begin{equation}\label{Greenmatrix}
G = \begin{bmatrix} G_{11} & G_{12} \\ G_{21} & G_{22} \end{bmatrix} : \begin{array}{c} C_c^{\infty}(\open\M;E) \\ \oplus \\ C^{\infty}(\Y;J_-) \end{array} \to
\begin{array}{c} C^{\infty}(\open\M;F) \\ \oplus \\ C^{\infty}(\Y;J_+) \end{array},
\end{equation}
where either of the bundles $E,F$ on $\M$ and $J_{\pm}$ on $\Y$ may be absent, in which case the matrix degenerates accordingly.

The lower-right corner operator $G_{22} : C^{\infty}(\Y;J_-) \to C^{\infty}(\Y;J_+)$ in \eqref{Greenmatrix} is merely a classical pseudodifferential operator of some order $\mu \in \R$. The remaining matrix entries map distributions on $\open\M$ to distributions on $\open\M$ ($G_{11}$), distributions on $\open\M$ to distributions on $\Y$ ($G_{21}$), or distributions on $\Y$ to distributions on $\open\M$ ($G_{12}$), and all these operators are smoothing in the sense that
\begin{equation*}
\begin{bmatrix} G_{11} & G_{12} \\ G_{21} & 0 \end{bmatrix} :
\begin{array}{c} C^{-\infty}_c(\open\M;E) \\ \oplus \\ C^{-\infty}(\Y;J_-) \end{array} \to
\begin{array}{c} C^{\infty}(\open\M;F) \\ \oplus \\ C^{\infty}(\Y;J_+) \end{array},
\end{equation*}
so they are represented by integral operators with $C^{\infty}$ kernels. The behavior of these kernels near $\N$ is crucial and is encoded in the mapping properties of $G$ that are part of the precise definition given below. This description refers to the choice of two weights $\gamma_1,\gamma_2 \in \R$, where $\gamma_1$ is associated with the input variables/function spaces, and $\gamma_2$ with the output variables/function spaces, and a suitable notion of order.

\begin{definition}
An operator $G$ of the form \eqref{Greenmatrix} is a residual Green operator (or a Green operator of order $-\infty$) with respect to the weights $(\gamma_1,\gamma_2) \in \R^2$ if there exists an $\varepsilon > 0$ such that
\begin{equation*}
G : \begin{array}{c} x^{\gamma_1}H_b^s(\M;E) \\ \oplus \\ H^s(\Y;J_-) \end{array} \to \begin{array}{c} x^{\gamma_2+\varepsilon}H_b^t(\M;F) \\ \oplus \\ H^t(\Y;J_+) \end{array}
\end{equation*}
is continuous for all $s,t \in \R$, or in short if
\begin{equation*}
G : \begin{array}{c} x^{\gamma_1}H_b^{-\infty}(\M;E) \\ \oplus \\ C^{-\infty}(\Y;J_-) \end{array} \to \begin{array}{c} x^{\gamma_2+\varepsilon}H_b^{\infty}(\M;F) \\ \oplus \\ C^{\infty}(\Y;J_+) \end{array}.
\end{equation*}
\end{definition}
 
Every residual Green operator is a compact operator in the spaces
\begin{equation*}
G : \begin{array}{c} x^{\gamma_1}H_e^s(\M;E) \\ \oplus \\ H^s(\Y;J_-) \end{array} \to \begin{array}{c} x^{\gamma_2}H_e^t(\M;F) \\ \oplus \\ H^t(\Y;J_+) \end{array}
\end{equation*}
for all $s,t \in \R$.

To describe the structure of Green operators of general orders $\mu \in \R$ the following notational convention will be useful: Recall from \eqref{Ring} that $\Ring$ is the ring of $C^{\infty}$ functions on $\M$ that on the boundary are constant on the fibers of the fibration $\wp$, i.e., $\varphi \in \Ring$ if and only if $\varphi|_{\N} = \wp^*\varphi_0$ with $\varphi_0 \in C^{\infty}(\Y)$. With every $\varphi \in \Ring$ we associate an operator
\begin{gather*}
\varphi = \begin{bmatrix} \varphi & 0 \\ 0 & \varphi_0 \end{bmatrix} :
\begin{array}{c} C^{-\infty}(\open\M;E) \\ \oplus \\ C^{-\infty}(\Y;J_-) \end{array} \to
\begin{array}{c} C^{-\infty}(\open\M;E) \\ \oplus \\ C^{-\infty}(\Y;J_-) \end{array}, \\
\varphi\begin{pmatrix} u \\ v \end{pmatrix} = \begin{bmatrix} \varphi & 0 \\ 0 & \varphi_0 \end{bmatrix}\begin{pmatrix} u \\ v \end{pmatrix} = \begin{pmatrix} \varphi u \\ \varphi_0 v \end{pmatrix},
\end{gather*}
which as indicated in the definition we will simply denote by $\varphi$ in the sequel.

\begin{definition}\label{GreenArbitraryOrder}
An operator $G$ of the form \eqref{Greenmatrix} is a Green operator of order $\mu \in \R$ with respect to the weights $(\gamma_1,\gamma_2) \in \R^2$ if there exists an $\varepsilon > 0$ such that the following holds:
\begin{enumerate}[(a)]
\item For all $\varphi,\psi \in \Ring$ such that $\supp(\varphi) \cap \supp(\psi) = \emptyset$, the operator
$\varphi G \psi$ is a residual Green operator with respect to the weights $(\gamma_1,\gamma_2)$.
\item For all $\varphi \in \dot{C}^{\infty}(\M)$ the operators $\varphi G$ and $G \varphi$ are residual Green operators with respect to the weights $(\gamma_1,\gamma_2)$.
\item Let ${\mathcal U}$ be any neighborhood of the boundary diffeomorphic to $\Omega\times\Z\times\lbra0,\delta\rpar$ for some $\delta > 0$ with $\Omega \subset \R^q$ open, and suppose that all bundles are trivial over ${\mathcal U}$. Let $\varphi,\psi \in \Ring$ be compactly supported in ${\mathcal U}$. Then, in coordinates and trivializations, the operator $\varphi G \psi$ acts locally as
\begin{equation*}
\varphi G \psi : \begin{array}{c} C_c^{\infty}(\R^q,C_c^{\infty}(\open\Z^{\wedge};E)) \\ \oplus \\ C_c^{\infty}(\R^q,\C^{N_-}) \end{array} \to \begin{array}{c} C_c^{\infty}(\R^q,C_c^{\infty}(\open\Z^{\wedge};F)) \\ \oplus \\ C_c^{\infty}(\R^q,\C^{N_+}) \end{array},
\end{equation*}
and as such is required to be of the form $\varphi G \psi = \op(g)$ with an operator-valued symbol
\begin{equation}\label{Greensymbol}
\begin{aligned}
g(y,\eta) &\in S^{\mu}_{\cl}\biggl(\R^q\times\R^q;
\begin{array}{c} \K^{-\infty,\gamma_1}_{-\infty} \\ \oplus \\ \C^{N_-} \end{array},
\begin{array}{c} \K^{\infty,\gamma_2+\varepsilon}_{\infty} \\ \oplus \\ \C^{N_+} \end{array}
\biggr) \\ &=
\bigcap_{s_i,t_i \in \R,\:i=1,2} S^{\mu}_{\cl}\biggl(\R^q\times\R^q;
\begin{array}{c} \K^{s_1,\gamma_1}_{t_1} \\ \oplus \\ \C^{N_-} \end{array},
\begin{array}{c} \K^{s_2,\gamma_2+\varepsilon}_{t_2} \\ \oplus \\ \C^{N_+} \end{array}\biggr).
\end{aligned}
\end{equation}
We use here the cone Sobolev spaces \eqref{Kegelspace} with the normalized dilation group action \eqref{dilationkappa} on them, and the finite-dimensional spaces $\C^{N_{\pm}}$ stemming from trivializing the bundles $J_{\pm}$ are equipped with the trivial group action $\tilde{\kappa}_{\varrho} \equiv \Id_{\C^{N_{\pm}}}$ for all $\varrho > 0$.
\end{enumerate}
\end{definition}

Because the local operator-valued symbol $g(y,\eta)$ in \eqref{Greensymbol} is a classical symbol, there exists an expansion
\begin{equation*}
g(y,\eta) \sim \sum\limits_{j=0}^{\infty}\chi(\eta)g_{(\mu-j)}(y,\eta)
\end{equation*}
like in \eqref{HomCompExp}. The principal symbol of $g(y,\eta)$ is the leading term $g_{(\mu)}(y,\eta)$ in this expansion. Patching these principal symbols shows that there exists an invariantly defined normal family
\begin{equation*}
T^*\Y\minus 0 \ni \pmb\eta \to G_{\wedge}(\pmb\eta) : \begin{array}{c} C_c^{\infty}(\Z^{\wedge}_y;E_{\Z^{\wedge}_y}) \\ \oplus \\ \pi_{\Y}^*J_{-,y} \end{array} \to \begin{array}{c} C^{\infty}(\Z^{\wedge}_y;F_{\Z^{\wedge}_y}) \\ \oplus \\ \pi_{\Y}^*J_{+,y} \end{array}, \quad y = \pi_{\Y}\pmb\eta,
\end{equation*}
associated with $G$. The normal family is twisted homogeneous of degree $\mu$ in the sense that
\begin{equation*}
G_{\wedge}(\varrho\, \pmb\eta) = \varrho^{\mu} \begin{bmatrix} \kappa_{\varrho} & 0 \\ 0 & \Id \end{bmatrix} G_{\wedge}(\pmb\eta) \begin{bmatrix} \kappa^{-1}_{\varrho} & 0 \\ 0 & \Id \end{bmatrix}
\end{equation*}
for $\varrho > 0$.

Now let $G : C_c^{\infty}(\open\M;E) \to C^{\infty}(\open\M;F)$ be a Green operator of order $\mu$ associated with the weights $(\gamma_1,\gamma_2) \in \R^2$, and let $P$ be any edge differential operator acting in sections of bundles such that the compositions $P\circ G$ or $G \circ P$ make sense. Definition~\ref{GreenArbitraryOrder}, Proposition~\ref{EdgeOpsOpValued}, and Proposition~\ref{composition} then imply that these compositions are again Green operators of the same order $\mu$ associated with the same weights, and
\begin{equation*}
\bigl(P\circ G\bigr)_{\wedge}(\pmb\eta) = P_{\wedge}(\pmb\eta)\circ G_{\wedge}(\pmb\eta) \text{ or }
\bigl(G\circ P\bigr)_{\wedge}(\pmb\eta) = G_{\wedge}(\pmb\eta)\circ P_{\wedge}(\pmb\eta),
\end{equation*} 
respectively. For the same reason, the analogous statement holds for compositions of edge differential operators and Green operators $G : C^{\infty}(\Y;J_-) \to C^{\infty}(\open\M;F)$ or $G : C_c^{\infty}(\open\M;E) \to C^{\infty}(\Y;J_+)$. Similarly, the composition $x^{-m}P \circ G$ of a wedge differential operator $x^{-m}P$ of order $m$ with $G$ is a Green operator of order $\mu+m$ associated with the weights $(\gamma_1,\gamma_2-m)$, and $G \circ x^{-m}P$ is a Green operator of order $\mu+m$ associated with the weights $(\gamma_1+m,\gamma_2)$.

More generally, analogous results hold for compositions of Green operators and edge pseudodifferential operators as follows from the calculus developed in \cite{SchuNH}.

\bigskip

Let
\begin{equation*}
G = \begin{bmatrix} G_{11} & G_{12} \\ G_{21} & G_{22} \end{bmatrix} : \begin{array}{c} C_c^{\infty}(\open\M;E) \\ \oplus \\ C^{\infty}(\Y;J_-) \end{array} \to
\begin{array}{c} C^{\infty}(\open\M;F) \\ \oplus \\ C^{\infty}(\Y;J_+) \end{array}
\end{equation*}
be a Green operator with respect to the weights $(\gamma_1,\gamma_2) \in \R^2$, and suppose that the order $\mu$ of $G$ satisfies $\mu \leq \gamma_1 - \gamma_2$. Theorem~\ref{WSpaceContinuity} and Remark~\ref{WandeSpaces} then imply that $G$ extends to a continuous operator
\begin{equation}\label{GineSpaces}
G : \begin{array}{c} x^{\gamma_1}H^{s}_e(\M;E) \\ \oplus \\ H^{\gamma_1+1/2}(\Y;J_-) \end{array} \to \begin{array}{c} x^{\gamma_2}H^{t}_e(\M;F) \\ \oplus \\ H^{\gamma_2+1/2}(\Y;J_+) \end{array}
\end{equation}
for all $s,t \in \R$. If $\mu < \gamma_1-\gamma_2$, then there exists $\varepsilon > 0$ such that
\begin{equation*}
G : \begin{array}{c} x^{\gamma_1}H^{s}_e(\M;E) \\ \oplus \\ H^{\gamma_1+1/2}(\Y;J_-) \end{array} \to \begin{array}{c} x^{\gamma_2+\varepsilon}H^{t}_e(\M;F) \\ \oplus \\ H^{\gamma_2+1/2+\varepsilon}(\Y;J_+) \end{array}.
\end{equation*}
In particular, $G$ acting in the spaces \eqref{GineSpaces} is then compact. This shows that $G_{\wedge}(\pmb\eta)$ determines the action of $G$ in \eqref{GineSpaces} modulo compact operators.

\subsection*{Ellipticity and the Fredholm property in the edge calculus}

Let $A \in x^{-m}\Diff^m_e(\M;E,F)$ be $w$-elliptic, i.e.,
\begin{equation}\label{Awelliptic-app}
\wsym(A) \text{ is invertible everywhere on } \wT^*\M \minus 0,
\end{equation}
and assume that
\begin{equation}\label{speceAssumption-app}
\spec_{e}(A) \cap \bigl(\Y\times\{\sigma \in \C;\; \Im(\sigma) = -m/2\}\bigr) = \emptyset.
\end{equation}
Let
\begin{equation*}
\begin{bmatrix} G_{11} & G_{12} \\ G_{21} & G_{22} \end{bmatrix} : \begin{array}{c} x^{m/2}H^{s}_e(\M;E) \\ \oplus \\ H^{m/2+1/2}(\Y;J_-) \end{array} \to \begin{array}{c} x^{-m/2}H^{t}_e(\M;F) \\ \oplus \\ H^{-m/2+1/2}(\Y;J_+) \end{array}
\end{equation*}
be a Green operator of order $m$ with respect to the weights $(m/2,-m/2)$. We assume that there exist $s_0,t_0 \in \R$ such that
\begin{equation}\label{normalAssumption-app}
\begin{bmatrix}
A_{\wedge}(\pmb\eta) + G_{11,\wedge}(\pmb\eta) & G_{12,\wedge}(\pmb\eta) \\ G_{21,\wedge}(\pmb\eta) & G_{22,\wedge}(\pmb\eta) \end{bmatrix} :
\begin{array}{c} \K^{s_0,m/2}_{t_0}(\Z_y^\wedge;E_{\Z_y^\wedge}) \\ \oplus \\ \pi_{\Y}^*J_{-,y} \end{array}
\to
\begin{array}{c} \K^{s_0-m,-m/2}_{t_0}(\Z_y^\wedge;F_{\Z_y^\wedge}) \\ \oplus \\ \pi_{\Y}^*J_{+,y} \end{array}
\end{equation}
is invertible for all $\pmb\eta \in T^*\Y\minus 0$, where $y = \pi_{\Y}\pmb\eta$.

The following theorem is a consequence of the edge pseudodifferential calculus developed by Schulze, see \cite[Section 3.3.5]{SchuNH}, \cite[Section 9.3.4]{EgorovSchulze}, or \cite[Section 3.5]{SchuWiley}:

\begin{theorem}\label{EdgeMatrixFredholm}
Under the ellipticity assumptions \eqref{Awelliptic-app}, \eqref{speceAssumption-app}, and \eqref{normalAssumption-app}, the operator
\begin{equation}\label{matrix-app}
\begin{bmatrix} A + G_{11} & G_{12} \\ G_{21} & G_{22} \end{bmatrix} : \begin{array}{c} x^{m/2}H^{s}_e(\M;E) \\ \oplus \\ H^{m/2+1/2}(\Y;J_-) \end{array} \to \begin{array}{c} x^{-m/2}H^{s-m}_e(\M;F) \\ \oplus \\ H^{-m/2+1/2}(\Y;J_+) \end{array}
\end{equation}
is a Fredholm operator for every $s \in \R$.
\end{theorem}

Theorem~\ref{EdgeMatrixFredholm} is proved in the edge calculus by constructing a parametrix that inverts the operator \eqref{matrix-app} modulo residual Green operators.



\begin{thebibliography}{99}

\bibitem{AlbinLeichtMazzPiazza12}
P.~Albin, {\'E}.~Leichtnam, R.~Mazzeo, and P.~Piazza, \emph{The signature package on Witt spaces}, Ann.~Sci.~{\'E}c.~Norm.~Sup{\'e}r. (4) \textbf{45} (2012), 241--310.

\bibitem{AlbinLeichtMazzPiazza13}
\bysame, \emph{Hodge theory on Cheeger spaces}, arXiv 1307.5473. 

\bibitem{AmLaNi07}
B.~Ammann, R.~Lauter, and V.~Nistor, \emph{Pseudodifferential operators on manifolds with a Lie structure at infinity}, Ann. of Math. (2) \textbf{165} (2007), no. 3, 717--747.

\bibitem{AtiyahBott}
M.~F.~Atiyah and R.~Bott, \emph{The index theorem for manifolds with boundary}, \emph{Differential Analysis} (Bombay Colloquium), pp.~175--186, Oxford Univ. Press, London, 1964.


\bibitem{APS}
M.~F.~Atiyah, V.~K.~Patodi, and I.~M.~Singer, \emph{Spectral asymmetry and Riemannian geometry: I}, Math.~Proc.~Cambridge Philos.~Soc. \textbf{77} (1975), 43--69.


\bibitem{BoossWoj1993}
B.~Boo{\ss}--Bavnbek and K.~P.~Wojciechowski, \emph{Elliptic boundary problems for Dirac operators}, Mathematics: Theory \& Applications, Birkh{\"a}user, Boston, 1993.

\bibitem{ChazarainPiriou}
J.~Chazarain and A. Piriou, \emph{Introduction to the Theory of Linear Partial Differential Equations}, Studies in Mathematics and its Applications, vol.~14, North-Holland Publishing Co., Amsterdam-New York, 1982.

\bibitem{EgorovSchulze}
Yu.~V.~Egorov and B.--W.~Schulze, \emph{Pseudo-Differential Operators, Singularities, Applications}, Operator Theory: Advances and Applications, vol.~\textbf{93}, Birkh{\"a}user, Basel, 1997.

\bibitem{EpRBMMe91}
C.~L.~Epstein, R.~B.~Melrose, and G.~A.~Mendoza, \emph{Resolvent of the Laplacian on strictly pseudoconvex domains}, Acta Math. \textbf{167} (1991), no. 1--2, 1--106.

\bibitem{GiMe01}
J.~B.~Gil and G.~A.~Mendoza, \emph{Adjoints of elliptic cone operators}, Amer. J. Math. \textbf{125} (2003), 357--408.

\bibitem{GiKrMe07}
J.~B.~Gil, T.~Krainer, and G.~A.~Mendoza, \emph{Geometry and spectra of closed extensions of elliptic cone operators}, Canad. J. Math. \textbf{59} (2007), no. 4, 742--794.

\bibitem{GiKrMe10}
\bysame, {On the closure of elliptic wedge operators}, to appear in  J.~Geom.~Anal, 28 pages, published online 11 July 2012.

\bibitem{GrubbBook}
G.~Grubb, \emph{Functional Calculus of Pseudodifferential Boundary Problems}, Second Edition, Progress in Mathematics, vol.~65, Birkh{\"a}user Boston, 1996.

\bibitem{GrubbSeeley}
G.~Grubb and R.~T.~Seeley, \emph{Weakly parametric pseudodifferential operators and Atiyah-Patodi-Singer boundary problems}, Invent.~Math. \textbf{121} (1995), 481--529.

\bibitem{Hor67}
L.~H\"ormander, \emph{Pseudo-differential operators and non-elliptic boundary problems} Ann. of Math. \textbf{83} 1966 129--209.

\bibitem{KapSchuSei08}
D.~Kapanadze, B.--W.~Schulze, and J.~Seiler, \emph{Operators with singular trace conditions on a manifold with edges}, Integr.~Equ.~Oper.~Theory \textbf{61} (2008), 241--279.

\bibitem{Ko91}
Yu.~A.~Kordyukov, \emph{$L^p$-theory of elliptic differential operators on manifolds of bounded geometry}, Acta Appl. Math. \textbf{23} (1991), 223--260.

\bibitem{KrMe12a}
T.~Krainer and G.~Mendoza, \emph{The kernel bundle of a holomorphic Fredholm family}, arXiv:1301.5811. To appear in Comm.~Partial Differential Equations.

\bibitem{KrMe12b}
\bysame, \emph{Elliptic systems of variable order}, arXiv:1301.5820. To appear in Rev.~Mat.~Iberoam.

\bibitem{Le97} 
M.~Lesch, \emph{Operators of {F}uchs Type, Conical Singularities, and Asymptotic Methods}, Teubner-Texte zur Math. vol 136, B.G. Teubner, Stuttgart, Leipzig, 1997.

\bibitem{Mazz91}
R.~Mazzeo, \emph{Elliptic theory of differential edge operators I}, Comm. Partial Differential Equations \textbf{16} (1991), 1615--1664.

\bibitem{MaMe87}
R. Mazzeo and R.~B.~Melrose, \emph{Meromorphic extension of the resolvent on complete spaces with asymptotically constant negative curvature}, 
J. Funct. Anal. \textbf{75} (1987), no. 2, 260--310. 

\bibitem{MaMe98}
\bysame, \emph{Pseudodifferential operators on manifolds with fibred boundaries}, in ``Mikio Sato: a great Japanese mathematician of the twentieth century,'' Asian J. Math. \textbf{2} (1998), no. 4, 833--866.

\bibitem{MazzVert12}
R.~Mazzeo and B.~Vertman, \emph{Analytic torsion on manifolds with edges}, Adv.~Math. \textbf{231} (2012), 1000--1040.

\bibitem{MazzVertCommunication}
\bysame, \emph{Elliptic theory of differential edge operators, II: Boundary value problems}, arXiv:1307.2266.

\bibitem{Mel81}
R.~B.~Melrose, \emph{Transformation of boundary value problems}, Acta Math. 147 (1981), 149--236.

\bibitem{Mel93}
\bysame, The Atiyah-Patodi-Singer index theorem. Research Notes in Mathematics, 4. A K Peters, Ltd., Wellesley, MA, 1993.

\bibitem{Mel95}
\bysame, \emph{Spectral and scattering theory for the Laplacian on asymptotically Euclidian spaces} Spectral and scattering theory (Sanda, 1992), 85--130, Lecture Notes in Pure and Appl. Math., 161, Dekker, New York, 1994.

\bibitem{Ro12}
F.~Rochon, \emph{Pseudodifferential operators on manifolds with foliated
boundaries},  J. Funct. Anal. \textbf{262} (2012) 1309--1362.

\bibitem{Shub92}
M.~A.~Shubin, \emph{Spectral theory of elliptic operators on noncompact manifolds}, in M\'ethodes semi-classiques, Vol. 1 (Nantes, 1991). Ast\'erisque \textbf{207} (1992), 5, 35--108.

\bibitem{SchuNH}
B.-W.~Schulze, \emph{{P}seudo-differential {O}perators on {M}anifolds with {S}ingularities}, North Holland, Amsterdam, 1991.

\bibitem{SchuWiley}
\bysame, \emph{Boundary Value Problems and Singular Pseudo-Differential Operators}, Pure and Applied Mathematics (New York). John Wiley, Chichester, 1998.

\bibitem{Schu2001}
\bysame, \emph{An algebra of boundary value problems not requiring Shapiro-Lopatinskij conditions}, J.~Funct.~Anal. \textbf{179} (2001), 374--408.

\bibitem{SchuSei2006}
B.--W.~Schulze and J.~Seiler, \emph{Edge operators with conditions of Toeplitz type}, J.~Inst.~Math.~ Jussieu \textbf{5} (2006), 101--123.

\bibitem{Seeley1969}
R.~T.~Seeley, \emph{Topics in pseudo-differential operators}. Pseudo-Diff. Operators (C.I.M.E., Stresa, 1968), pp. 167--305, Edizioni Cremonese, Rome, 1969.

\bibitem{SeilerDiss}
J.~Seiler, \emph{Pseudodifferential Calculus on Manifolds with Non-Compact Edges}, PhD Thesis, University of Potsdam, Germany, 1997.

\bibitem{SeilerCont}
\bysame, \emph{Continuity of edge and corner pseudodifferential operators}, Math.~Nachr. \textbf{205} (1999), 163--182.

\bibitem{Seiler2010}
\bysame, \emph{Natural domains for edge-degenerate differential operators}, arXiv:1006.0039.


\bibitem{Waterstraat}
N.~Waterstraat, \emph{The index bundle for Fredholm morphisms}, Rend.~Semin.~Mat.~Univ.~Politec. Torino \textbf{69} (2011), 299--315.


\end{thebibliography}
\end{document}